\documentclass{amsart}
\usepackage{bbm}

\usepackage{amssymb}
\usepackage{enumerate}
\usepackage{mathrsfs}
\usepackage{hyperref}
\usepackage[justification=centering, skip=20pt]{caption}
\usepackage{tikz}
\usepackage[all,cmtip]{xy}
\usepackage{graphicx}

\usepackage{setspace}

\hypersetup{
    colorlinks,
    citecolor=black,
    filecolor=black,
    linkcolor=black,
    urlcolor=black
    }

\usepackage{geometry}
\newtheorem{thm}{Theorem}[section]
\newtheorem{cor}[thm]{Corollary}
\newtheorem{prop}[thm]{Proposition}
\newtheorem{lem}[thm]{Lemma}

\newtheorem{defn}[thm]{Definition}

\newtheorem{exmp}[thm]{Example}

\newtheorem{rem}[thm]{Remark}

\let\emptyset\varnothing

\numberwithin{equation}{section}

\DeclareMathAlphabet{\mathbbmsl}{U}{bbm}{m}{sl}

\newcommand\numberthis{\addtocounter{equation}{1}\tag{\theequation}}

\usepackage{etoolbox}

\makeatletter
\patchcmd{\@settitle}{\uppercasenonmath\@title}{}{}{}
\patchcmd{\@setauthors}{\MakeUppercase}{}{}{}
\patchcmd{\section}{\scshape}{}{}{}
\makeatother

\title[Strichartz estimates for the Schr\"odinger flow on compact Lie groups]{Strichartz estimates for the Schr\"odinger flow on compact Lie groups}

\author[Y. Zhang]{Yunfeng Zhang}
\address{Department of Mathematics, University of Connecticut, Storrs, CT 06269}%
\email{yunfeng.zhang@uconn.edu}

\begin{document}

\onehalfspacing

\begin{abstract}
We establish scale-invariant Strichartz estimates for the Schr\"odinger flow on any compact Lie group equipped with canonical rational metrics. In particular, full Strichartz estimates without loss for some non-rectangular tori are given. The highlights of this paper include estimates for some Weyl type sums defined on rational lattices, different decompositions of the Schr\"odinger kernel that accommodate different positions of the variable inside the maximal torus relative to the cell walls, and an application of the BGG-Demazure operators or Harish-Chandra's integral formula to the estimate of the difference between characters. 
\end{abstract}


\maketitle

\tableofcontents

\section{Introduction}
We start with a complete Riemannian manifold $(M,g)$ of dimension $d$,  associated to which are the Laplace-Beltrami operator $\Delta_g$ and the volume form measure $\mu_g$. Then it is well known that $\Delta_g$ is essentially self-adjoint on $L^2(M):=L^2(M,d\mu_g)$; see \cite{Str83} for a proof. This gives the functional calculus of $\Delta_g$, and in particular gives the one-parameter unitary operators $e^{it\Delta_g}$  which provides the solution to the linear Schr\"odinger equation on $(M,g)$. We refer to $e^{it\Delta_g}$ as the  \textit{Schr\"odinger flow}. The functional calculus of $\Delta_g$ also gives the definition of the Bessel potentials, and thus the definition of the Sobolev space
\begin{equation*}
H^s(M):=\{u\in L^2(M) \mid \|u\|_{H^s(M)}:=\|(I-\Delta)^{s/2}u\|_{L^2(M)}<\infty\}.
\end{equation*}
We are interested in obtaining estimates of the form 
\begin{equation}\label{Strichartz}
\|e^{it\Delta_g}f\|_{L^pL^r(I\times M)}\leq C \|f\|_{H^s(M)}
\end{equation}
where $I\subset\mathbb{R}$ is a fixed time interval, and 
$L^pL^q(I\times M)$ is the space of $L^p$ functions on $I$ with values in $L^q(M)$. Such estimates are often called Strichartz estimates (for the Schr\"odinger flow), in honor of Robert Strichartz \cite{Str77} who first derived such estimates for the wave equation on Euclidean spaces.

The significance of Strichartz estimates is evident in many ways. Strichartz estimates have important applications in the field of nonlinear Schr\"odinger equations, in the sense that many perturbative results often require good control on the linear solution which is exactly provided by Strichartz estimates. Strichartz estimates can also be interpreted as Fourier restriction estimates, which play a fundamental rule in the field of classical harmonic analysis. Furthermore, the relevance of the distribution of eigenvalues and the norm of eigenfunctions of $\Delta_g$ in deriving the estimates makes Strichartz estimates also a subject in the field of spectral geometry. 

Many cases of Strichartz estimates for the Schr\"odinger flow are known in the literature.   For noncompact manifolds, first we have the sharp Strichartz estimates on the Euclidean spaces obtained in \cite{GV95, KT98}:
\begin{equation}\label{StrEuc0}
\|e^{it\Delta}f\|_{L^pL^q(\mathbb{R}\times \mathbb{R}^d)}\leq C\|f\|_{L^2(\mathbb{R}^d)}
\end{equation}
where $\frac{2}{p}+\frac{d}{q}=\frac{d}{2}$, $p,q\geq 2$, $(p,q,d)\neq (2,\infty,2)$. Such pairs $(p,q)$ are called \textit{admissible}. This implies by Sobolev embedding that 
\begin{equation}\label{StrEuc}
\|e^{it\Delta}f\|_{L^pL^r(\mathbb{R}\times\mathbb{R}^d)}\leq C\|f\|_{H^s(\mathbb{R}^d)}
\end{equation}
where
\begin{equation}\label{scaling}
s=\frac{d}{2}-\frac{2}{p}-\frac{d}{r}\geq 0,
\end{equation} 
$p,q\geq 2$, $(p,r,d)\neq(2,\infty,2)$. Note that the equality in \eqref{scaling} can be derived from a standard scaling argument, and we call exponent triples $(p,r,s)$ that satisfy \eqref{scaling} as well as the corresponding Strichartz estimates \textit{scale-invariant}. 
Similar Strichartz estimates hold on many noncompact manifolds. For example, see \cite{AP09, Ban07, IS09, Pie06} for Strichartz estimates on the real hyperbolic spaces, \cite{APV11, Pie08, BD07} for Damek-Ricci spaces which include all rank-1 symmetric spaces of noncompact type,  \cite{Bou11} for asymptotically hyperbolic manifolds, \cite{HTW06} for asymptotically conic manifolds, \cite{BT08, ST02} for some perturbed Schr\"odinger equations on Euclidean spaces, and \cite{FMM15} for symmetric spaces $G/K$ where $G$ is complex.  

For compact manifolds however, Strichartz estimates such as  \eqref{StrEuc0} are expected to fail. The Sobolev exponent $s$ in \eqref{Strichartz} is expected to be positive for \eqref{Strichartz} to possibly hold. And we also expect sharp Strichartz estimates that are \textit{non-scale-invariant}, in the sense that the exponents $(p,r,s)$ in \eqref{Strichartz} satisfy
\begin{equation*}
s>\frac{d}{2}-\frac{2}{p}-\frac{d}{r}. 
\end{equation*}
For example, from the results in \cite{ST02, BGT04}, we know that on a general compact Riemannian manifold $(M,g)$ it holds that for any finite interval $I$, 
\begin{equation}\label{StrCpt}
\|e^{it\Delta_g}f\|_{L^pL^r(I\times M)}\leq C\|f\|_{H^{1/p}(M)}
\end{equation}
for all admissible pairs $(p,r)$. These estimates are non-scale-invariant, and the special case of which when $(p,r,s)=(2,\frac{2d}{d-2}, 1/2)$ can be shown to be sharp on spheres of dimension $d\geq 3$ equipped with canonical Riemannian metrics. On the other hand, scale-invariant estimates are out of reach of the local methods employed in \cite{ST02, BGT04}, and they are not well explored yet in the literature. To my best knowledge, the only known results in the literature in this direction are on Zoll manifolds, which include all compact symmetric spaces of rank 1, the standard sphere being a typical example; and on rectangular tori. We summarize the results here. Consider the scale-invariant estimates
\begin{equation}\label{StrCri}
\|e^{it\Delta_g}f\|_{L^p(I\times M)}\leq C\|f\|_{H^{\frac{d}{2}-\frac{d+2}{p}}(M)}. 
\end{equation}
In the direction of Zoll manifolds, 
\eqref{StrCri} is first proved in \cite{BGT07} for the standard three-sphere for $p=6$. Then in \cite{Her13}, \eqref{StrCri} is proved for all $p>4$ for any three-dimensional Zoll manifold, but the methods employed in that paper in fact prove \eqref{StrCri} for $p>4$ for any Zoll manifold with dimension $d\geq 3$ and for $p\geq 6$ for any Zoll surface ($d=2$). The paper crucially uses the property of Zoll manifolds that the spectrum of the Laplace-Beltrami operator is clustered around a sequence of squares, and the spectral cluster estimates (\cite{Sog88}) which are optimal on spheres. 
In the direction of tori, \eqref{StrCri} was first proved in \cite{Bou93} for $p\geq \frac{2(d+4)}{d}$ on square tori, by interpolating  
the distributional Strichartz estimate 
\begin{equation}\label{StrDis}
\lambda\cdot\mu\{(t,x)\in I\times\mathbb{T}^d \mid |e^{it\Delta_g}\varphi(N^{-2}\Delta_g)f(x)|>\lambda\}^{1/p}\leq C N^{\frac{d}{2}-\frac{d+2}{p}}\|f\|_{L^2(\mathbb{T}^d)}\leq C\|f\|_{H^{\frac{d}{2}-\frac{d+2}{p}}(\mathbb{T}^d)}. 
\end{equation}
for $\lambda>N^{d/4}$, $p>\frac{2(d+2)}{d}$, $N\geq 1$, with the trivial subcritical Strichartz estimate
\begin{equation}\label{StrTri}
\|e^{it\Delta_g}f\|_{L^2(I\times\mathbb{T}^d)} \leq C\|f\|_{L^2(\mathbb{T}^d)}. 
\end{equation}
The estimate \eqref{StrDis} is a consequence of an arithmetic version of dispersive estimates: 
\begin{equation}\label{DisAri}
\|e^{it\Delta_g}\varphi(N^{-2}\Delta_g)\|_{L^\infty(\mathbb{T}^d)}\leq C
\left(\frac{N}{\sqrt{q}(1+N\|\frac{t}{T}-\frac{a}{q}\|^{1/2})}\right)^{d}\|f\|_{L^1(\mathbb{T}^d)}
\end{equation}
where $\|\cdot\|$ stands for the distance from 0 on the standard circle with length 1,  $\|\frac{t}{T}-\frac{a}{q}\|<\frac{1}{qN}$, $a,q$ are nonnegative integers with $a<q$ and $(a,q)=1$, and $q<N$. Here $T$ is the period for the Schr\"odinger flow $e^{it\Delta_g}$. Then in \cite{Bou13}, the author improved \eqref{StrTri} into a stronger subcritical Strichartz estimate 
\begin{equation}\label{StrSub}
\|e^{it\Delta_g}f\|_{L^{\frac{2(d+1)}{d}}(I\times\mathbb{T}^d)}\leq C\|f\|_{L^2(\mathbb{T}^d)}
\end{equation}
which yields \eqref{StrCri} for $p\geq \frac{2(d+3)}{d}$. Eventually, \eqref{StrCri}
with an $\varepsilon$-loss is proved for the full range $p>\frac{2(d+2)}{d}$ in \cite{BD15}, and \eqref{StrDis} can be used to remove this $\varepsilon$-loss. 
Then authors in \cite{GOW14, KV16} extended the results to all rectangular tori. We will see in this paper that by a slight adaptation of the methods in \cite{Bou93}, we may generalize \eqref{StrDis} to all rational (not necessarily rectangular) tori $\mathbb{T}^d=\mathbb{R}^d/\Gamma$ where $\Gamma\cong\mathbb{Z}^d$ is a lattice such that there exists some $D\neq 0$ for which $\langle\lambda,\mu\rangle\in D^{-1}\mathbb{Z}$ for all $\lambda,\mu\in\Gamma$, which can also be used for the removal of the $\varepsilon$-loss of the results in \cite{BD15} to yield \eqref{StrCri} for the full range $p>\frac{2(d+2)}{d}$ on such rational tori.


The understanding of Strichartz estimates on compact manifolds is far from complete. It is not known in general how the exponents $(p,r,s)$ in the sharp Strichartz estimates are related to the geometry and topology of the underlying manifold. Also, there still are important classes of compact manifolds on which the Strichartz estimates have not been explored yet. Note that both standard tori and spheres on which Strichartz estimates are known are special cases of compact globally symmetric spaces, and since all compact globally symmetric spaces share the same behavior of geodesic dynamics as tori, from a semiclassical point of view, it's natural to conjecture that similar Strichartz estimates should hold on general compact globally symmetric spaces. An important class of such spaces is the class of compact Lie groups. The goal of this paper is to prove scale-invariant Strichartz estimates of the form \eqref{StrCri} for $M=G$ being any connected compact Lie group equipped with a canonical \textit{rational metric} in the sense that is described below, for all $p\geq\frac{2(r+4)}{r}$, $r$ being the \textit{rank} of $G$. In particular, full Strichartz estimates without loss for some non-rectangular tori will be given.

\section{Statement of the Main Theorem}

\subsection{Rational Metric}\label{RationalMetric}
Let $G$ be a connected compact Lie group and $\mathfrak{g}$ be its Lie algebra. By the classification theorem of connected compact Lie groups, see Chapter 10, Section 7.2, Theorem 4 in \cite{Pro07}), there exists an exact sequence of Lie group homomorphisms
\begin{align*}
1\rightarrow A\rightarrow \tilde{G}\cong\mathbb{T}^n\times K\rightarrow G\rightarrow1
\end{align*}
where $\mathbb{T}^n$ is the $n$-dimensional torus, $K$ is a compact simply connected semisimple Lie group, and $A$ is a finite and central subgroup of the \textit{covering group} $\tilde{G}$. As a compact simply connected semisimple Lie group, $K$ is a direct product $K_1\times K_2\times\cdots\times K_m$ of compact simply connected simple Lie groups. 

Now each $K_i$ is equipped with the canonical bi-invariant Riemannian metric $g_i$ that is induced from the negative of the Cartan-Killing form. We use $\langle\  ,\  \rangle$ to denote the Cartan-Killing form. Then we equip the torus factor $\mathbb{T}^n$ with a flat metric $g_0$ inherited from its representation as the quotient $\mathbb{R}^n/2\pi\Gamma$ and require that there exists some $D\in\mathbb{N}$ such that $\langle \lambda,\mu\rangle\in D^{-1}\mathbb{Z}$ for all $\lambda,\mu\in\Gamma$. 
Then we equip $\tilde{G}\cong \mathbb{T}^n\times K_1\times \cdots\times K_m$ the bi-invariant metric 
\begin{equation}\label{tildeg}
\tilde{g}=\otimes_{j=0}^m \beta_jg_{j},
\end{equation} 
$\beta_j>0$, $j=0,\ldots, m$. Then $\tilde{g}$ induces a bi-invariant metric $g$ on $G$. 

\begin{defn}\label{rationalmetric}
Let $g$ be the bi-invariant metric induced from $\tilde{g}$ in \eqref{tildeg} as described above. 
We call $g$ a \textit{rational metric} provided the numbers $\beta_0,\cdots,\beta_m$ are rational multiples of each other. If not, we call it an \textit{irrational metric}. 
\end{defn}

Provided the numbers $\beta_0,\cdots,\beta_m$ are rational multiples of each other, the periods of the Schr\"odinger flow $e^{it\Delta_{\tilde{g}}}$ on each factor of $\tilde{G}$ are rational multiples of each other, which implies that the Schr\"odinger flow on $\tilde{G}$ as well as on $G$ is also periodic (see Section \ref{The Schrodinger Kernel}).

\subsection{Main Theorem}

We define the \textit{rank} of $G$ to be the dimension of any of its maximal torus. This paper mainly proves the following theorem.

\begin{thm}\label{Main}
Let $G$ be a connected compact Lie group equipped with a rational metric $g$. Let $d$ be the dimension of $G$ and $r$ the rank of $G$. Let $I\subset\mathbb{R}$ be a finite time interval. Consider the scale-invariant Strichartz estimate
\begin{equation}\label{StrCptGrp} 
\|e^{it\Delta_g}f\|_{L^p(I\times G)}\leq C\|f\|_{H^{\frac{d}{2}-\frac{d+2}{p}}(G)}.
\end{equation}
Then the following statements hold true. \\
(i) \eqref{StrCptGrp} 
holds for all $p\geq 2+\frac{8}{r}$. \\
(ii) Let $G=\mathbb{T}^d$ be a flat torus equipped with a rational metric, that is, we can write $\mathbb{T}^d=\mathbb{R}^d/2\pi\Gamma$ such that there exists some $D\in\mathbb{R}$ for which $\langle\lambda,\mu\rangle\in D^{-1}\mathbb{Z}$ for all $\lambda,\mu\in\Gamma$. Then 
\eqref{StrCptGrp} holds for all $p>2+\frac{4}{d}$.  
\end{thm}

The framework for the proof of this theorem will be based on \cite{Bou93}, in which the author proves some Strichartz estimates for the case of square tori, based on the Hardy-Littlewood circle method. We also refer to \cite{Bou89} for applications of the circle method to Fourier restriction problems on tori. Note that part (ii) of the above theorem provides full expected Strichartz estimates without loss for some non-rectangular tori. We then have the following immediate corollary. 

\begin{cor}
Let $d=3,4$ and let $\mathbb{T}^d$ be the flat torus equipped with a rational metric (not necessarily rectangular). Then the nonlinear Schr\"odinger equation $i\partial u_t=-\Delta u\pm |u|^{\frac{4}{d-2}}u$ is locally wellposed for initial data in $H^{1}(\mathbb{T}^d)$. Furthermore, for $d=3$, we have $i\partial u_t=-\Delta u\pm |u|^{2}u$ is locally wellposed for initial data in $H^{\frac{1}{2}}(\mathbb{T}^d)$.
\end{cor}

We refer to \cite{HTT11} and \cite{KV16} for the definition of local well-posedness and a proof of this corollary. 

\begin{rem}
To the best of my knowledge, the only known optimal range of $p$ for \eqref{StrCptGrp} to hold is on square tori $\mathbb{T}^d$, with $p>2+\frac{4}{d}$ (\cite{Bou93}); and on spheres $\mathbb{S}^d$ ($d\geq 3$), with $p>4$ (\cite{BGT04, Her13}). For a general compact Lie group, we do not yet have a conjecture about the optimal range. 
We will prove (Theorem \ref{MainEstimate}) the following distributional estimate: for any $p>2+\frac{4}{r}$, 
\begin{equation}
\lambda\cdot\mu\{(t,x)\in I\times G \mid |e^{it\Delta_g}\varphi(N^{-2}\Delta_g)f(x)|>\lambda\}^{1/p}\leq C N^{\frac{d}{2}-\frac{d+2}{p}}\|f\|_{L^2(G)}
\end{equation}
for all $\lambda\gtrsim N^{\frac{d}{2}-\frac{r}{4}}$. It seems reasonable to conjecture that the above distributional estimate could be upgraded to the estimate \eqref{StrCptGrp} for all $p>2+\frac{4}{r}$ (which is the case for the tori). But this still will not be the optimal range for a general compact Lie group, by looking at the example of the three sphere $\mathbb{S}^3$, which is isomorphic to the group $SU(2)$. The optimal range for $\mathbb{S}^3$ is $p>4$, while Theorem \ref{Main} proves the range $p\geq 10$, and the above conjecture indicates the range $p>6$. Estimate \eqref{StrCptGrp} for $\mathbb{S}^3$ on the optimal range $p>4$ is proved in \cite{Her13} by crucially using the $L^p$-estimates of the spectral clusters for the Laplace-Beltrami operator (\cite{Sog88}) which are optimal on spheres. On tori and more generally compact Lie groups with rank higher than 1, such spectral cluster estimates fail to be optimal and do not help provide the desired Strichartz estimates. On the other hand, the Stein-Tomas argument in our proof of Theorem \ref{Main} seems only sensitive to the $L^\infty$-estimate of the Schr\"{o}dinger kernel (Theorem \ref{dispersive for Schrodinger}) but not to the $L^p$-estimate (as in Proposition \ref{LpLp}). This failure of incorporating $L^p$-estimates for either the spectral clusters or the Schr\"{o}dinger kernel may be one of the reasons why Theorem \ref{Main} is still a step away from the optimal range. 
\end{rem}

\subsection{Organization of the Paper}
The organization of the paper is as follows. In Section \ref{First Reductions}, we will first reduce the Strichartz estimates on $G\cong\tilde{G}/A$ to the spectrally localized Strichartz estimates with respect to Littlewood-Paley projections of the product type on the covering group $\tilde{G}$. In Section \ref{Preliminaries on harmonic analysis}, we will review the basic facts of structures and harmonic analysis on compact Lie groups, including the Fourier transform, root systems, structure of maximal tori, Weyl's character and dimension formulas, and the functional calculus of the  Laplace-Beltrami operator. In Section \ref{The Schrodinger Kernel} we will explicitly write down the Schr\"{o}dinger kernel and interpret the Strichartz estimates as Fourier restriction estimates on the space-time, which then makes applicable the argument of Stein-Tomas type in Section \ref{The Stein-Tomas Argument}. Then comes the core of the paper, Section \ref{Dispersive Estimates on Major Arcs}, in which we will derive dispersive estimates for the Schr\"{o}dinger kernel as the time variable lies in major arcs. 
In Section \ref{Weyl Type Sums and Weyl Differencing for Rational Lattices}, we will estimate some Weyl type exponential sums over the  so-called rational lattices, which in particular will imply the desired bound on the Schr\"{o}dinger kernel for the non-rectangular rational tori. In Section \ref{First Reduction}, we will rewrite the Schr\"{o}dinger kernel for compact Lie groups into an exponential sum over the whole weight lattice instead of just one chamber of the lattice, and will prove the desired bound on the kernel for the case when the variable in the maximal torus stays away from all the cell walls by an application of the Weyl type sum estimate established in Section \ref{Weyl Type Sums and Weyl Differencing for Rational Lattices}. 
In Section \ref{Pseudo-polynomial Behavior of Characters}, we will record two approaches to the pseudo-polynomial behavior of characters, which will be applied to proving the desired bound on the Schr\"{o}dinger kernel when the variable in the maximal torus stays close to the identity. 
In Section \ref{From the Weight Lattice to the Root Lattice}, we further extend the result to the case when the variable in the maximal torus stays close to some corner. Section \ref{RootSubsystems} will finally deal with the case when the variable in the maximal torus stays away from all the corners but close to some cell walls. These cell walls will be identified as those of a root subsystem, and we will then decompose the Schr\"{o}dinger kernel into exponential sums over the root lattice of this root subsystem, thus reducing the problem into one similar to those already discussed in previous sections. This will finish the proof of the main theorem. In Section \ref{Lp Estimates}, we will derive $L^p(G)$ estimates on the Schr\"{o}dinger kernel as an upgrade of the $L^\infty(G)$-estimate.  

Throughout the paper: 
\begin{itemize}
\item 
$A\lesssim B$ means $A\leq CB$ for some constant $C$. 
\item $A\lesssim_{a,b,\cdots}B$ means $A\leq CB$ for some constant $C$ that depends on $a,b,\cdots$.  
\item 
$\Delta,\mu$ are short for the Laplace-Beltrami operator $\Delta_g$ and the associated volume form measure $\mu_g$ respectively when the underlying Riemannian metric $g$ is clear from context.
\item 
$L^p_x, H^s_x, L^p_t, L^p_tL^q_x, L^{p}_{t,x}$ are short for $L^p(M), H^s(M), L^p(I), L^pL^q(I\times M), L^p(I\times M)$ respectively when the underlying manifold $M$ and time interval $I$ are clear from context. 
\item $p'$ denotes the number such that $\frac{1}{p}+\frac{1}{p'}=1$. 
\end{itemize}

\section{First Reductions} 
\label{First Reductions}

\subsection{Littlewood-Paley Theory}\label{Littlewood-Paley Theory and First Reduction}

Let $(M,g)$ be a compact Riemannian manifold and $\Delta$ be the Laplace-Beltrami operator. Let $\varphi$ be a bump function on $\mathbb{R}$. Then for $N\geq 1$, $P_N:=\varphi(N^{-2}\Delta)$ defines a bounded operator on $L^2(M)$ through the functional calculus of $\Delta$. These operators $P_N$ are often called the \textit{Littlewood-Paley projections}. We reduce the problem of obtaining  Strichartz estimates for $e^{it\Delta}$ to those for $P_Ne^{it\Delta}$. 

\begin{prop}\label{LP}
Fix $p, q\geq 2$, $s\geq 0$. Then the Strichartz estimate \eqref{Strichartz} is equivalent to the following statement: Given any bump function $\varphi$, 
\begin{equation*}
\|P_Ne^{it\Delta}f\|_{L^pL^q(I\times M)}\lesssim N^{s}\|f\|_{L^2(M)}, 
\end{equation*}
holds for all dyadic natural numbers $N$ (that is, for $N=2^m$, $m\in\mathbb{Z}_{\geq 0}$). In particular, \eqref{StrCptGrp} reduces to 
\begin{equation}\label{StrLoc}
\|P_Ne^{it\Delta}f\|_{L^p(I\times G)}\leq N^{\frac{d}{2}-\frac{d+2}{p}}\|f\|_{L^2(G)}.
\end{equation}
\end{prop}
This reduction is classical. We refer to \cite{BGT04} for a proof. 

We also record here the Bernstein type inequalities that will be useful in the sequel. 

\begin{prop}[Corollary 2.2 in \cite{BGT04}] Let $d$ be the dimension of $M$. Then for all $1\leq p\leq r\leq \infty$, 
\begin{align}\label{Bernstein}
\|P_Nf\|_{L^r(M)}\lesssim N^{d(\frac{1}{p}-\frac{1}{r})}\|f\|_{L^p(M)}.
\end{align}
\end{prop}

Note that the above proposition in particular implies that \eqref{StrLoc} holds for $N\lesssim 1$ or $p=\infty$.

\subsection{Reduction to a Finite Cover}

\begin{prop}\label{finitecover}
Let $\pi: (\tilde{M},\tilde{g})\to(M,g)$ be a Riemannian covering map between compact Riemannian manifolds (then automatically with finite fibers). Let $\Delta_{\tilde{g}}$, $\Delta_{g}$ be the Laplace-Beltrami operators on $(\tilde{M},\tilde{g})$ and $(M,g)$ respectively and let $\tilde{\mu}$ and $\mu$ be the normalized volume form measures respectively, which define the $L^p$ spaces. Let $\pi^*$ be the pull back map. Define 
$$C^\infty_{\pi}(\tilde{M}):=\pi^*(C^\infty(M)),$$ and similarly define $C_\pi(\tilde{M})$, $L^p_{\pi}(\tilde{M})$ and $H^s_{\pi}(\tilde{M})$.  Then the following statements hold. \\
(i) $\pi^*: C(M)\to C_\pi(\tilde{M})$ and $\pi^*: C^\infty(M)\to C_\pi^\infty(\tilde{M})$ are well-defined and are linear isomorphisms. \\
(ii) $\pi^*: L^p(M)\to L^p_\pi(\tilde{M})$ is well-defined and is an isometry.\\
(iii) $\Delta_{\tilde{g}}$ maps $C^\infty_\pi(\tilde{M})$ into $C^\infty_\pi(\tilde{M})$ and the diagram
\begin{align*}
\xymatrix{
C^\infty(M) 
\ar[r]^{\pi^*} \ar[d]_{\Delta_g} 
& 
C^\infty_\pi(\tilde{M}) \ar[d]^{\Delta_{\tilde{g}}} \\
C^\infty(M) \ar[r]^{\pi^*}& C^\infty_\pi(\tilde{M})
}
\end{align*}
commutes. \\
(iv) $e^{it\Delta_{\tilde{g}}}$ maps $L^2_\pi(\tilde{M})$ into $L^2_\pi(\tilde{M})$ and is an isometry, and the diagrams
\begin{align}\label{eitDeltapicommutes}
\xymatrix{
L^2(M) 
\ar[r]^{\pi^*} \ar[d]_{e^{it\Delta_g}} 
& 
L^2_\pi(\tilde{M}) \ar[d]^{e^{it\Delta_{\tilde{g}}}} \\
L^2(M) \ar[r]^{\pi^*}& L^2_\pi(\tilde{M})
}
\ \ \ \ 
\xymatrix{
L^2(M) 
\ar[r]^{\pi^*} \ar[d]_{P_N} 
& 
L^2_\pi(\tilde{M}) \ar[d]^{P_N} \\
L^2(M) \ar[r]^{\pi^*}& L^2_\pi(\tilde{M})
}
\end{align}
commute, where $P_N$ stands for both $\varphi(N^{-2}\Delta_g)$ and $\varphi(N^{-2}\Delta_{\tilde{g}})$.  \\
(v) $\pi^*: H^s(M)\to H_\pi^s(\tilde{M})$ is well-defined and is an isometry. 
\end{prop}

\begin{proof}
Parts (i), (ii) and (iii) are direct consequences of the definition of a Riemannian covering map. For part (iv), note that (i), (ii) and (iii) together imply that the triples $(L^2(M), C^\infty(M), \Delta_{g})$ and $(L^2_\pi(\tilde{M}), C^\infty_\pi(\tilde{M}), \Delta_{\tilde{g}})$ are isometric as systems of essentially self-adjoint operators on Hilbert spaces, and thus have isometric functional calculus. This implies (iv). Note that the $H^s(M)$ and $H^s_\pi(\tilde{M})$ norms are also defined in terms of the isometric functional calculus of $(L^2(M), C^\infty(M), \Delta_{g})$ and $(L^2_\pi(\tilde{M}), C^\infty_\pi(\tilde{M}), \Delta_{\tilde{g}})$ respectively, which implies (v). 
\end{proof}

Combining Proposition \ref{LP} and \ref{finitecover}, Theorem \ref{Main} is reduced to the following. 

\begin{thm}\label{ReducedMainConj}
Let $K_i$'s be simply connected simple Lie groups and let $G=\mathbb{T}^n\times K_1\times\cdots\times K_m$ be equipped with a rational metric as in Definition \ref{rationalmetric}. Then 
\begin{equation}\label{StrCptGrp2}
\|P_Ne^{it\Delta}f\|_{L^p(I\times G)}\lesssim 
N^{\frac{d}{2}-\frac{d+2}{p}}\|f\|_{L^2(G)}
\end{equation}
holds for $p\geq 2+\frac{8}{r}$ and $N\gtrsim 1$. 
\end{thm}

\subsection{Littlewood-Paley Projections of Product Type}
\label{Littlewood-Paley Projections of the Product Type}

Let $(M,g)$ be the Riemannian product of the compact Riemannian manifolds $(M_j,g_j)$, $j=0,\cdots, m$. 
Any eigenfunction of the Laplace-Beltrami operator $\Delta$ on $M$ with the eigenvalue $\lambda\leq 0$ is of the form $\prod_{j=0}^m\psi_{\lambda_j}$, where each $\psi_{\lambda_j}$ is an eigenfunction of $\Delta_{j}$ on $M_i$ with eigenvalue $\lambda_j\leq 0$, $j=0,\cdots, m$, such that $\lambda=\lambda_0+\cdots+\lambda_n$.

Given any bump function $\varphi$ on $\mathbb{R}$, there always exist bump functions $\varphi_j$, $j=0,\cdots, m$, such that for all $(x_0,\cdots, x_m)\in\mathbb{R}_{\leq 0}^{m+1}$ with $\varphi(x_0+\cdots+x_m)\neq 0$, we have 
$\prod_{j=0}^m\varphi_j(x_j)=1$. In particular, 
\begin{align*}
\varphi\cdot \prod_{j=0}^m\varphi_j(x_j)=\varphi. 
\end{align*}
For $N\geq 1$, define 
\begin{align*}
P_N:&=\varphi(N^{-2}\Delta), \\
{\bf P}_N:&=\varphi_0(N^{-2}\Delta_0)\otimes\cdots\otimes \varphi_m(N^{-2}\Delta_m), 
\end{align*} 
as bounded operators on $L^2(M)$. 
We call ${\bf P}_N$ a \textit{Littlewood-Paley projection of the product type}. We have
\begin{align*}
{\bf P}_N\circ P_N=P_N.
\end{align*} 
This implies that we can further reduce Theorem \ref{ReducedMainConj} into the following. 

\begin{thm}
Let $G=\mathbb{T}^n\times K_1\times\cdots\times K_m$ be equipped with a rational metric. Let $\Delta_0,\Delta_1,\cdots,\Delta_{m}$ be respectively the Laplace-Beltrami operators on $\mathbb{T}^n, K_1,\cdots, K_m$. Let $\varphi_j$ be any bump function for each $j=0,\cdots, m$. For $N\geq 1$, let 
${\bf P}_N=\otimes_{j=0}^{m}\varphi_j(N^{-2}\Delta_j)$. Then 
\begin{equation}\label{StrCptGrp3}
\|{\bf P}_Ne^{it\Delta}f\|_{L^p(I\times G)}\lesssim 
N^{\frac{d}{2}-\frac{d+2}{p}}\|f\|_{L^2(G)}
\end{equation}
holds for $p\geq 2+\frac{8}{r}$ and $N\gtrsim 1$. 
\end{thm}

On the other hand, similarly, for each Littlewood-Paley projection ${\bf P}_N$ of the product type, there exists a bump function $\varphi$ such that $P_N=\varphi(N^{-2}\Delta)$ satisfies $P_N\circ {\bf P}_N={\bf P}_N$. Noting that $\|{\bf P}_Nf\|_{L^2}\lesssim \|f\|_{L^2}$,  \eqref{Bernstein} then implies 
\begin{align}\label{Bernstein'}
\|{\bf P}_Nf\|_{L^r(M)}\lesssim N^{d(\frac{1}{2}-\frac{1}{r})}\|f\|_{L^2(M)}.
\end{align}
for all $2\leq r\leq \infty$.

\section{Preliminaries on harmonic analysis on compact Lie groups}
\label{Preliminaries on harmonic analysis}
\subsection{Fourier Transform}\label{Fourier Transform}

Let $G$ be a compact group and $\hat{G}$ be its Fourier dual, i.e. the set of equivalent classes of irreducible unitary representations of $G$. For $\lambda\in\hat{G}$, let $\pi_\lambda: V_\lambda\to V_\lambda$ be the irreducible unitary representation in the class $\lambda$, and let $d_\lambda=\text{dim}(V_\lambda)$. Let $\mu$ be the normalized Haar measure on $G$. Then for $f\in L^2(G)$, define the Fourier transform
\begin{equation*}
\hat{f}(\lambda)=\int_G f(x)\pi_{\lambda}(x^{-1})d\mu.
\end{equation*}
Then the inverse Fourier transform 
\begin{equation*}
f(x)=\sum_{\lambda\in\hat{G}}d_\lambda\text{tr}(\hat{f}(\lambda)\pi_\lambda(x))
\end{equation*}
converges in $L^2(G)$. 
We have the Plancherel identities
\begin{equation}\label{Plancherel}
\|f\|_{L^2(G)}=(\sum_{\lambda\in\hat{G}}d_\lambda\|\hat{f}(\lambda)\|^2_{HS})^{1/2},
\end{equation}
\begin{equation}\label{PlancherelInnerproduct}
\langle f,g\rangle_{L^2(G)}=\sum_{\lambda\in\hat{G}}d_\lambda
\text{tr}(\hat{f}(\lambda)\hat{g}(\lambda)^*).
\end{equation}
Here $\|\cdot\|_{HS}$ denotes the Hilbert-Schmidt norm of endomorphisms.  

For the convolution 
\begin{equation*}
(f*g)(x)=\int_{G}f(xy^{-1})g(y)\ d\mu(y),
\end{equation*}
we have
\begin{equation}\label{f*ghat}
(f*g)^{\wedge}(j)=\hat{f}(j)\hat{g}(j),
\end{equation}
If $\hat{g}(\lambda)=c_\lambda\cdot \text{Id}_{d_\lambda\times d_\lambda}$, where $c_\lambda$ is a scalar, then
\begin{equation}\label{convolutionL2}
\|f*g\|_{L^2(G)}\leq\sup_{\lambda}|c_\lambda|\cdot\|f\|_{L^2(G)}.
\end{equation}
We also have the Hausdorff-Young inequality
\begin{align}\label{Hausdorff-Young}
\|\hat{f}(\lambda)\|_{HS}\leq d_\lambda^{1/2}\|f\|_{L^1(G)} \ \text{ for all } \lambda\in\hat{G}.
\end{align}

\subsection{Root System and the Laplace-Beltrami Operator}
\label{Representation Theory of Compact Semisimple Lie Groups}

Let $G$ be a compact simply connected semisimple Lie group of dimension $d$ and $\mathfrak{g}$ be its Lie algebra, and let $\mathfrak{g}_{\mathbb{C}}$ denote the complexification of $\mathfrak{g}$. Choose a \textit{maximal torus} $B\subset G$ and let $r$ be the dimension of $B$. Let $\mathfrak{b}$ be the Lie algebra of $B$, which is a \textit{Cartan subalgebra} of $\mathfrak{g}$, and let $\mathfrak{b}_{\mathbb{C}}$ denote its complexification. The Fourier dual $\hat{B}$ of $B$ is isomorphic to a lattice $\Lambda\subset i\mathfrak{b}^*$, which is the \textit{weight lattice}, under the isomorphism 
\begin{align}\label{LambdahatB}
\Lambda\xrightarrow{\sim}\hat{B},\ \  \lambda\mapsto e^\lambda. 
\end{align}


We have the \textit{root space decomposition} $\mathfrak{g}_{\mathbb{C}}=\mathfrak{b}_{\mathbb{C}}\oplus(\oplus_{\alpha\in \Phi}\mathfrak{g}_{\mathbb{C}}^{\alpha})$. 
Here $\Phi\subset i\mathfrak{b}^*$, 
$$\mathfrak{g}_{\mathbb{C}}^{\alpha}=\{X\in\mathfrak{g}_{\mathbb{C}}:\text{Ad}_b(X)=e^{\alpha}(b)X\text{ for all }b\in B\},$$
and $\text{dim}_{\mathbb{C}}\mathfrak{g}_{\mathbb{C}}^{\alpha}=1$. This implies 
\begin{align}\label{|Phi|+r=d}
|\Phi|+r=d. 
\end{align}
The Cartan-Killing form $\langle\ ,\ \rangle$ on $i\mathfrak{b}^*$  becomes a real inner product, and $(\Psi, \langle\ ,\ \rangle)$ becomes an \textit{integral root system}, that is, a finite set $\Phi$ in a finite dimensional real inner product space with the following requirements
\begin{equation}\label{RootSystem}
\left\{
\begin{array}{rl}
\text{(i)} \ \ &\Phi=-\Phi;\\
\text{(ii)}\ \ &\alpha\in \Phi, k\in\mathbb{R}, k\alpha\in\Phi \Rightarrow k=\pm1; \\
\text{(iii)} \ \ &s_\alpha\Phi=\Phi \text{ for all }\alpha\in\Phi;\\
\text{(iv)} \ \ &2\frac{\langle\alpha,\beta\rangle}{\langle\alpha,\alpha\rangle}\in\mathbb{Z} \text{ for all }\alpha,\beta\in\Phi.
\end{array}
\right.
\end{equation}
Here $s_\alpha$ is the reflection about the hyperplane $\alpha^{\perp}$ orthogonal to $\alpha$, that is,
\begin{equation*}
s_\alpha(x):=x-2\frac{\langle x,\alpha\rangle}{\langle\alpha,\alpha\rangle}\alpha.
\end{equation*}
Let $P$ be a system of positive roots such that $\Phi=P\sqcup -P$. Then by \eqref{|Phi|+r=d}, we have 
\begin{align}\label{|P|}
|P|=\frac{d-r}{2}. 
\end{align} 
We can describe the weight lattice $\Lambda$ purely in terms of the root system
\begin{align*}
\Lambda=\{\lambda\in i\mathfrak{b}^*\mid \frac{2\langle\lambda,\alpha\rangle}{\langle\alpha,\alpha\rangle}\in\mathbb{Z}, \text{ for all }\alpha\in\Phi \}. 
\end{align*}
The set $\Phi$ of roots generate the \textit{root lattice} $\Gamma$ and we have $\Gamma\subset \Lambda$ and $\Lambda/\Gamma$ is finite.   

Let
\begin{align*}
\Lambda^+:=\{\lambda\in i\mathfrak{b}^*\mid \frac{2\langle\lambda,\alpha\rangle}{\langle\alpha,\alpha\rangle}\in\mathbb{Z}_{\geq 0}, \text{ for all }\alpha\in P \}
\end{align*}
be the set of dominant weights. We describe $\Lambda, \Lambda^+$ in terms of a basis. Let $\{\alpha_1,\cdots,\alpha_r\}$ be the set of simple roots in $P$. Let $\{w_1,\cdots, w_r\}$ be the corresponding fundamental weights, i.e. the dual basis to the coroot basis $\{\frac{2\alpha_1}{\langle\alpha_1,\alpha_1\rangle}, \cdots, \frac{2\alpha_r}{\langle\alpha_r,\alpha_r\rangle}\}$. 
Then 
\begin{align*}
\Lambda&=\mathbb{Z}w_1+\cdots+\mathbb{Z}w_r, \\
\Lambda^+&=\mathbb{Z}_{\geq 0}w_1+\cdots+\mathbb{Z}_{\geq 0}w_r.
\end{align*}
Let 
\begin{align}\label{fundamentalchamber}
C=\mathbb{R}_{> 0}w_1+\cdots+\mathbb{R}_{> 0}w_r
\end{align}
be the \textit{fundamental Weyl chamber}, and we have the decomposition 
\begin{align}\label{chamberdecomposition}
i\mathfrak{b}^*=(\bigsqcup_{s\in W} sC)\  \sqcup\  (\bigcup_{\alpha\in\Phi}\{\lambda\in i\mathfrak{b}^*: \langle\lambda,\alpha\rangle=0\}),
\end{align}
where $W$ is the Weyl group. Here $\sqcup$ stands for disjoint union.

Define 
\begin{equation}\label{definition of rho}
\rho:=\frac{1}{2}\sum_{\alpha\in P}\alpha=\sum_{i=1}^rw_i.
\end{equation} 
Then we have
\begin{align*}
\hat{G}\cong \Lambda^+
\end{align*}
such that the irreducible representation $\pi_\lambda$ corresponding to $\lambda\in\Lambda^+$ has the character $\chi_\lambda$ and dimension $d_\lambda$ given by Weyl's formulas
\begin{align}
\chi_{\lambda}\Large|_{B}&=\frac{\sum_{s\in W}(\det s) e^{s(\lambda+\rho)}}{\sum_{s\in W}(\det s) e^{s\rho}},\label{CharacterChi}\\
d_\lambda&=\frac{\prod_{\alpha\in P}\langle\alpha,\lambda+\rho\rangle}{\prod_{\alpha\in P}\langle\alpha,\rho\rangle}.\label{Weyldimension}
\end{align}

Let $H\in\mathfrak{b}$. We can think of $-iH$ as a real linear functional on $i\mathfrak{b}^*$, and by the Cartan-Killing inner product on $i\mathfrak{b}^*$, we thus get a correspondence between $H\in\mathfrak{b}$ and an element in $i\mathfrak{b}^*$,  still denoted as $H$. Under this correspondence, $e^{\lambda(H)}=e^{i\langle\lambda, H\rangle}$ and we rewrite Weyl's character formula as 
\begin{align}\label{WeylCharacter}
\chi_{\lambda}(\exp H)=\frac{\sum_{s\in W}\det s\ e^{i\langle s(\lambda+\rho), H\rangle}}{\sum_{s\in W}\det s\ e^{i\langle\rho, H\rangle}}.
\end{align}
Also under this correspondence between $\mathfrak{b}$ and $i\mathfrak{b}^*$, we have 
\begin{align*}
B\cong i\mathfrak{b}^*/2\pi \Gamma^{\vee}, 
\end{align*}
where 
\begin{align*}
\Gamma^{\vee}=\mathbb{Z}\frac{2\alpha_1}{\langle\alpha_1,\alpha_1\rangle}+\cdots+\mathbb{Z}\frac{2\alpha_r}{\langle\alpha_r,\alpha_r\rangle}
\end{align*}
is the \textit{coroot lattice}. 

We define the \textit{cells} to be the connected components of $\{H\in i\mathfrak{b}^*/2\pi \Gamma^{\vee}\mid\langle\alpha, H\rangle\notin 2\pi \mathbb{Z}\}$ and call 
$\{H\in i\mathfrak{b}^*/2\pi \Gamma^{\vee}\mid\langle\alpha, H\rangle\in 2\pi \mathbb{Z}\}$ the \textit{cell walls}.

We also record here Weyl's integral formula. 
Let $f\in L^1(G)$ be invariant under the adjoint action of $G$. Then 
\begin{equation}\label{Weyl integration formula}
\int_G f\ d\mu=\frac{1}{|W|}\int_B f(b)|D_P(b)|^2\ db
\end{equation}
Here $d\mu, db$ are respectively the normalized Haar measures of $G$ and $B$, and 
\begin{equation*}
D_P(H)=\sum_{s\in W}\det s\ e^{i\langle\rho, H\rangle}
\end{equation*} 
is the \textit{Weyl denominator}.

Finally we describe the functional calculus of the Laplace-Beltrami operator $\Delta$. Given any irreducible unitary representation $(\pi_\lambda, V_\lambda)$ of $G$ in the class $\lambda\in\hat{G}\cong\Lambda^+$, the operator $\Delta$ acts on the space $\mathcal{M}_\lambda=\{\text{tr}(\pi_\lambda T)\mid T\in\text{End}(V_\lambda)\}$ of matrix coefficients by 
\begin{align*}
\Delta f=-k_\lambda f, \ \ \text{for all }f\in\mathcal{M}_\lambda,  \ \lambda\in \hat{G},
\end{align*}
where 
\begin{equation}\label{klambda}
k_\lambda=|\lambda+\rho|^2-|\rho|^2.
\end{equation}
Let $f\in L^2(G)$ and consider the inverse Fourier transform 
$
f(x)=\sum_{\lambda\in\Lambda^+}d_\lambda\text{tr}(\pi_\lambda(x)\hat{f}(\lambda))$, then for any bounded Borel function $F:\mathbb{R}\to\mathbb{C}$, we have 
\begin{equation*}
F(\Delta)f=\sum_{\lambda\in\Lambda^+}F(-k_\lambda)d_\lambda\text{tr}(\pi_\lambda(x)\hat{f}(\lambda)). 
\end{equation*}
In particular, we have 
\begin{equation}\label{SchFlo}
e^{it\Delta}f=\sum_{\lambda\in\Lambda^+}e^{-itk_\lambda}d_\lambda\text{tr}(\pi_\lambda(x)\hat{f}(\lambda)),
\end{equation}
\begin{equation}\label{SchKer}
P_Ne^{it\Delta}f=\sum_{\lambda\in\Lambda^+}\varphi(-\frac{k_\lambda}{N^2})e^{-itk_\lambda}d_\lambda\text{tr}(\pi_\lambda(x)\hat{f}(\lambda)).
\end{equation}

\begin{exmp}\label{SU21}
Let $M=\text{SU}(2)$, which is of dimension 3 and rank 1.  Let $\mathfrak{a}\cong\mathbb{R}$ be the Cartan subalgebra and 
$A\cong \mathbb{R}/2\pi\mathbb{Z}$ be the maximal torus. The root system is $\{\pm \alpha\}$, where $\alpha$ acts on $\mathfrak{a}$ by $\alpha(\theta)=2\theta$. The fundamental weight  is $w=\frac{1}{2}\alpha$. We normalize the Cartan-Killing form so that  
$|w|=1$. 
The Weyl group $W$ is of order 2, and acts on $\mathfrak{a}$ as well as $\mathfrak{a}^*$ through multiplication by $\pm 1$. For $m\in \mathbb{Z}_{\geq 0}\cong\mathbb{Z}_{\geq 0}w=\Lambda^+$, we have 
\begin{align}
d_m&=m+1,\label{dm} \\ 
\chi_{m}(\theta)&=
\frac{e^{i(m+1)\theta}-e^{-i(m+1)\theta}}
{e^{i\theta}-e^{-i\theta}}=\frac{\sin (m+1)\theta}{\sin\theta}, \ \ \theta\in\mathbb{R}/2\pi\mathbb{Z}, \label{chim}\\
k_m&=(m+1)^2-1. \label{km}
\end{align}
\end{exmp}

\section{The Schr\"odinger Kernel}
\label{The Schrodinger Kernel}

Let $f\in L^2(G)$. Then \eqref{SchKer} implies  
\begin{equation*}
(P_Ne^{it\Delta }f)^{\wedge}(\lambda)=\varphi(\frac{k_\lambda}{N^2})e^{-itk_\lambda}\hat{f}(\lambda). 
\end{equation*} 
Define
\begin{equation*}
(K_N(t,\cdot))^{\wedge}(\lambda)
=\varphi(\frac{k_\lambda}{N^2})e^{-itk_\lambda}\text{Id}_{d_\lambda\times d_\lambda},
\end{equation*}
which implies 
\begin{align}\label{kernel} 
K_N(t,x)=\sum_{\lambda\in\Lambda^+}\varphi(\frac{k_\lambda}{N^2})e^{-itk_\lambda}d_\lambda\chi_\lambda(x). 
\end{align}
Then we can write 
\begin{equation*}
P_Ne^{it\Delta}f=K_N(t,\cdot)*f=f*K_N(t,\cdot),
\end{equation*}
and we call $K_N(t,x)$ the \textit{Schr\"odinger kernel}. Incorporating  \eqref{Weyldimension}, \eqref{WeylCharacter} and \eqref{klambda} into \eqref{kernel}, we get 
\begin{align}\label{SchrodingerKernel}
K_N(t,x)=\sum_{\lambda\in\Lambda^+}e^{-it(|\lambda+\rho|^2-|\rho|^2)}
\varphi(\frac{|\lambda+\rho|^2-|\rho|^2}{N^2})\frac{\prod_{\alpha\in P}\langle \alpha,\lambda+\rho\rangle}{\prod_{\alpha\in P}\langle \alpha,\rho\rangle}
\frac{\sum_{s\in W}\det{(s)}e^{i\langle s(\lambda+\rho),H\rangle}}{\sum_{s\in W}\det{(s)}e^{i\langle s(\rho), H\rangle}}
\end{align}

\begin{exmp} 
Specializing the Schr\"odinger kernel \eqref{SchrodingerKernel} to $G=\text{SU}(2)$, using \eqref{dm}, \eqref{chim}, and \eqref{km}, we have
\begin{align}\label{kernel for SU(2)}
K_N(t,\theta)&=\sum_{m=0}^\infty\varphi(\frac{(m+1)^2-1}{N^2})(m+1)e^{-i((m+1)^2-1)t}\frac{e^{i(m+1)\theta}-e^{-i(m+1)\theta}}
{e^{i\theta}-e^{-i\theta}}, \ \ \theta\in\mathbb{R}/2\pi\mathbb{Z}.
\end{align}

\end{exmp}

More generally, let $G=\mathbb{R}^n/2\pi\Gamma_0\times K_1\times\cdots\times K_m$ be equipped with a rational metric $g$ as in Definition \ref{rationalmetric}. Let $\Lambda_0$ be the dual lattice of $\Gamma_0$, and $\Lambda_j$ be the weight lattice for $K_j$, $j=1,\cdots, m$. Let ${\bf P}_N=\otimes_{j=0}^{m}\varphi_j(N^{-2}\Delta_j)$ be a Littlewood-Paley projection of the product type as described in Section \ref{Littlewood-Paley Projections of the Product Type}. Define the \textit{Schr\"odinger kernel} ${\bf K}_N$ on $G$ by 
\begin{align}\label{DefnSchr}
{\bf P}_Ne^{it\Delta}f=f*{\bf K}_N(t,\cdot)={\bf K}_N(t,\cdot)*f.
\end{align}
Then 
\begin{align}\label{TheKernel}
{\bf K}_N=\prod_{j=0}^{m}K_{N,j},
\end{align}
where the $K_{N,j}$'s are respectively the Schr\"odinger kernels on each component of $G$
\begin{align*}
K_{N,0}&=\sum_{\lambda_0\in\Lambda_0}\varphi_0(\frac{-|\lambda_0|^2}{\beta_0 N^2})e^{-it\beta_0^{-1}|\lambda_0|^2}e^{i\langle\lambda_0,H_0\rangle}, \\
K_{N, j}&=\sum_{\lambda_j\in\Lambda_j^+}\varphi_{j}(\frac{-|\lambda_j+\rho_j|^2+|\rho_j|^2}{\beta_jN^2})e^{it\beta_j^{-1}(-|\lambda_j+\rho_j|^2+|\rho_j|^2)}d_{\lambda_j}\chi_{\lambda_j},
\end{align*} 
$j=1,\cdots, m$. Here the $\rho_j$'s are defined in terms of \eqref{definition of rho}. We also write 
\begin{align*}
{\bf K}_N=\sum_{\lambda\in\widehat{G}}\varphi(\lambda, N) e^{-itk_\lambda}d_\lambda\chi_\lambda, 
\end{align*}
where
\begin{align*}
\lambda&=(\lambda_0,\ldots,\lambda_m)\in 
\widehat{G}=\Lambda_0 \times \Lambda_1^+\times\cdots\times\Lambda_m^+, \\
-k_\lambda&=-\beta_0^{-1}|\lambda_0|^2+\sum_{j=1}^m\beta_j^{-1}(-|\lambda_j+\rho_j|^2+|\rho_j|^2),\numberthis \label{TheEigenvalue}\\
\varphi(\lambda, N)&=\varphi_0(\frac{-|\lambda_0|^2}{\beta_0N^2})\cdot\prod_{j=1}^n\varphi_{j}(\frac{-|\lambda_j+\rho_j|^2+|\rho_j|^2}{\beta_jN^2}),\numberthis \label{cutoff}\\
d_\lambda&=\prod_{j=1}^m d_{\lambda_j}, \ \chi_\lambda=e^{i\langle\lambda_0,H_0\rangle}\prod_{j=1}^m\chi_{\lambda_j}. 
\end{align*}

Tracking all the definitions, we get the following lemma. 

\begin{lem} \label{CountingLemma}
Let $d,r$ be respectively the dimension and rank of $G$.  \\
(i) $|\{\lambda\in\widehat{G}\mid k_\lambda\lesssim N^2\}|\lesssim N^r$.\\
(ii) $d_\lambda\lesssim N^{\frac{d-r}{2}}$, uniformly for all $\lambda\in\hat{G}$ such that $k_\lambda\lesssim N^2$. 
\end{lem}

Now we interpret the Strichartz estimates on $G$ as \textit{Fourier restriction estimates}.  

\begin{lem}\label{DGamma}
For a compact simply connected semisimple Lie group $G$ and its weight lattice $\Lambda$, there exists $D\in{\mathbb{N}}$ such that $\langle\lambda_1,\lambda_2\rangle\in D^{-1}{\mathbb{Z}}$ for all $\lambda_1,\lambda_2\in\Lambda$. 
\end{lem}
\begin{proof}
Let $\Phi$ be the set of roots for $G$. Then by Lemma 4.3.5 in \cite{Var84}, $\langle\alpha,\beta\rangle$ are rational numbers for all $\alpha,\beta\in \Phi$. Let $S=\{\alpha_1,\cdots,\alpha_r\}\subset \Phi$ be a system of simple roots. Since the set of fundamental weights $\{w_1,\cdots,w_n\}$ forms a dual basis to $\{\frac{2\alpha_1}{\langle\alpha_1,\alpha_1\rangle},\cdots, \frac{2\alpha_r}{\langle\alpha_r,\alpha_r\rangle}\}$ with respect to the Cartan-Killing form $\langle\ ,\ \rangle$, and $\langle\alpha_i,\alpha_j\rangle$ are rational numbers for all $i,j=1,\cdots,r$, we have that the $w_j$'s can be expressed as linear combinations of the $\alpha_j$'s with rational coefficients. This implies that $\langle w_i,w_j\rangle$ are rational numbers for all $i,j=1,\cdots,r$. Since there are only finitely many such numbers as $\langle w_i,w_j\rangle$, there exists $D\in\mathbb{N}$ so that $\langle w_i,w_j\rangle\in D^{-1}\mathbb{Z}$ for all $i,j=1,\cdots,r$. Thus $\langle\lambda_1,\lambda_2\rangle\in D^{-1}\mathbb{Z}$ for all $\lambda_1,\lambda_2\in\Lambda$, since $\Lambda=\mathbb{Z}w_1+\cdots+\mathbb{Z}w_n$.  
\end{proof}

For $G=\mathbb{R}^n/2\pi\Gamma_0\times K_1\times\cdots\times K_m$, by the previous lemma, there exists for each $j=1,\cdots,m$ some $D_j\in\mathbb{N}$ such that $\langle\lambda,\mu\rangle\in D_j^{-1}\mathbb{Z}$ for all $\lambda,\mu\in \Lambda_j^+$, which implies by \eqref{definition of rho} that 
$$-|\lambda_j+\rho_j|^2+|\rho_j|^2=-|\lambda_j|^2-\langle\lambda_j,2\rho_j\rangle\in D_j^{-1}\mathbb{Z}$$ for all $\lambda_j\in \Lambda_j$. Also recall that we require that there exists some $D\in\mathbb{N}$ such that $\langle u,v \rangle\in D^{-1}\mathbb{Z}$ for all $u,v\in\Gamma_0$. This implies that there also exists some $D_0\in\mathbb{N}$ such that $\langle\lambda,\mu\rangle\in D_0^{-1}\mathbb{Z}$ for all $\lambda,\mu\in \Lambda_0$. 
By Definition \ref{rationalmetric} of a rational metric, there exists some $D_{*}>0$ such that 
\begin{align*}
\beta_0^{-1},\cdots, \beta_m^{-1}\in D_{*}^{-1}\mathbb{N}. 
\end{align*}
Define 
\begin{align}\label{ThePeriod}
T=2\pi D_{*}\cdot\prod_{j=0}^m D_j.
\end{align}
Then \eqref{TheEigenvalue} implies that $T k_\lambda\in 2\pi\mathbb{Z}$, which then implies that the Schr\"odinger kernel as in \eqref{TheKernel} is periodic in $t$ with a period of $T$. 
Thus we may view the time variable $t$ as living on the circle $\mathbb{T}=\mathbb{R}/T\mathbb{Z}$. 
Now the formal dual to the operator 
\begin{align}\label{T}
{\bf T}: L^2(G)\to L^p(\mathbb{T}\times G), \ \ f\mapsto {\bf P}_Ne^{it\Delta}
\end{align}
is computed to be 
\begin{align}\label{T^*}
{\bf T}^*: L^{p'}(\mathbb{T}\times G)\to L^2(G), \ \ F\mapsto \int_{\mathbb{T}}{\bf P}_Ne^{-is\Delta}F(s,\cdot)\ \frac{ds}{T},
\end{align}
and thus 
\begin{align}\label{TT^*}
{\bf T}{\bf T}^*: L^{p'}(\mathbb{T}\times G)\to L^{p}(\mathbb{T}\times G), \ \ F\mapsto \int_{\mathbb{T}}{\bf P}_N^2e^{i(t-s)\Delta}F(s,\cdot)\ \frac{ds}{T}
=\tilde{{\bf K}}_N*F,
\end{align}
where 
\begin{align*}
\tilde{{\bf K}}_N=\sum_{\lambda\in\hat{G}}\varphi^2(\lambda, N)e^{-itk_\lambda}d_\lambda\chi_\lambda={\bf K}_N*{\bf K}_N.
\end{align*}
Note that the cutoff function $\varphi^2(\lambda, N)$ still defines a Littlewood-Paley projection of the product type and $\tilde{\bf K}_N$ is the associated Schr\"odinger kernel. 
Now the argument of ${\bf TT}^*$ says that the boundedness of the operators \eqref{T}, \eqref{T^*} and \eqref{TT^*} are all equivalent, thus the Strichartz estimate in \eqref{StrLoc} is equivalent to the \textit{space-time Strichartz estimate} 
\begin{equation}\label{SpaceTimeStr}
\|\tilde{{\bf K}}_N*F\|_{L^p(\mathbb{T}\times G)}\lesssim N^{d-\frac{2(d+2)}{p}}\|F\|_{L^{p'}(\mathbb{T}\times G)}.
\end{equation} 

We have the \textit{space-time Fourier transform} on 
$\mathbb{T}\times G$ as follows. For $(n,\lambda)\in \frac{2\pi}{T}{\mathbb{Z}}\times \hat{G}$, we have 
\begin{align}\label{STFTKN}
\widehat{{\bf K}}_N(n,\lambda)=\left\{\begin{array}{ll}\varphi(\lambda, N)\cdot \text{Id}_{d_\lambda\times d_\lambda},&\text{ if }n=-k_{\lambda},\\
0, & \text{ otherwise}.
\end{array}\right.
\end{align} Similarly, for $f\in L^2(G)$, we have 
\begin{align}\label{STFTPNf}
({\bf P}_Ne^{it\Delta}f(x))^{\wedge}(n,\lambda)=\left\{\begin{array}{ll}\varphi(\lambda, N)\cdot \hat{f}(\lambda),&\text{ if }n=-k_{\lambda},\\
0, & \text{ otherwise}.
\end{array}\right.
\end{align}
For $m(t)=\sum_{n\in \frac{2\pi}{T}\mathbb{Z}}\hat{m}(n)e^{itn}$, we compute 
\begin{align}\label{F transform of kernel}
(m{\bf K}_N)^{\wedge}(n,\lambda)=\hat{m}(n+k_\lambda)
\varphi(\lambda, N)\text{Id}_{d_\lambda\times d_\lambda}.
\end{align}

\section{The Stein-Tomas Argument}
\label{The Stein-Tomas Argument}
Throughout this section, $\mathbb{S}^1$ stands for the standard circle of unit length, and $\|\cdot\|$ stands for the distance from 0 on $\mathbb{S}^1$. Define
\begin{align*}
\mathcal{M}_{a,q}:=\{t\in\mathbb{S}^1\mid \|t-\frac{a}{q}\|<\frac{1}{qN}\}
\end{align*}
where  
\begin{align*}
a\in\mathbb{Z}_{\geq 0}, \ \ q\in\mathbb{N}, \ \ a<q, \ \ (a,q)=1, \ \ q<N.  
\end{align*}
We call such $\mathcal{M}_{a,q}$'s as \textit{major arcs}, which are reminiscent of the Hardy-Littlewood circle method. 
We will prove the following key dispersive estimate. 

\begin{thm}\label{dispersive for Schrodinger}
Let ${\bf K}_N$ be the Schr\"{o}dinger kernel \eqref{TheKernel} and $T$ be the period \eqref{ThePeriod}. Then  
\begin{align*}
|{\bf K}_N(t,x)|\lesssim \frac{N^d}{(\sqrt{q}(1+N\|\frac{t}{2\pi D}-\frac{a}{q}\|^{1/2}))^{r}}
\end{align*}
for $\frac{t}{2\pi D}\in\mathcal{M}_{a,q}$, uniformly in $x\in G$.  
\end{thm}

Noting the product structure \eqref{TheKernel} of ${\bf K}_N$, the above theorem reduces to the cases on irreducible components of $G$. 

\begin{thm}\label{MainEstimate}
(i) Given $G=\mathbb{T}^d=\mathbb{R}^d/2\pi\Gamma$ such that there exists $D\in\mathbb{R}$ for which $\langle\lambda,\mu\rangle\in D^{-1}\mathbb{Z}$ for all $\lambda,\mu\in\Gamma$. Then the Schr\"{o}dinger kernel 
\begin{align*}
K_N(t,H)=\sum_{\lambda\in\Lambda}\varphi(\frac{|\lambda|^2}{N^2})e^{-it|\lambda|^2+i\langle\lambda, H\rangle}
\end{align*}
satisfies 
\begin{align*}
|K_N(t,H)|\lesssim (\frac{N}{\sqrt{q}(1+N\|\frac{t}{2\pi D}-\frac{a}{q}\|^{\frac{1}{2}})})^d
\end{align*}
for $\frac{t}{2\pi D}\in\mathcal{M}_{a,q}$, uniformly in $H\in\mathbb{T}^n$. \\
(ii) Let $G$ be a compact simply connected semisimple Lie group. Let $\Lambda$ be the weight lattice for which $\langle\lambda,\mu\rangle\in D^{-1}\mathbb{Z}$ for all $\lambda,\mu\in\Lambda$ for some $D\in\mathbb{R}$. Let $K_N$ be the Schr\"{o}dinger kernel as defined in \eqref{SchrodingerKernel}. Then 
\begin{align}\label{KeyEst}
|K_N(t,x)|\lesssim \frac{N^d}{(\sqrt{q}(1+N\|\frac{t}{2\pi D}-\frac{a}{q}\|^{1/2}))^{r}}
\end{align}
for $\frac{t}{2\pi D}\in\mathcal{M}_{a,q}$, uniformly in $x\in G$. 
\end{thm}

We will prove this theorem in the next section. Now we show how this theorem implies Strichartz estimates. 

\begin{thm}\label{simpleconnected}
Let $G=\mathbb{T}^n\times K_1\times \cdots\times K_m$ be equipped with a rational metric $\tilde{g}$ and $T$ be a period of the Schr\"odinger flow as in \eqref{ThePeriod}. Let $d,r$ be the dimension and rank of $G$ respectively. 
Let $f\in L^2(G)$, $\lambda>0$ and define 
\begin{align*}
m_\lambda=\mu\{(t,x)\in \mathbb{T}\times G\mid |P_Ne^{it\Delta}f(x)|>\lambda|\}
\end{align*}
where $\mu=dt\cdot d\mu_G$, with $dt$ being the standard measure on $\mathbb{T}=\mathbb{R}/ T \mathbb{Z}$ and $d\mu_G$ being the Haar measure on $G$. 
Let 
\begin{align*}
p_0=\frac{2(r+2)}{r}.
\end{align*} 
Then the following statements hold true. \\
(I)
\begin{align*}
m_\lambda\lesssim_{\varepsilon} N^{\frac{dp_0}{2}-(d+2)+\varepsilon}\lambda^{-p_0}\|f\|^{p_0}_{L^{2}(G)}, \ \ \text{for all }\lambda\gtrsim N^{\frac{d}{2}-\frac{r}{4}}, \ \ \varepsilon>0.
\end{align*} 
(II)
\begin{align*}
m_\lambda\lesssim N^{\frac{dp}{2}-(d+2)}\lambda^{-p}\|f\|^{p}_{L^{2}(G)}, \ \ \text{for all }\lambda\gtrsim N^{\frac{d}{2}-\frac{r}{4}}, \ \ p>p_0.
\end{align*}
(III) 
\begin{align}\label{MainTheorem}
\|{\bf P}_Ne^{it\Delta}f\|_{L^p(\mathbb{T}\times G)}\lesssim N^{\frac{d}{2}-\frac{d+2}{p}}\|f\|_{L^2(G)}
\end{align}
holds for all $p\geq 2+\frac{8}{r}$. \\
(IV) Assume it holds that 
\begin{align}\label{epsilonassumption}
\|{\bf P}_Ne^{it\Delta}f\|_{L^p(\mathbb{T}\times G)}\lesssim_{\varepsilon} N^{\frac{d}{2}-\frac{d+2}{p}+\varepsilon}\|f\|_{L^2(G)}
\end{align}
for some $p>p_0$,  then \eqref{MainTheorem} holds for all $q>p$. 
\end{thm}

The proof strategy of this theorem is a Stein-Tomas type argument, similar to the proofs of Propositions 3.82, 3.110, 3.113 in \cite{Bou93}. The new ingredient is the non-abelian Fourier transform. We detail the proof in the following. 

Let $\omega\in C_c^\infty(\mathbb{R})$ such that $\omega\geq 0$, $\omega(x)=1$ for all $|x|\leq 1$ and $\omega(x)=0$ for all $|x|\geq 2$. Let $N$ be a dyadic natural number. Define 
\begin{align*}
\omega_{\frac{1}{N^2}}&:=\omega(N^2\cdot),\\
\omega_{\frac{1}{NM}}&:=\omega(NM\cdot)-\omega(2NM\cdot),
\end{align*}
where
\begin{align*}
1\leq M<N, \ \ M\text{ dyadic}.
\end{align*}
Let 
\begin{align*}
N_1=\frac{N}{2^{10}}, \ \ 1\leq Q<N_1, \ \ Q\text{ dyadic}.
\end{align*}
Then 
\begin{align}\label{partitionofunity1}
\sum_{Q\leq M\leq N}\omega_{\frac{1}{NM}}=1,\ \ \text{on }\left[-\frac{1}{NQ},\frac{1}{NQ}\right],
\end{align}
\begin{align}\label{partitionofunity2}
\sum_{Q\leq M\leq N}\omega_{\frac{1}{NM}}=0,\ \ \text{outside }\left[-\frac{2}{NQ},\frac{2}{NQ}\right].
\end{align}
Write 
\begin{align}\label{partition on circle}
1=\sum_{1\leq Q\leq N_1}\sum_{Q\leq M\leq N}
\left[\left(\sum_{\substack{(a,q)=1,\\ Q\leq q<2Q}}\delta_{a/q}\right)*\omega_{\frac{1}{NM}}\right](\frac{t}{T})+\rho(t).
\end{align}
Note the major arc disjointness property
\begin{align*}
\left(\frac{a_1}{q_1}+\left[-\frac{2}{NQ_1},\frac{2}{NQ_1}\right]\right)\cap\left(\frac{a_2}{q_2}+\left[-\frac{2}{NQ_2},\frac{2}{NQ_2}\right]\right)=\emptyset
\end{align*}
for $(a_i,q_i)=1$, $Q_i\leq q_i<2Q_i$, $i=1,2$, $Q_1\leq Q_2\leq N_1$.  This in particular implies 
\begin{align}\label{0<rho<1}
0\leq \rho(t)\leq 1, \ \ \text{for all } t\in\mathbb{R}/T\mathbb{Z},
\end{align}
\begin{align}\label{deltaaqhat}
\left[\left(\sum_{\substack{(a,q)=1,\\ Q\leq q<2Q}}\delta_{a/q}\right)*\omega_{\frac{1}{NM}}(\frac{\cdot}{T})\right]
^{\wedge}(0)=\frac{1}{T}\int_0^T\left(\sum_{\substack{(a,q)=1,\\ Q\leq q<2Q}}\delta_{a/q}\right)*\omega_{\frac{1}{NM}}(\frac{t}{T})\ dt\leq \frac{2Q^2}{NM},
\end{align}
which implies 
\begin{align}\label{hatrho0}
1\geq |\hat{\rho}(0)|\geq
1-\sum_{1\leq Q\leq N_1}\sum_{Q\leq M\leq N}\left|\left[\left(\sum_{\substack{(a,q)=1,\\ Q\leq q<2Q}}\delta_{a/q}\right)*\omega_{\frac{1}{NM}}(\frac{\cdot}{T})\right]
^{\wedge}(0)\right|\geq 1-\frac{8N_1}{N}\geq\frac{1}{2}. 
\end{align}
By Dirichlet's lemma on rational approximations, for any $\frac{t}{T}\in\mathbb{S}^1$, there exists $a,q$ with $a\in\mathbb{Z}_{\geq 0}$, $q\in\mathbb{N}$, $(a,q)=1$, $q\leq N$, such that  
$|\frac{t}{T}-\frac{a}{q}|<\frac{1}{qN}$. If $\rho(\frac{t}{T})\neq 0$, then \eqref{partitionofunity1} implies 
$q>N_1=\frac{N}{2^{10}}$. This implies by \eqref{KeyEst} and \eqref{0<rho<1} that 
\begin{align}\label{minorarcLinfity}
\|\rho(t)K_N(t,x)\|_{L^\infty(\mathbb{T}\times G)}\lesssim N^{d-\frac{r}{2}}.
\end{align}

Now define coefficients $\alpha_{Q,M}$ such that
\begin{align}\label{introduce alphaQM}
\left[\left(\sum_{\substack{(a,q)=1,\\ Q\leq q<2Q}}\delta_{a/q}\right)*\omega_{\frac{1}{NM}}(\frac{\cdot}{T})\right]
^{\wedge}(0)=\alpha_{Q,M}\hat{\rho}(0),
\end{align}
then \eqref{deltaaqhat} and \eqref{hatrho0} imply 
\begin{align}\label{alphaQMbound}
\alpha_{Q,M}\lesssim\frac{Q^2}{NM}.
\end{align}
Write
\begin{align*}
{\bf K}_N(t,x)=&
\sum_{Q\leq N_1}\sum_{Q\leq M\leq N}{\bf K}_N(t,x)\left[\left(\left(\sum_{(a,q)=1,Q\leq q<2Q}\delta_{a/q}\right)*\omega_{\frac{1}{NM}}(\frac{\cdot}{T})\right)-\alpha_{Q,M}\rho\right](t)\\
&+\left(1+\sum_{Q,M}\alpha_{Q,M}\right){\bf K}_N(t,x)\rho(t),\numberthis
\end{align*}
and define
\begin{align}\label{introduce LambdaQM}
\Lambda_{Q,M}(t,x):={\bf K}_N(t,x)\left[\left(\left(\sum_{(a,q)=1,Q\leq q<2Q}\delta_{a/q}\right)*\omega_{\frac{1}{NM}}(\frac{\cdot}{T})\right)-\alpha_{Q,M}\rho\right](t).
\end{align}
Then from \eqref{KeyEst}, \eqref{minorarcLinfity}, \eqref{alphaQMbound}, we have
\begin{align}\label{majorarc}
\|\Lambda_{Q,M}\|_{L^\infty(\mathbb{T}\times G)}\lesssim N^{d-\frac{r}{2}}(\frac{M}{Q})^{r/2}.
\end{align}

Next, we estimate $\hat{\Lambda}_{Q,M}$. From \eqref{F transform of kernel}, for 
\begin{align*}
n\in \frac{2\pi}{T}\mathbb{Z}\cong \hat{\mathbb{T}},\ \ \lambda\in\hat{G}, 
\end{align*} 
we have 
\begin{align}\label{Lambda and lambda}
\hat{\Lambda}_{Q,M}(n,\lambda)=\lambda_{Q,M}(n,\lambda)\cdot
\text{Id}_{d_\lambda\times d_\lambda}
\end{align}
where
\begin{align}\label{lambdaQM}
\lambda_{Q,M}(n,\lambda)=\varphi(\lambda, N)
\left[\left(\sum_{(a,q)=1,Q\leq q<2Q}\delta_{a/q}\right)^{\wedge}\cdot
\hat{\omega}_{\frac{1}{NM}}(T\cdot )-\alpha_{Q,M}\hat{\rho}\right](n+k_\lambda).
\end{align}
Note that \eqref{introduce alphaQM} immediately implies 
\begin{align}\label{bound on fourier transform of Lambda2}
\lambda_{Q,M}(n,\lambda)=0, \ \ \text{for }n+k_\lambda=0.
\end{align}
Let $d(m,Q)$ denote the number of divisors of $m$ less than $Q$, using Lemma 3.33 in \cite{Bou93}, 
\begin{align}\label{bound on fourier transform of deltas}
|(\sum_{(a,q)=1,Q\leq q< 2Q}\delta_{a/q})^{\wedge}(Tn)|\lesssim_{\varepsilon} d(\frac{Tn}{2\pi},Q)Q^{1+\varepsilon},\ \  n\neq0,\ \varepsilon>0,
\end{align}
we get
\begin{align}\label{lambdaQMbound}
|\lambda_{Q,M}(n,\lambda)|\lesssim_{\varepsilon}\varphi(\lambda,N)\frac{Q^{1+\varepsilon}}{NM}d(\frac{T(n+k_\lambda)}{2\pi},Q)
+\frac{Q^2}{NM}|\hat{\rho}(n+k_\lambda)|.
\end{align}
Using
\begin{align*}
d(m,Q)\lesssim_{\varepsilon} m^\varepsilon,
\end{align*} 
\eqref{bound on fourier transform of deltas} and \eqref{partition on circle}, 
we have 
\begin{align}\label{hatrhon}
|\hat{\rho}(n)|\leq\sum_{1\leq Q\leq N_1}\sum_{Q\leq M\leq N}\frac{d(\frac{Tn}{2\pi},Q)Q^{1+\varepsilon}}{NM}\lesssim 
\frac{N^\varepsilon}{N}, \ \ \text{for }n\neq 0, |n|\lesssim N^2,
\end{align}
thus 
\begin{align*}
|\lambda_{Q,M}(n,\lambda)|&\lesssim_\varepsilon \varphi(\lambda, N)\frac{Q}{NM}\left[Q^\varepsilon d(\frac{T(n+k_\lambda)}{2\pi},Q)+\frac{Q}{N^{1-\varepsilon}}\right]\\
&\lesssim_{\varepsilon} \varphi(\lambda,N)\frac{QN^\varepsilon}{NM}, \ \ \text{for } |n|\lesssim N^2.\numberthis \label{bound on fourier transform of Lambda}
\end{align*}

\begin{prop}\label{EstimatesOnArcs}
(i) Assume that $f\in L^1(\mathbb{T}\times G)$. Then 
\begin{align}\label{majorarcLinfty}
\|f*\Lambda_{Q,M}\|_{L^\infty(\mathbb{T}\times G)}
\lesssim N^{d-\frac{r}{2}}(\frac{M}{Q})^{r/2}\|f\|_{L^1(\mathbb{T}\times G)}.
\end{align}
(ii) Assume that $f\in L^2(\mathbb{T}\times G)$. Assume also
\begin{equation}\label{hatfnlambda}
\hat{f}(n,\lambda)=0,\ \  \text{for } |n|\gtrsim N^2.
\end{equation} 
Then 
\begin{align}\label{majorarcL2}
\|f*\Lambda_{Q,M}\|_{L^2(\mathbb{T}\times G)}
\lesssim_\varepsilon \frac{QN^\varepsilon}{NM}\|f\|_{L^2(\mathbb{T}\times G)},
\end{align}
\begin{align}\label{majorarcL2epsilon}
\|f*\Lambda_{Q,M}\|_{L^2(\mathbb{T}\times G)}\lesssim_{\tau,B}\frac{Q^{1+2\tau}L}{NM}\|f\|_{L^2(\mathbb{T}\times G)}
+M^{-1}L^{-B/2}N^{d/2}\|f\|_{L^1(\mathbb{T}\times G)}.
\end{align}
for all 
\begin{align}\label{condition for Q}
L>1, \ \ 0<\tau<1, \ \ B>\frac{6}{\tau},\ \  N>(LQ)^B.
\end{align}
\end{prop}
\begin{proof}
Using \eqref{majorarc}, we have
\begin{align*}
\|f*\Lambda_{Q,M}\|_{L^\infty(\mathbb{T}\times G)}
\leq\|f\|_{L^1(\mathbb{T}\times G)}\|\Lambda_{Q,M}\|_{L^\infty(\mathbb{T}\times G)}\lesssim
 N^{d-\frac{r}{2}}(\frac{M}{Q})^{r/2}\|f\|_{L^1(\mathbb{T}\times G)}.
\end{align*}
This proves (i). \eqref{majorarcL2} is a consequence of  \eqref{convolutionL2}, \eqref{Lambda and lambda},   and \eqref{bound on fourier transform of Lambda}. To prove \eqref{majorarcL2epsilon}, we use  \eqref{Plancherel}, \eqref{f*ghat} and \eqref{Lambda and lambda} to get
\begin{align*}
\|f*\Lambda_{Q,M}\|_{L^2(\mathbb{T}\times G)}
&=\left(\sum_{n,\lambda}d_\lambda \|\hat{f}(n,\lambda)\|_{HS}^2\cdot|\lambda_{Q,M}(n,\lambda)|^2\right)^{1/2},
\end{align*}
which combined with \eqref{bound on fourier transform of Lambda2}, \eqref{lambdaQMbound}, and \eqref{hatrhon} yields 
\begin{align*}
\|f*\Lambda_{Q,M}\|_{L^2(\mathbb{T}\times G)}\lesssim_{\varepsilon}
&\frac{Q^{1+\varepsilon}}{NM}\left(\sum_{n,\lambda}
\varphi(\lambda, N)^2d_\lambda\|\hat{f}(n,\lambda)\|^2_{HS}d(\frac{T(n+k_\lambda)}{2\pi},Q)^2\right)^{1/2}\\
&+
\frac{Q^2}{MN^{2-\varepsilon}}\|f\|_{L^2(\mathbb{T}\times G)}.\numberthis
\end{align*}
Using Lemma 3.47 in \cite{Bou93} and Lemma \ref{CountingLemma}, we have 
\begin{align*}
&\left|\{(n,\lambda)\mid |n|, k_\lambda\lesssim N^2, d(\frac{T(n+k_\lambda)}{2\pi},Q)>D\}\right|\\
&\lesssim_{\tau, B} (D^{-B}Q^{\tau}N^2+Q^B)\cdot 
\max_{|m|\lesssim N^2}|\{(n,\lambda)\mid n+k_\lambda=m\}|\\
&\lesssim_{\tau, B} (D^{-B}Q^{\tau}N^2+Q^B)\cdot 
|\{\lambda\in\hat{G}\mid k_\lambda\lesssim N^2\}|\\
&\lesssim_{\tau, B} (D^{-B}Q^{\tau}N^2+Q^B)\cdot 
N^r.
\numberthis \label{d(n+klambda)>D}
\end{align*}

Now \eqref{Hausdorff-Young} gives
\begin{align*}
\|\hat{f}(n,\lambda)\|^2_{HS}\lesssim d_\lambda\|f\|^2_{L^1(\mathbb{T}\times G)},
\end{align*}
and Lemma \ref{CountingLemma} gives 
\begin{align*}
|\varphi(\lambda, N)d^2_\lambda|\lesssim N^{d-r},
\end{align*}
which together with \eqref{d(n+klambda)>D} 
imply
\begin{align*}
\|f*\Lambda_{Q,M}\|_{L^2(\mathbb{T}\times G)}\lesssim_{\tau, B}&(\frac{Q^{1+\varepsilon}D}{NM}+\frac
{Q^2}{MN^{2-\varepsilon}})\|f\|_{L^2(\mathbb{T}\times G)}\\
&+\frac{Q^{1+\varepsilon}}{NM}\cdot Q\cdot(D^{-B/2}Q^{\tau}N+Q^{B/2})N^{d/2}\|f\|_{L^1(\mathbb{T}\times G)}.
\numberthis
\end{align*}
This implies \eqref{majorarcL2epsilon} assuming the conditions in \eqref{condition for Q}. 
\end{proof}

Now interpolating \eqref{majorarcLinfty} and \eqref{majorarcL2}, we get
\begin{align}\label{majorarcLp}
\|f*\Lambda_{Q,M}\|_{L^p(\mathbb{T}\times G)}\lesssim_\varepsilon N^{d-\frac{r}{2}-\frac{2d-r+2}{p}+\varepsilon}M^{\frac{r}{2}-\frac{r+2}{p}}
Q^{-\frac{r}{2}+\frac{r+2}{p}}\|f\|_{L^{p'}(\mathbb{T}\times G)}.
\end{align}
Interpolating \eqref{majorarcLinfty} and \eqref{majorarcL2epsilon} for 
\begin{align}\label{condition for p tau}
p>\frac{2(r+2)}{r}+10\tau,\ \ \text{which implies } \sigma=\frac{r}{2}-\frac{r+2+4\tau}{p}>0,
\end{align}
we get 
\begin{align*}
\|f*\Lambda_{Q,M}\|_{L^p(\mathbb{T}\times G)}\lesssim_{\tau, B} & N^{d-\frac{r}{2}-\frac{2d-r+2}{p}}M^{\frac{r}{2}-\frac{r+2}{p}}
Q^{-\sigma}L^{\frac{2}{p}}\|f\|_{L^{p'}(\mathbb{T}\times G)}\\
&+Q^{-\frac{2}{r}(1-\frac{2}{p})}M^{\frac{r}{2}-\frac{r+2}{p}}L^{-\frac{B}{p}}N^{d-\frac{r}{2}-\frac{d-r}{p}}\|f\|_{L^1(\mathbb{T}\times G)}.\numberthis \label{majorarcLpepsilon}
\end{align*}

Now we are ready to prove Theorem \ref{simpleconnected}. 
\begin{proof}[Proof of Theorem \ref{simpleconnected}]
Without loss of generality, we assume that $\|f\|_{L^2(G)}=1$. Then for $F={\bf P}_Ne^{it\Delta}f$,  \eqref{Bernstein} implies 
\begin{align}\label{FL2}
\|F\|_{L^2_x}\lesssim 1,
\end{align}
\begin{align}\label{FLinfty}
\|F\|_{L^\infty_x}\lesssim N^{\frac{d}{2}}.
\end{align}  
Let 
\begin{align}\label{DefinitionH}
H=\chi_{|F|>\lambda}\cdot\frac{F}{|F|}.
\end{align} 
Let $\tilde{\tilde{{\bf P}}}_N$ be a Littlewood-Paley projection of the product type such that $\tilde{\tilde{{\bf P}}}_N\circ{\bf P}_N={\bf P}_N$. Let $\tilde{\tilde{{\bf K}}}_N$ be the Schr\"odinger kernel associated to $\tilde{\tilde{{\bf P}}}_Ne^{it\Delta}$. Then by \eqref{f*ghat}, \eqref{STFTKN}, and \eqref{STFTPNf}, we have 
\begin{align*}
F*\tilde{\tilde{{\bf K}}}_N=F.
\end{align*}
Let $Q_{N^2}$ be the Littlewood-Paley projection operator on $L^2(\mathbb{T}\times G)$ defined by 
\begin{align*}
(Q_{N^2}H)^{\wedge}:=\varphi(\frac{-k_\lambda-n^2}{N^4})\widehat{H}(n,\lambda)
\end{align*}
for some bump function $\varphi$ such that $Q_{N^2}\circ{\bf P}_N={\bf P}_N$. Then by \eqref{PlancherelInnerproduct} and \eqref{STFTPNf}, we have 
\begin{align*}
\langle F, H\rangle_{L^2_{t,x}}=\langle Q_{N^2}F, H\rangle_{L^2_{t,x}}=\langle F, Q_{N^2}H\rangle_{L^2_{t,x}}.
\end{align*}
Then we can write 
\begin{align*}
\lambda m_\lambda\leq\langle F,H\rangle_{L^2_{t,x}}=
\langle F*\tilde{\tilde{{\bf K}}}_N, Q_{N^2}H\rangle_{L^2_{t,x}}.
\end{align*}
Using \eqref{Plancherel} and \eqref{f*ghat} again, we get 
\begin{align*}
\lambda m_\lambda&\leq \langle F, Q_{N^2}H*\tilde{\tilde{{\bf K}}}_N\rangle_{L^2_{t,x}}
\leq\|F\|_{L^2_{t,x}}\|Q_{N^2}H*\tilde{\tilde{{\bf K}}}_N\|_{L^2_{t,x}}\\
&\lesssim \|Q_{N^2}H*\tilde{\tilde{{\bf K}}}_N\|_{L^2_{t,x}}=\langle Q_{N^2}H*\tilde{\tilde{{\bf K}}}_N, Q_{N^2}H*\tilde{\tilde{{\bf K}}}_N\rangle_{L^2_{t,x}}
=\langle Q_{N^2}H, Q_{N^2}H*(\tilde{\tilde{{\bf K}}}_N*\tilde{\tilde{{\bf K}}}_N)\rangle_{L^2_{t,x}}. \numberthis \label{lambdamlambda}
\end{align*}
Let 
\begin{align*}
H'=Q_{N^2}H, \ \ \tilde{{\bf K}}_N=\tilde{\tilde{{\bf K}}}_N*\tilde{\tilde{{\bf K}}}_N.
\end{align*}
Note that $H'$ by definition satisfies the assumption in  \eqref{hatfnlambda} and we can apply Proposition \ref{EstimatesOnArcs}. Also note that $\tilde{{\bf K}}_N$ is still a Schr\"odinger kernel associated to a Littlewood-Paley projection operator of the product type. Finally note that the Bernstein type inequalities  \eqref{Bernstein} and the definition \eqref{DefinitionH} of $H$ give
\begin{align}\label{H'bounds}
\|H'\|_{L^p_{t,x}}\lesssim \|H\|_{L^p_{t,x}}\lesssim m_\lambda^{\frac{1}{p}}. 
\end{align}
Write
\begin{align*}
\Lambda=\sum_{1\leq Q\leq N_1}\sum_{Q\leq M\leq N}\Lambda_{Q,M}, \ \ \tilde{{\bf K}}_N=\Lambda+(\tilde{{\bf K}}_{N}-\Lambda), 
\end{align*}
where $\Lambda_{Q,M}$ is defined as in \eqref{introduce LambdaQM} except that ${\bf K}_N$ is replaced by $\tilde{{\bf K}}_N$. We have by \eqref{lambdamlambda}
\begin{align*}
\lambda^2 m^2_\lambda
&\lesssim \langle H', H'*\Lambda\rangle_{L^2_{t,x}}+\langle H',H'*(\tilde{{\bf K}}_N-\Lambda)\rangle_{L^2_{t,x}}\\
&\lesssim  \|H'\|_{L^{p'}_{t,x}}\|H'*\Lambda\|_{L^p_{t,x}}+\|H'\|_{L^1_{t,x}}^2\|\tilde{{\bf K}}_N-\Lambda\|_{L^\infty_{t,x}}.\numberthis\label{distribution estimate}
\end{align*}
Using \eqref{majorarcLp} for $p=p_0:=\frac{2(r+2)}{r}$, then summing over $Q,M$, and noting \eqref{H'bounds}, we have 
\begin{align*} 
\|H'\|_{L^{p'}_{t,x}}\|H'*\Lambda\|_{L^p_{t,x}}\lesssim 
N^{d-\frac{2d+4}{p_0}+\varepsilon}\|H'\|^2_{L^{p_0'}_{t,x}}\lesssim N^{d-\frac{2d+4}{p_0}+\varepsilon}m_\lambda^{\frac{2}{p_0'}}.
\end{align*}
From \eqref{minorarcLinfity} and \eqref{alphaQMbound} we get
\begin{align}\label{minorarcLinfityfull}
\|\tilde{{\bf K}}_N-\Lambda\|_{L^\infty_{t,x}}\lesssim N^{d-\frac{r}{2}},
\end{align}
which implies 
\begin{align}\label{H'KN-Lambda}
\|H'\|_{L^1_{t,x}}^2\|\tilde{{\bf K}}_N-\Lambda\|_{L^\infty_{t,x}}\lesssim
N^{d-\frac{r}{2}}\|H'\|_{L^1_{t,x}}^2\lesssim N^{d-\frac{r}{2}}m_\lambda^2.
\end{align}
Then we have
\begin{align*}
\lambda^2m_\lambda^2\lesssim N^{d-\frac{2d+4}{p_0}+\varepsilon}m_\lambda^{\frac{2}{p_0'}}
+N^{d-\frac{r}{2}}m_\lambda^2,
\end{align*}
which implies for $\lambda\gtrsim N^{\frac{d}{2}-\frac{r}{4}}$
\begin{align*}
m_\lambda\lesssim_\varepsilon N^{p_0(\frac{d}{2}-\frac{d+2}{p_0})+\varepsilon}\lambda^{-p_0}.
\end{align*}
Thus part (I) is proved. To prove part (II) for some fixed $p$, using part (I) and \eqref{FLinfty}, it suffices to prove it for $\lambda\gtrsim N^{\frac{d}{2}-\varepsilon}$. Summing \eqref{majorarcLpepsilon} over $Q,M$ in the range indicated by \eqref{condition for Q}, we get
\begin{align}\label{H'Lambda1}
\|H'*\Lambda_1\|_{L^p_{t,x}}\lesssim L N^{d-\frac{2d+4}{p}}\|H'\|_{L^{p'}_{t,x}}+L^{-B/p}N^{d-\frac{d+2}{p}}\|H'\|_{L^1_{t,x}},
\end{align}
where
\begin{align*}
\Lambda_1:=\sum_{Q<Q_1,Q\leq M\leq N}\Lambda_{Q,M}
\end{align*}
and $Q_1$ is the largest $Q$-value satisfying \eqref{condition for Q}.
For values $Q\geq Q_1$, use \eqref{majorarcLp} to get
\begin{align}\label{H'Lambda-Lambda1}
\|H'*(\Lambda-\Lambda_1)\|_{L^p_{t,x}}\lesssim_{\varepsilon}  N^{d-\frac{2d+4}{p}+\varepsilon}Q_1^{-(\frac{r}{2}-\frac{r+2}{p})}\|H'\|_{L^{p'}_{t,x}}.
\end{align}
Using \eqref{distribution estimate}, \eqref{H'KN-Lambda}, \eqref{H'Lambda1} and \eqref{H'Lambda-Lambda1}, we get
\begin{align*}
\lambda^2m_\lambda^2\lesssim N^{d-\frac{2(d+2)}{p}}(L+\frac{N^\varepsilon}
{Q_1^{\frac{r}{2}-\frac{r+2}{p}}})m_\lambda^{2/p'}+
L^{-B/p}N^{d-\frac{d+2}{p}}m_\lambda^{1+\frac{1}{p'}}+N^{d-\frac{r}{2}}m_\lambda^2.
\end{align*}
For $\lambda\gtrsim N^{\frac{d}{2}-\frac{r}{4}}$, the last term of the above inequality can be dropped. Let $Q_1=N^\delta$ such that $\delta>0$ and  
\begin{align}\label{Condition for delta}
(LN^{\delta})^B<N
\end{align}
such that \eqref{condition for Q} holds. 
Note that  
\begin{align*}
L>1>\frac{N^\varepsilon}{Q_1^{\frac{r}{2}-\frac{r+2}{p}}}
\end{align*}
for $p>p_0+10\tau$ and $\varepsilon$ sufficiently small, thus
\begin{align*}
\lambda^2m_\lambda^2\lesssim N^{d-\frac{2(d+2)}{p}}Lm_\lambda^{2/p'}+
L^{-B/p}N^{d-\frac{d+2}{p}}m_\lambda^{1+\frac{1}{p'}}.
\end{align*}
This implies
\begin{align*}
m_\lambda&\lesssim N^{p(\frac{d}{2}-\frac{d+2}{p})}L^{\frac{p}{2}}\lambda^{-p}+N^{p(d-\frac{d+2}{p})}L^{-B}\lambda^{-2p}\\
&\lesssim N^{-d-2}(\frac{N^{d/2}}{\lambda})^{p}L^{\frac{p}{2}}+N^{-d-2}(\frac{N^{d/2}}{\lambda})^{2p}L^{-B}.
\end{align*}
Let 
\begin{align*}
L=(\frac{N^{d/2}}{\lambda})^{\tau}, \ \ B>\frac{p}{\tau}
\end{align*}
and $\delta$ be sufficiently small so that \eqref{Condition for delta} holds, then  
\begin{align*}
m_{\lambda}\lesssim N^{-d-2}(\frac{N^{d/2}}{\lambda})^{p+\frac{p\tau}{2}}.
\end{align*} 
Note that conditions for $p,\tau$ indicated in \eqref{condition for p tau} implies that $p+\frac{p\tau}{2}$ can take any exponent $>p_0=\frac{2(r+2)}{r}$. This completes the proof of part (II). 

The proofs of parts (III) and (IV) are then identical to the proofs of Propositions 3.110 and 3.113 respectively in \cite{Bou93}. 
\end{proof}

\begin{proof}[Proof of Theorem \ref{Main}]
Part (i) is a direct consequence of Theorem \ref{simpleconnected}(III). Part (ii) is a direct consequence of Theorem \ref{simpleconnected}(IV) and the result from \cite{BD15} that full Strichartz estimates hold on any torus with an $\varepsilon$-loss. 
\end{proof}

\section{Dispersive Estimates on Major Arcs}
\label{Dispersive Estimates on Major Arcs}

In this section, we prove Theorem \ref{MainEstimate}. 

\subsection{Weyl Type Sums on Rational Lattices}
\label{Weyl Type Sums and Weyl Differencing for Rational Lattices}

\begin{defn}\label{definitionofarationallattice}
Let $L=\mathbb{Z}w_1+\cdots+\mathbb{Z}w_r$ be a lattice on an inner product space $(V, \langle\ , \ \rangle)$. We say $L$ is a \textit{rational lattice} provided that there exists some $D\in\mathbb{R}$ such that $\langle w_i,w_j\rangle\in D^{-1}\mathbb{Z}$. We call the number $D$ a \textit{period} of $L$. 
\end{defn}

By Lemma \ref{DGamma}, any weight lattice $\Lambda$ is a rational lattice with respect to the Cartan-Killing form. As a sublattice of $\Lambda$, the root lattice $\Gamma$ is also rational. 

Let $f$ be a function on $\mathbb{Z}^r$ and define the \textit{difference operator} $D_i$ by 
\begin{equation}\label{DefinitionDi}
D_{i}f(n_1,\cdots,n_r):=f(n_1,\cdots,n_{i-1}, n_i+1, n_{i+1},\cdots, n_r)-f(n_1,\cdots, n_r)
\end{equation} 
for $i=1,\cdots, r$. The Leibniz rule for $D_i$ reads
\begin{align}\label{Leibniz}
D_i(\prod_{j=1}^nf_j)=
\sum_{l=1}^n\sum_{1\leq k_1<\cdots<k_l\leq n}
D_if_{k_1}\cdots D_if_{k_l}\cdot \prod_{\substack{j \neq k_1,\cdots, k_l\\ 1\leq j\leq n}}f_j.
\end{align}
Note that there are $2^n-1$ terms in the right side of the above formula. 

\begin{defn}
Let $L\cong\mathbb{Z}^r$ be a lattice of rank $r$. Given $A\in\mathbb{R}$, we say a function $f$ on $L$ is a pseudo-polynomial of degree $A$ provided for each $n\in\mathbb{Z}_{\geq0}$, 
\begin{align}\label{DecCon}
|D_{i_1}\cdots D_{i_n}f(n_1,\cdots,n_r)|\lesssim N^{A-n}
\end{align}
holds uniformly in $|n_i|\lesssim N$, $i=1,\cdots, r$, for all $i_j=1,\cdots, r$, $j=1,\cdots, n$, and $N\geq 1$. 
\end{defn}

A direct application of the Leibniz rule \eqref{Leibniz} gives the following lemma. 
\begin{lem}\label{pseudoleibniz}
Let $L$ be a lattice and $f,g$ two functions on $L$. Assume $f,g$ are  pseudo-polynomials of degrees $A,B$ respectively. Then $f\cdot g$ is a pseudo-polynomial of degree $A+B$. 
\end{lem}

Now we have the following estimate on Weyl type sums, which generalizes the classical Weyl inequality in one dimension, as in Lemma 3.18 of \cite{Bou93}. 

\begin{lem}\label{WeylSum}
Let $L=\mathbb{Z}w_1+\cdots+\mathbb{Z}w_r$ be a rational lattice in the inner product space $(V,\langle\ ,\  \rangle)$ with a period $D>0$. Let $\varphi$ be a bump function on $\mathbb{R}$ and $N\geq 1$, $A\in\mathbb{R}$. Suppose $f:L\to\mathbb{C}$ a pseudo-polynomial of degree $A$.  
Let
\begin{align}\label{absorb}
F(t,H)=\sum_{\lambda\in L}e^{-it|\lambda|^2+i\langle\lambda, H\rangle}\varphi(\frac{|\lambda|^2}{N^2})\cdot f
\end{align}
for $t\in\mathbb{R}$ and $H\in V$. Then for $\frac{t}{2\pi D}\in\mathcal{M}_{a,q}$, we have
\begin{align}\label{appendix}
|F(t,H)|\lesssim \frac{N^{A+r}}{(\sqrt{q}(1+N\|\frac{t}{2\pi D}-\frac{a}{q}\|^{1/2}))^{r}}
\end{align}
uniformly in $H\in V$.  
\end{lem}

Note that Part (i) of Theorem \ref{MainEstimate} is a direct consequence of this lemma. 

\begin{proof}
By the Weyl differencing trick, write 
\begin{align*}
|F|^2&=\sum_{\lambda_1,\lambda_2\in L}e^{-it(|\lambda_1|^2-|\lambda_2|^2)+i\langle\lambda_1-\lambda_2,H\rangle}
\varphi(\frac{|\lambda_1|^2}{N^2})\varphi(\frac{|\lambda_2|^2}{N^2})f(\lambda_1)
\overline{f(\lambda_2)}\\
&=\sum_{\mu=\lambda_1-\lambda_2}
e^{-it|\mu|^2+i\langle\mu,H\rangle}
\sum_{\lambda=\lambda_2}e^{-i2t\langle\mu,\lambda\rangle}
\varphi(\frac{|\mu+\lambda|^2}{N^2})\varphi(\frac{|\lambda|^2}{N^2})f(\mu+\lambda)
\overline{f(\lambda)}\\
&\leq\sum_{|\mu|\lesssim N}|\sum_{\lambda}e^{-i2t\langle\mu,\lambda\rangle}
\varphi(\frac{|\mu+\lambda|^2}{N^2})\varphi(\frac{|\lambda|^2}{N^2})f(\mu+\lambda)
\overline{f(\lambda)}|.
\end{align*}
Now let $L=\mathbb{Z}w_1+\cdots+\mathbb{Z}w_r$. Write 
\begin{equation*}
\lambda=\sum_{i=1}^r n_iw_i,
\end{equation*} 
and 
\begin{equation*}
g(\lambda)=\varphi(\frac{|\mu+\lambda|^2}{N^2})\varphi(\frac{|\lambda|^2}{N^2})
f(\mu+\lambda)\overline{f(\lambda)}. 
\end{equation*}
Note that as functions in $\lambda\in L$, both $\varphi(\frac{|\mu+\lambda|^2}{N^2})|$ and $\varphi(\frac{|\lambda|^2}{N^2})$ are pseudo-polynomials of degree 0, and both $f(\mu+\lambda)$ and $f(\lambda)$ are pseudo-polynomials of degree $A$, which implies by Lemma \ref{pseudoleibniz} that $g(\lambda)$ is a pseudo-polynomial of degree $2A$. That is, $g(\lambda)$ satisfies
\begin{align}\label{Di1Ding}
|D_{i_1}\cdots D_{i_n}g(\lambda)|\lesssim N^{2A-n}.
\end{align}
uniformly for $|\lambda|\lesssim N$ and $N\geq 1$, for all $i_1,\cdots, i_n\in \{1,\cdots, r\}$. 
Write
\begin{align}\label{WeylInduction}
\sum_{\lambda\in L}e^{-i2t\langle\mu,\lambda\rangle}g(\lambda)
&=
\sum_{n_1,\cdots,n_r\in\mathbb{Z}}(\prod_{i=1}^re^{-itn_i\langle\mu,2w_i\rangle})g(\lambda).
\end{align}
By summation by parts twice, we have
\begin{align}\label{partstwice}
\sum_{n_1\in\mathbb{Z}}e^{-itn_1\langle\mu,2w_1\rangle}g
=(\frac{e^{-it\langle\mu,2w_1\rangle}}{1-e^{-it\langle\mu,2w_1
\rangle}})^2\sum_{n_1\in\mathbb{Z}}e^{-itn_1\langle\mu,2w_1\rangle}D^2_{1}g(n_1,\cdots,n_r),
\end{align}
then \eqref{WeylInduction} becomes 
\begin{align*}
\sum_{\lambda\in L}e^{-i2t\langle\mu,\lambda\rangle}g
= (\frac{e^{-it\langle\mu,2w_1\rangle}}{1-e^{-it\langle\mu,2w_1
\rangle}})^2\sum_{n_1,\cdots,n_r\in\mathbb{Z}}(\prod_{i=1}^re^{-itn_i\langle\mu,2w_i\rangle})D_1^2g(n_1,\cdots,n_r).
\end{align*}
Then we can carry out the procedure of summation by parts twice with respect to other variables $n_2,\cdots,n_r$. But we require that only when 
$$|1-e^{-it\langle\mu,2w_i\rangle}|\geq\frac{1}{N}$$ do we carry out the procedure to the variable $n_i$. Using \eqref{Di1Ding}, 
 we obtain
\begin{align*}
&|\sum_{\lambda}e^{-i2t\langle\mu,\lambda\rangle}
\varphi(\frac{|\mu+\lambda|^2}{N^2})\varphi(\frac{|\lambda|^2}{N^2})f(\mu+\lambda)
\overline{f(\lambda)}|\\
&\lesssim N^{2A-r}\prod_{i=1}^r
\frac{1}{(\max\{1-e^{-it\langle\mu,2w_i\rangle},\frac{1}{N}\})^2}\\
&\lesssim N^{2A-r}\prod_{i=1}^r
\frac{1}{(\max\{\|\frac{1}{2\pi}t\langle\mu,2w_i\rangle\|,\frac{1}{N}\})^2}.
\end{align*}
Writing $\mu=\sum_{j=1}^rm_jw_j$, $m_j\in\mathbb{Z}$, we have
\begin{align*}
|F|^2&\lesssim N^{2A-r}\sum_{\substack{|m_j|\lesssim N, \\ j=1,\cdots, r}}
\prod_{i=1}^r\frac{1}{(\max\{\|\frac{1}{2\pi}t\sum_{j=1}^r
m_j\langle w_j,2w_i\rangle\|,\frac{1}{N}\})^2}.
\end{align*}
Let 
\begin{align}\label{nimi}
n_i=\sum_{j=1}^rm_j\langle w_j,2w_i\rangle\cdot D,\ i=1,\cdots, r,
\end{align} 
where $D>0$ is the period of $L$ so that $\langle w_j,w_i\rangle\in D^{-1}{\mathbb{Z}}$. Then $n_i\in \mathbb{Z}$. Note that the matrix $\left(\langle w_j,2w_i\rangle D\right)_{i,j}$ is non-degenerate, which implies that for each vector $(n_1,\cdots, n_r)\in \mathbb{Z}^r$, there exists at most one vector $(m_1,\cdots, m_r)\in \mathbb{Z}^r$ so that \eqref{nimi} holds, thus 
\begin{align*}
|F|^2&\lesssim 
N^{2A-r}\sum_{\substack{|n_i|\lesssim N, \\ i=1,\cdots, r}}\prod_{i=1}^r\frac{1}{(\max\{\|\frac{t}{2\pi D}n_i\|,\frac{1}{N}\})^2}\\
&\lesssim N^{2A-r}\prod_{i=1}^r\left(\sum_{|n_i|\lesssim N}\frac{1}{(\max\{\|\frac{t}{2\pi D}n_i\|,\frac{1}{N}\})^2}\right).
\end{align*}
Then by a standard estimate as in the proof of the classical Weyl inequality in one dimension, we have
\begin{align*}
\sum_{|n_i|\lesssim N}\frac{1}{(\max\{\|\frac{t}{2\pi D}n_i\|,\frac{1}{N}\})^2}
\lesssim \frac{N^3}{(\sqrt{q}(1+N\|\frac{t}{2\pi D}-\frac{a}{q}\|^{1/2}))^2},
\end{align*}
which implies the desired result 
\begin{align*}
|F|^2\lesssim \frac{N^{2A+2r}}{(\sqrt{q}(1+N\|\frac{t}{2\pi D}-\frac{a}{q}\|^{1/2}))^{2r}}.
\end{align*}
\end{proof}

\begin{rem}\label{WeylSumRemark}
Let $\lambda_0$ be a constant vector in $\mathbb{R}^r$ and $C$ a constant real number. Then we can slightly generalize the form of the function $F(t,H)$ in the above lemma into
\begin{equation*}
F(t,H)=\sum_{\lambda\in L}e^{-it|\lambda+\lambda_0|^2+i\langle\lambda, H\rangle}\varphi(\frac{|\lambda+\lambda_0|^2+C}{N^2})\cdot f
\end{equation*}
such that the conclusion of the lemma still holds. 
\end{rem}

\subsection{From a Chamber to the Whole Weight Lattice}
\label{First Reduction} 
To prove Part (ii) of Theorem \ref{MainEstimate}, we first rewrite the Schr\"odinger kernel as an exponential sum over the whole weight lattice $\Lambda$ instead of just a chamber of it, in order to apply Lemma \ref{WeylSum}. 

\begin{lem}
Recall that $D_P(H)=\sum_{s\in W}\det s\ e^{i\langle\rho, H\rangle}$ is the Weyl denominator. We have
\begin{align}
K_N(t,x)&=\frac{e^{it|\rho|^2}}{(\prod_{\alpha\in P}\langle\alpha,\rho\rangle)D_P(H)}\sum_{\lambda\in \Lambda}e^{-it|\lambda|^2+i\langle \lambda,H\rangle}\varphi(\frac{|\lambda|^2-|\rho|^2}{N^2})\prod_{\alpha\in P}\langle \alpha,\lambda\rangle\label{Schrodinger kernel1}\\
&=\frac{e^{it|\rho|^2}}{(\prod_{\alpha\in P}\langle\alpha,\rho\rangle)|W|}\sum_{\lambda\in \Lambda}e^{-it|\lambda|^2}\varphi(\frac{|\lambda|^2-|\rho|^2}{N^2})\prod_{\alpha\in P}\langle \alpha,\lambda\rangle\frac{\sum_{s\in W}\det{(s)}e^{i\langle s(\lambda),H\rangle}}{\sum_{s\in W}\det{(s)}e^{i\langle s(\rho), H\rangle}}\label{Schrodinger kernel2}
\end{align}

\end{lem}
\begin{proof}
To prove \eqref{Schrodinger kernel2}, first note that from Proposition \ref{Anti-invariant by invariant} below, $\prod_{\alpha\in P}\langle\alpha,\cdot\rangle$ is an \textit{anti-invariant polynomial}, that is, 
\begin{align}\label{Weyl group dimension}
\prod_{\alpha\in P}\langle\alpha,s(\lambda)\rangle=\det(s)\prod_{\alpha\in P}\langle\alpha,\lambda\rangle, 
\end{align}
for all $\lambda\in i\mathfrak{b}^*$. Recall that the Weyl group $W$ acts on $i\mathfrak{b}^*$ isometrically, that is, 
\begin{align}\label{Weyl group isometry}
|s(\lambda)|=|\lambda|, \ \ \text{for all } s\in W, \ \ \lambda\in i\mathfrak{b}^*.
\end{align} 
Also recalling the definition \eqref{definition of rho} of $\rho$ and the definition \eqref{fundamentalchamber} of the fundamental chamber $C$, we may rewrite $K_N$ as in \eqref{SchrodingerKernel} into  
\begin{align*}
K_N(t,x)=\frac{e^{it|\rho|^2}}{(\prod_{\alpha\in P}\langle\alpha,\rho\rangle)D_P}\sum_{\lambda\in \Lambda\cap C}e^{-it|\lambda|^2}\varphi(\frac{|\lambda|^2-|\rho|^2}{N^2})
\prod_{\alpha\in P}\langle\alpha,\lambda\rangle\sum_{s\in W}\det(s)e^{i\langle s(\lambda),H\rangle}
\end{align*}
Using the \eqref{Weyl group dimension} and \eqref{Weyl group isometry}, we write 
\begin{align*} 
K_N(t,x)&=\frac{e^{it|\rho|^2}}{(\prod_{\alpha\in P}\langle\alpha,\rho\rangle)D_P}\sum_{s\in W}\sum_{\lambda\in \Lambda\cap C}e^{-it|\lambda|^2}\varphi(\frac{|\lambda|^2-|\rho|^2}{N^2})
\prod_{\alpha\in P}\langle\alpha,s(\lambda)\rangle e^{i\langle s(\lambda),H\rangle}\\
&=\frac{e^{it|\rho|^2}}{(\prod_{\alpha\in P}\langle\alpha,\rho\rangle)D_P}\sum_{s\in W}\sum_{\lambda\in \Lambda\cap C}e^{-it|s(\lambda)|^2}\varphi(\frac{|s(\lambda)|^2-|\rho|^2}{N^2})
\prod_{\alpha\in P}\langle\alpha,s(\lambda)\rangle e^{i\langle s(\lambda),H\rangle}\\
&=\frac{e^{it|\rho|^2}}{(\prod_{\alpha\in P}\langle\alpha,\rho\rangle)D_P}\sum_{\lambda\in \sqcup_{s\in W}s(\Lambda\cap C)}e^{-it|\lambda|^2}\varphi(\frac{|\lambda|^2-|\rho|^2}{N^2})
\prod_{\alpha\in P}\langle\alpha,\lambda\rangle e^{i\langle \lambda,H\rangle},\numberthis
\end{align*}
which then implies by \eqref{chamberdecomposition} that 
\begin{align*}
K_N(t,x)=\frac{e^{it|\rho|^2}}{(\prod_{\alpha\in P}\langle\alpha,\rho\rangle)D_P}\sum_{\lambda\in \Lambda}e^{-it|\lambda|^2}\varphi(\frac{|\lambda|^2-|\rho|^2}{N^2})
\prod_{\alpha\in P}\langle\alpha,\lambda\rangle e^{i\langle \lambda,H\rangle}.
\end{align*}
This proves \eqref{Schrodinger kernel1}. To prove  \eqref{Schrodinger kernel2}, write 
\begin{align}\label{KerExp}
\sum_{\lambda\in \Lambda}e^{-it|\lambda|^2+i\langle\lambda, H\rangle}
\varphi(\frac{|\lambda|^2-|\rho|^2}{N^2})\prod_{\alpha\in P}\langle\alpha,\lambda\rangle
=\sum_{\lambda\in \Lambda}e^{-it|s(\lambda)|^2+i\langle s(\lambda), H\rangle}
\varphi(\frac{|s(\lambda)|^2-|\rho|^2}{N^2})\prod_{\alpha\in P}\langle\alpha,s(\lambda)\rangle,
\end{align}
which implies using \eqref{Weyl group dimension} and \eqref{Weyl group isometry} that
\begin{equation*}
\sum_{\lambda\in \Lambda}e^{-it|\lambda|^2+i\langle\lambda, H\rangle}
\varphi(\frac{|\lambda|^2-|\rho|^2}{N^2})\prod_{\alpha\in P}\langle\alpha,\lambda\rangle
=\det(s)\sum_{\lambda\in \Lambda}e^{-it|\lambda|^2+i\langle s(\lambda), H\rangle}
\varphi(\frac{|\lambda|^2-|\rho|^2}{N^2})\prod_{\alpha\in P}\langle\alpha,\lambda\rangle,
\end{equation*}
which further implies 
\begin{align*}
\sum_{\lambda\in \Lambda}e^{-it|\lambda|^2+i\langle\lambda, H\rangle}
\varphi(\frac{|\lambda|^2-|\rho|^2}{N^2})\prod_{\alpha\in P}\langle\alpha,\lambda\rangle
=\frac{1}{|W|}\sum_{\lambda\in \Lambda}e^{-it|\lambda|^2}
\varphi(\frac{|\lambda|^2-|\rho|^2}{N^2})\prod_{\alpha\in P}\langle\alpha,\lambda\rangle\sum_{s\in W}\det(s)e^{i\langle s(\lambda), H\rangle}.
\end{align*}
This combined with \eqref{Schrodinger kernel1} yields  \eqref{Schrodinger kernel2}. 
\end{proof}

\begin{exmp}
Specializing \eqref{Schrodinger kernel1} and  \eqref{Schrodinger kernel2} to the Schr\"{o}dinger kernel \eqref{kernel for SU(2)} for $G=\text{SU}(2)$, we get 
\begin{align}\label{kernel1 for SU(2)}
K_N(t,\theta)&=\frac{e^{it}}{e^{i\theta}-e^{-i\theta}}\sum_{m\in\mathbb{Z}} e^{-itm^2+im\theta}
\varphi(\frac{m^2-1}{N^2})m\\
&=\frac{e^{it}}{2}\sum_{m\in\mathbb{Z}}e^{-itm^2}
\varphi(\frac{m^2-1}{N^2})m\cdot\frac{e^{im\theta}-e^{-im\theta}}{e^{i\theta}-e^{-i\theta}},  \ \ \theta\in\mathbb{R}/2\pi\mathbb{Z}.\label{kernel2 for SU(2)}
\end{align}
\end{exmp}

\begin{cor}\label{First}
\eqref{KeyEst} holds for the following two scenarios. \\
\underline{Scenario 1}: $x=1_G$, where $1_G$ is the identity element of $G$. \\\underline{Scenario 2}: $\|\frac{1}{2\pi}\langle\alpha,H\rangle\|\gtrsim\frac{1}{N}$ for any $x$ conjugate to $\exp H$. This is to say that the variable $H$ is away from all the cell walls $\{H\mid \|\frac{1}{2\pi}\langle\alpha,H\rangle\|=0\text{ for some }\alpha\in P\}$ by a distance of $\gtrsim\frac{1}{N}$. 
\end{cor}
\begin{proof}
Scenario 1: When $x=1_G$, the character equals $\chi_\lambda(1_G)=d_\lambda=\prod_{\alpha\in P}\langle\alpha,\lambda\rangle/\prod_{\alpha\in P}\langle\alpha,\rho\rangle$. Then by \eqref{Schrodinger kernel2}, the Schr\"odinger kernel at $x=1_G$ equals 
\begin{equation}\label{KN(t,1)}
K_N(t,1_G)=\frac{e^{it|\rho|^2}}{(\prod_{\alpha\in P}\langle\alpha,\rho\rangle)^2|W|}\sum_{\lambda\in \Lambda}e^{-it|\lambda|^2}\varphi(\frac{|\lambda|^2-|\rho|^2}{N^2})(\prod_{\alpha\in P}\langle \alpha,\lambda\rangle)^2.
\end{equation}
Note that $f(\lambda)=\left(\prod_{\alpha\in P}\langle \alpha,\lambda\rangle\right)^2$ is a polynomial in the variable $\lambda=n_1w_1+\cdots+n_rw_r\in\Lambda$ of degree $2|P|$, which equals $d-r$ by \eqref{|P|}. Thus $f$ is also a pseudo-polynomial of degree $d-r$. Then the desired estimate is a direct consequence of Lemma \ref{WeylSum}. \\
Scenario 2: By Lemma 4.13.4 of Chapter 4 in \cite{Var84}, the Weyl denominator $D_P=\sum_{s\in W}(\det s) e^{i\langle s(\rho), H\rangle}$ can be rewritten as  
\begin{align}\label{Weyldenominator}
D_P=e^{-i\langle\rho, H\rangle}\prod_{\alpha\in P} (e^{i\langle\alpha, H\rangle}-1). 
\end{align}
Note that 
\begin{align*}
1\lesssim \frac{|e^{i\langle\alpha,H\rangle}-1|}{\|\frac{1}{2\pi}\langle\alpha,H\rangle\|}\lesssim 1.
\end{align*}
Then by assumption the Weyl denominator 
satisfies 
\begin{equation}\label{BoundOnDPH}
|D_P(H)|\gtrsim \prod_{\alpha\in P}\|\frac{1}{2\pi}\langle\alpha,H\rangle\|\gtrsim N^{-|P|}.  
\end{equation}
Let 
$$F=\sum_{\lambda\in\Lambda}e^{-it|\lambda|^2+i\langle\lambda,H\rangle}\varphi(\frac{|\lambda|^2-|\rho|^2}{N^2})\cdot f,$$ 
where $f=\prod_{\alpha\in P}\langle\alpha,\lambda\rangle$. Note that $f$ is a polynomial and thus also a pseudo-polynomial of degree $|P|$ in $\lambda$. Applying Lemma \eqref{WeylSum} to $F$ we get 
\begin{equation*}
|K_N(t,x)|=|\frac{e^{it|\rho|^2}}{(\prod_{\alpha\in P}\langle\alpha,\rho\rangle)D_P(H)}|\cdot |F|\lesssim|\frac{1}{D_P(H)}|\cdot |F|\lesssim N^{|P|}\cdot \frac{N^{r+|P|}}{(\sqrt{q}(1+N\|\frac{t}{2\pi D}-\frac{a}{q}\|^{1/2}))^{r}}.
\end{equation*}
Recalling $|P|=\frac{d-r}{2}$, we establish \eqref{KeyEst} for Scenario 2. 
\end{proof}

\begin{exmp}\label{SU(2)0}
We specialize the Schr\"{o}dinger kernel \eqref{kernel1 for SU(2)} and \eqref{kernel2 for SU(2)} to the case of $G=\text{SU}(2)$.  Scenario 1 in the above corollary corresponds to when $\theta\in 2\pi\mathbb{Z}$ and 
\begin{align}\label{KNtheta1}
K_N(t,\theta)=\frac{e^{it}}{2}\sum_{m\in\mathbb{Z}}e^{-itm^2}
\varphi(\frac{m^2-1}{N^2})m^2, \ \ |K_N(t,\theta)|\lesssim\left|\sum_{m\in\mathbb{Z}}e^{-itm^2}
\varphi(\frac{m^2-1}{N^2})m^2\right|.
\end{align}
Scenario 2 corresponds to when $|e^{i\theta}-e^{-i\theta}|\gtrsim \frac{1}{N}$, equivalently, when $\theta$ is away from the cell walls $\{0,\pi\}$ by a distance $\gtrsim\frac{1}{N}$. In this case, 
\begin{align}\label{KNtheta2}
|K_N(t,\theta)|\lesssim \left|\frac{1}{e^{i\theta}-e^{-i\theta}}\right|\cdot \left|\sum_{m\in\mathbb{Z}} e^{-itm^2+im\theta}
\varphi(\frac{m^2-1}{N^2})m\right|.
\end{align}
Then we get the desired estimates for \eqref{KNtheta1} and \eqref{KNtheta2} using Lemma \ref{WeylSum}. 
\end{exmp}

\subsection{Pseudo-polynomial Behavior of Characters}
\label{Pseudo-polynomial Behavior of Characters}
We have established the key estimates \eqref{KeyEst} for when the variable $\exp H$ in the maximal torus is either the identity or away from all the cell walls by a distance of $\gtrsim\frac{1}{N}$. To establish \eqref{KeyEst} fully, we need to look at the scenarios when the variable $\exp H$ is close to the some of the cell walls within a distance of $\lesssim\frac{1}{N}$. In this section, we first deal with the scenario when the variable $\exp H$ is close to all the cell walls within a distance of $\lesssim\frac{1}{N}$. To achieve this end, we first prove the following crucial lemma on the pseudo-polynomial behavior of characters.

\begin{lem}\label{EstBGG1}
Let $\mu\in i\mathfrak{b}^*$. For $\lambda\in i\mathfrak{b}^*$, define 
\begin{align*}
\chi^\mu(\lambda, H)
=\frac{\sum_{s\in W}\det{(s)}e^{i\langle s(\lambda+\mu),H\rangle}}{\sum_{s\in W}\det{(s)}e^{i\langle s(\rho), H\rangle}}. 
\end{align*}
Let $L\cong\mathbb{Z}^r$ be the weight lattice or the root lattice (or any sublattice of full rank of the weight lattice), and viewing $\chi^\mu(\lambda, H)$ as a function in $\lambda\in L$, we have 
\begin{equation}\label{Dchi}
|D_{i_1}\cdots D_{i_k}\chi^\mu(\lambda, H)|\lesssim N^{\frac{d-r}{2}-k}
\end{equation}
holds uniformly in $|\lambda|\lesssim N$, $|H|\lesssim N^{-1}$, and $N\geq 1$, for all $k\in\mathbb{Z}_{\geq 0}$. In other words, $\chi^\mu(\lambda, H)$ is a pseudo-polynomial of degree $\frac{d-r}{2}$ in $\lambda$ uniformly in $|H|\lesssim N^{-1}$. 
\end{lem}

Using this lemma, applying Lemma \ref{WeylSum} to the Schr\"odinger kernel $K_N$ in the form of \eqref{Schrodinger kernel2}, we immediately get the following corollary. 

\begin{cor}\label{neighborhoodofidentity}
Inequality \eqref{KeyEst} holds uniformly when $x\in G$ is conjugate to $\exp H$ such that $|H|\lesssim N^{-1}$. In other words, when $x$ is within $\lesssim N^{-1}$ a distance from the identity $1_G$. 
\end{cor}

We now prove Lemma \ref{EstBGG1} for $L\cong\mathbb{Z}w_1+\cdots+\mathbb{Z}w_r$ being the weight lattice (the case for the root lattice or any other sublattice can be proved similarly). First note that as $|H|\lesssim N^{-1}$ for $N$ large enough, by \eqref{Weyldenominator}, we have 
\begin{align*}
\left|\frac{\prod_{\alpha\in P}\langle \alpha, H\rangle}{D_P}\right|\approx 1. 
\end{align*}
Thus it suffices to show \eqref{Dchi} replacing $\chi^\mu(\lambda,H)$ by  
\begin{align}\label{chi'lambdaH}
\chi^\mu_1(\lambda, H)=\frac{\sum_{s\in W}\det{(s)}e^{i\langle s(\lambda+\mu),H\rangle}}{\prod_{\alpha\in P}\langle \alpha, H\rangle}. 
\end{align}

\subsubsection{Approach 1: via BGG-Demazure Operators}
The idea is to expand the numerator of $\chi^\mu_1(\lambda, H)$ into a power series of polynomials in $H\in i\mathfrak{b}^*$, which are \textit{anti-invariant} with respect to the Weyl group $W$, and then to estimate the quotients of these polynomial over the denominator $\prod_{\alpha\in P}\langle \alpha, H\rangle$. We will see that these quotients are in fact polynomials in $H\in i\mathfrak{b}^*$, and can be more or less explicitly computed by the  \textit{BGG-Demazure operators}. We now review the basic definitions and facts of the BGG-Demazure operators and the related invariant theory. A good reference is Chapter IV in \cite{Hil82}. 

From now on, we fix an inner product space $(\mathfrak{a}, \langle\ , \ \rangle)$ and let $\Phi$ be an integral root system in the dual space $(\mathfrak{a}^*, \langle\ , \ \rangle)$. 
Let $P(\mathfrak{a})$ be the space of polynomial functions on $\mathfrak{a}$. The orthogonal group $O(\mathfrak{a})$ with respect to the inner product on $\mathfrak{a}$, in particular the Weyl group, acts on $P(\mathfrak{a})$ by 
\begin{align*}
(sf)(H):=f(s^{-1}H), \ \ s\in O(\mathfrak{a}), \ \ f\in P(\mathfrak{a}), \ \ H\in \mathfrak{a}. 
\end{align*}

\begin{defn}
For $\alpha\in \mathfrak{a}^*$, let $s_\alpha:\mathfrak{a}\to \mathfrak{a}$ denote the reflection about the hyperplane 
$$\{H\in\mathfrak{a}: \alpha(H)=0\},$$
that is, 
\begin{equation*}
s_\alpha(H):=H-2\frac{\alpha(H)}{\langle\alpha,\alpha\rangle}H_\alpha
\end{equation*}
where $H\in \mathfrak{a}$. Here $H_\alpha$ corresponds to $\alpha$ through the identification $\mathfrak{a}\xrightarrow{\sim}\mathfrak{a}^*$. 
Define the \textit{BGG-Demazure operator} $\Delta_\alpha: P(\mathfrak{a})\to P(\mathfrak{a})$ associated to $\alpha\in \mathfrak{a}^*$ by 
\begin{equation*}
\Delta_\alpha(f)=\frac{f-s_\alpha(f)}{\alpha}.
\end{equation*}
\end{defn}

As an example, we compute $\Delta_\alpha(\lambda^m)$ for $\lambda\in \mathfrak{a}^*$. 
\begin{align*}
\Delta_\alpha(\lambda^m)
&=\frac{\lambda^m-\lambda(\cdot-2\frac{\alpha}{\langle \alpha,\alpha\rangle}H_\alpha)^m}{\alpha}\\
&=\frac{\lambda^m-(\lambda-2\frac{\langle\lambda,\alpha,\rangle}{\langle\alpha,\alpha\rangle}\alpha)^m}{\alpha}\\
&=\sum_{i=1}^m(-1)^{i-1}\binom{m}{i}\frac{2^i}{\langle\alpha,\alpha\rangle^i}\langle\lambda,\alpha\rangle^i\alpha^{i-1}\lambda^{m-i} \numberthis \label{BGGMon}.
\end{align*}
This computation in particular implies that for any $f\in P(\mathfrak{a})$, the operator $\Delta_\alpha(f)$ lowers the degree of $f$ by at least 1. 

Let $P(\mathfrak{a})^W$ denote the subspace of $P(\mathfrak{a})$ that are invariant under the action of the Weyl group $W$, that is, 
\begin{equation*}
P(\mathfrak{a})^W:=\{f\in P(\mathfrak{a}) \mid  sf=f \text{ for all }s\in W\}. 
\end{equation*}
We call $P(\mathfrak{a})^W$ the space of \textit{invariant polynomials}. We also define 
\begin{equation*} 
P(\mathfrak{a})^W_{\det}:=\{f\in P(\mathfrak{a}) \mid  sf=(\det s) f\text{ for all }s\in W\}.
\end{equation*}
We call $P(\mathfrak{a})^W_{\det}$ the space of \textit{anti-invariant polynomials}. We have the following proposition which states that $P(\mathfrak{a})^W_{\det}$ is a free $P(\mathfrak{a})^W$-module of rank 1. 

\begin{prop}[Chapter II, Proposition 4.4 in \cite{Hil82}]\label{Anti-invariant by invariant}
Define $d_{\det}\in P(\mathfrak{a})$ by 
\begin{equation*}
d_{\det}=\prod_{\alpha\in P}\alpha.
\end{equation*}
Then $d_{\det}\in P(\mathfrak{a})^W_{\det}$ and 
\begin{equation*}
P(\mathfrak{a})^W_{\det}=d_{\det} \cdot P(\mathfrak{a})^W. 
\end{equation*}
\end{prop}

By the above proposition, given any anti-invariant polynomial  $f$, we have $f=d\cdot g$ where $g$ is invariant. We call $g$ the \textit{invariant part} of $f$. The BGG-Demazure operators provide a procedure that computes the invariant part of any anti-invariant  polynomial. We describe this procedure as follows. The Weyl group $W$ is generated by the reflections  $s_{\alpha_1}, \cdots, s_{\alpha_r}$ where $S=\{\alpha_1,\cdots,\alpha_r\}$ is the set of simple roots. Define the \textit{length} of $s\in W$ to be the smallest number $k$ such that $s$ can be written as $s=s_{\alpha_{i_1}}\cdots s_{\alpha_{i_k}}$. The longest element $s$ in $W$ is of length $|P|=\frac{d-r}{2}$, 
and such $s$ is unique; see Section 1.8 in \cite{Hum90}. Write $s=s_{\alpha_{i_1}}\cdots s_{\alpha_{i_L}}$. Set 
\begin{equation*}
\delta=\Delta_{\alpha_{i_1}}\cdots\Delta_{\alpha_{i_L}}
\end{equation*}
and note that it is well defined in the sense it does not depend on the particular choice of the decomposition $s=s_{\alpha_{i_1}}\cdots s_{\alpha_{i_L}}$; see Chapter IV, Proposition 1.7 in \cite{Hil82}. 

\begin{prop}[Chapter IV, Proposition 1.6 in \cite{Hil82}]\label{BGGAnti}
We have 
\begin{equation*}
\delta f =\frac{|W|}{d_{\det }}\cdot f
\end{equation*}
for all $f\in P(\mathfrak{a})_{\det}^W$.  
\end{prop}
That is, the operator $\delta$ produces the invariant part of any anti-invariant polynomial (modulo a multiplicative constant). 
As an example, we compute $\delta=\Delta_{\alpha_{i_1}}\cdots\Delta_{\alpha_{i_{L}}}$ on $\lambda^m$. Proceed inductively  using \eqref{BGGMon}, we arrive at the following proposition. 

\begin{prop}\label{BGGMon2}
Let $m\geq L$. Then
\begin{align*} 
\delta(\lambda^m)=\sum_{\theta, a(\alpha,\beta),b(\gamma),c(\zeta), \eta\in\mathbb{Z}} (-1)^\theta \prod_{\alpha\leq\beta} \langle\alpha_{i_{\alpha}},  \alpha_{i_{\beta}}\rangle^{a(\alpha,\beta)} \prod_{\gamma} \langle\lambda,\alpha_{i_\gamma}\rangle^{b(\gamma)}\prod_{\zeta}\alpha_{i_\zeta}^{c(\zeta)}\lambda^{\eta}
\end{align*}
such that the following statements are true. \\
(1) In each term of the sum,  $\sum_{\gamma}b(\gamma)+\eta=m$. \\
(2) In each term of the sum, $\sum_{\zeta}c(\zeta)+\eta=m-L$. \\
(3) In each term of the sum, $\sum_{\gamma}b(\gamma)-\sum_{\zeta}c(\zeta)=L$. \\
(4) In each term of the sum, $|a(\alpha,\beta)|\leq mL$ and $b(\gamma), c(\zeta), \eta=0,1,\ldots,m$. \\
(5) There are in total less than $3^{mL}$ terms in the sum. 
\end{prop}

Note that since each BGG-Demazure operator $\Delta_{\alpha_{i_j}}$ in $\delta=\Delta_{\alpha_{i_1}}\cdots\Delta_{\alpha_{i_{L}}}$ lowers the degree of polynomials by at least 1, $\delta$ lowers the   degree by at least $L$. Thus 
\begin{equation}\label{m<|P|}
\delta(\lambda^m)=0, \ \ \text{ for } m<L.
\end{equation}

\begin{exmp}
We specialize the discussion to the case $M=\text{SU}(2)$. Recall that $\mathfrak{a}^*=\mathbb{R}w$ where $w$ is the fundamental weight, and $\Phi=\{\pm \alpha\}$ with $\alpha=2w$. $P(\mathfrak{a})$ consists of polynomials in the variable $\lambda\in\underset{1}{\mathbb{R}}\underset{\mapsto}{\cong} \underset{w}{\mathbb{R}w}$. For $\lambda\in\underset{1}{\mathbb{R}}\underset{\mapsto}{\cong} \underset{w}{\mathbb{R}w}$,  and $f\in P(\mathfrak{a})$, we have  
\begin{align*}
(\delta f)(\lambda)&=\frac{f(\lambda)-f(-\lambda)}{2\lambda},
\\
\delta (\lambda^m)&=\left\{\begin{array}{ll}
\lambda^{m-1}, & m \text{ odd},\\
0, & m\text{ even},
\end{array}
\right. \\
d_{\det}(\lambda)&=2\lambda.\numberthis \label{ddetlambda}
\end{align*}
\end{exmp}

We can now finish the proof of \eqref{Dchi}. 
\begin{proof}[Proof of Lemma \ref{EstBGG1}]
Recall that it suffices to prove \eqref{Dchi} replacing $\chi^\mu(\lambda,H)$ by  
$\chi^\mu_1(\lambda, H)$ in \eqref{chi'lambdaH}. 
Using power series expansions, write 
\begin{align*}
\sum_{s\in W}(\det s) e^{i\langle\lambda+\mu, H\rangle}&=\sum_{s\in W}\det s\ \sum_{m=0}^\infty \frac{1}{m!}(i\langle s(\lambda+\mu), H\rangle)^m\\
&=\sum_{m=0}^\infty \frac{i^m}{m!}
\sum_{s\in W}\det s\ \langle s(\lambda+\mu), H\rangle^m.\numberthis \label{chipower}
\end{align*}
Note that 
\begin{equation}\label{fmH0}
f_m(H)=f_m(\lambda)=f_m(\lambda, H):=\sum_{s\in W}\det s\ \langle s(\lambda+\mu), H\rangle^m
\end{equation} 
is an anti-invariant polynomial in $H$ with respect to the Weyl group $W$, thus by Proposition \ref{BGGAnti}, 
\begin{equation*}
f_m(H)=\frac{d_{\det}(H)}{|W|}\cdot \delta f_m(H)=\frac{\prod_{\alpha\in P}\langle\alpha,H\rangle}{|W|}\cdot \delta f_m(H).
\end{equation*}
This implies that we can rewrite \eqref{chi'lambdaH} as  
\begin{equation*}
\chi^{\mu_1}(\lambda,H)=\frac{1}{|W|}\sum_{m=0}^\infty\frac{i^m}{m!} \delta f_m(H). 
\end{equation*}
Thus to prove \eqref{Dchi}, it suffices to prove that 
\begin{equation*}
\sum_{m=0}^\infty\frac{1}{m!}\left|D_{i_1}\cdots D_{i_k} \left(\delta f_m(\lambda)\right)\right|\lesssim N^{L-k},
\end{equation*}
for all $k\in\mathbb{Z}_{\geq 0}$, uniformly in $|n_i|\lesssim N$, where $\lambda=n_1w_1+\cdots+n_rw_r$. Then by $\eqref{fmH0}$, it suffices to prove that 
\begin{equation*}
\sum_{m=0}^\infty\frac{1}{m!}\left|D_{i_1}\cdots D_{i_k} \left(\delta \left[\left(s(\lambda+\mu)\right)^m\right]\right)\right|\lesssim N^{L-k},\  \ \forall s\in W. 
\end{equation*}
Without loss of generality, it suffices to show 
\begin{equation}\label{Differencing}
\sum_{m=0}^\infty\frac{1}{m!}\left|D_{i_1}\cdots D_{i_k} \left(\delta \left[(\lambda+\mu)^m\right]\right)\right|\lesssim N^{L-k}.
\end{equation}
Noting \eqref{m<|P|}, it suffices to consider cases when $m\geq L$. We apply Proposition \ref{BGGMon2} to write 
\begin{align*}
&\delta((\lambda+\mu)^m)(H)\\
&=\sum_{\theta, a(\alpha,\beta),b(\gamma),c(\zeta), \eta} (-1)^\theta \prod_{\alpha\leq\beta} \langle\alpha_{i_{\alpha}},  \alpha_{i_{\beta}}\rangle^{a(\alpha,\beta)} \prod_{\gamma} \langle \lambda+\mu,\alpha_{i_\gamma}\rangle^{b(\gamma)}\prod_{\zeta}\langle\alpha_{i_\zeta},H\rangle^{c(\zeta)}\langle\lambda+\mu,H\rangle^{\eta}. \numberthis \label{deltalambda+mu}
\end{align*}
First note that for $\lambda=n_1w_1+\cdots+n_rw_r$, $|n_i|\lesssim N$, $i=1,\cdots, r$, we have 
\begin{equation}\label{firstbound}
1\lesssim |\langle\alpha_{i},\alpha_{j}\rangle|\lesssim 1, \ \  |\langle\lambda+\mu,\alpha_{i}\rangle|\lesssim N,
\end{equation}
and by the assumption $|H|\lesssim N^{-1}$, 
\begin{equation}\label{secondbound}
|\langle\alpha_{i},H\rangle|\lesssim N^{-1}, \ \ |\langle\lambda+\mu,H\rangle|
=\left|\left(\sum_{i=1}^rn_i\langle w_i,H\rangle\right)+\langle \mu,H\rangle\right|\lesssim 1.
\end{equation}
These imply 
\begin{align}\label{delta(lambda+mu)^m}
|\delta((\lambda+\mu)^m)(H)|
\leq 
\sum_{\theta,a(\alpha,\beta),b(\gamma),c(\zeta),\eta}C^{\sum_{\alpha,\beta}|a(\alpha,\beta)|+\sum_{\gamma}b(\gamma)+\sum_{\zeta}c(\zeta)+\eta}N^{\sum_{\gamma}c(\gamma)-\sum_{\zeta}c(\zeta)}
\end{align}
for some constant $C$ independent of $m$. 
Now we derive a similar estimate for $D_i\left(\delta\left[(\lambda+\mu)^m\right]\right)(H)$. By \eqref{deltalambda+mu}, 
\begin{align*}
D_i\left(\delta\left[(\lambda+\mu)^m\right]\right)(H)&=\sum_{\theta, a(\alpha,\beta),b(\gamma),c(\zeta), \eta} (-1)^\theta \prod_{\alpha\leq\beta} \langle\alpha_{i_{\alpha}},  \alpha_{i_{\beta}}\rangle^{a(\alpha,\beta)}\prod_{\zeta}\langle\alpha_{i_\zeta},H\rangle^{c(\zeta)}\\
&\cdot D_i\left(\prod_{\gamma} \langle \lambda+\mu ,\alpha_{i_\gamma}\rangle^{b(\gamma)}\langle\lambda+\mu,H\rangle^{\eta}\right).\numberthis\label{Didelta}
\end{align*}
For $\lambda=n_1w_1+\cdots+n_rw_r$,  we compute 
\begin{align*}
D_i\left(\langle\lambda+\mu,\alpha_{i_\gamma}\rangle\right)&=\langle\alpha_i,\alpha_{i_\gamma}\rangle,\\
D_i\left(\langle\lambda+\mu,H\rangle\right)
&=\langle\alpha_i, H\rangle. 
\end{align*}
The above two formulas combined with \eqref{firstbound}, \eqref{secondbound}, and the Leibniz rule \eqref{Leibniz} for $D_i$ imply  
\begin{align*}
\left|D_i\left(\prod_{\gamma} \langle \lambda+\mu ,\alpha_{i_\gamma}\rangle^{b(\gamma)}\langle\lambda+\mu,H\rangle^{\eta}\right)\right|
\leq C^{\sum_{\gamma}b(\gamma)+\eta} N^{\sum_{\gamma}b(\gamma)-1}. 
\end{align*}
This combined with \eqref{firstbound}, \eqref{secondbound} and \eqref{Didelta} implies 
\begin{align*}
\left|D_i\left(\delta\left[(\lambda+\mu)^m\right]\right)(H)\right|&\lesssim
\sum_{\theta,a(\alpha,\beta),b(\gamma),c(\zeta),\eta}C^{\sum_{\alpha,\beta}|a(\alpha,\beta)|+\sum_{\gamma}b(\gamma)+\sum_{\zeta}c(\zeta)+\eta}N^{\sum_{\gamma}b(\gamma)-\sum_{\zeta}c(\zeta)-1}.
\end{align*}
Inductively, we have 
\begin{align*}
\left|D_{i_1}\cdots D_{i_k}\left(\delta\left[(\lambda+\mu)^m\right]\right)(H)\right|&\lesssim
\sum_{\theta,a(\alpha,\beta),b(\gamma),c(\zeta),\eta}C^{\sum_{\alpha,\beta}|a(\alpha,\beta)|+\sum_{\gamma}b(\gamma)+\sum_{\zeta}c(\zeta)+\eta}N^{\sum_{\gamma}b(\gamma)-\sum_{\zeta}c(\zeta)-k},
\end{align*}
for some constant $C$ independent of $m$. This by Proposition \ref{BGGMon2} then implies  
\begin{align*}
\left|D_{i_1}\cdots D_{i_k}\left(\delta\left[(\lambda+\mu)^m\right]\right)(H)\right|\leq 3^{mL}C^{CmL}N^{L-k}\leq C^{m}N^{L-k}
\end{align*}
for some positive constant $C$ independent of $m$. This estimate implies \eqref{Differencing}, 
noting that 
\begin{equation}\label{Cmm!}
\sum_{m=0}^\infty\frac{C^{m}}{m!}\lesssim 1. 
\end{equation}
This finishes the proof. 
\end{proof}

\subsubsection{Approach 2: via Harish-Chandra's Integral Formula}

This very short approach expresses $\chi^\mu_1(\lambda, H)$ as an integral over the group $G$. 
We apply the Harish-Chandra's integral formula (see \cite{HC57}), which reads 
\begin{align*}
\sum_{s\in W} \det(s) e^{\langle s\lambda, \mu\rangle}
=\frac{\prod_{\alpha\in P}\langle \alpha, \lambda\rangle\cdot \prod_{\alpha\in P}\langle \alpha, \mu\rangle}{\prod_{\alpha\in P}\langle \alpha, \rho\rangle}\int_{G}e^{\langle \text{Ad}_g (\lambda), \mu\rangle}\ dg. 
\end{align*}
where $\lambda, \mu\in\mathfrak{b}_{\mathbb{C}}^*$, and $dg$ is the normalized Haar measure on $G$. Then we can rewrite $\chi^\mu_1(\lambda, H)$ as 
\begin{align*}
\chi^\mu_1(\lambda, H)=\frac{i^{|P|}\prod_{\alpha\in P}\langle \alpha, \lambda+\rho\rangle}{\prod_{\alpha\in P}\langle \alpha, \rho\rangle} \int_G e^{i\langle\lambda+\rho, \text{Ad}_g(H)\rangle}\ dg. 
\end{align*}
Note that 
$$\frac{i^{|P|}\prod_{\alpha\in P}\langle \alpha, \lambda+\rho\rangle}{\prod_{\alpha\in P}\langle \alpha, \rho\rangle}$$ 
is a polynomial in $\lambda\in\Lambda$ of degree $|P|=\frac{d-r}{2}$. Also, as $|H|\lesssim N^{-1}$, we have $|Ad_g(H)|\lesssim N^{-1}$ uniformly in $g\in G$, which implies that the integral 
$$f(\lambda)=\int_G e^{i\langle\lambda+\rho, \text{Ad}_g(H)\rangle}\ dg$$ 
as a function in $\lambda$ is a pseudo-polynomial of degree $0$, uniformly in $|H|\lesssim N^{-1}$. Then by the Leibniz rule, $\chi'(\lambda, H)$ as a function of $\lambda$ is a pseudo-polynomial of degree $\frac{d-r}{2}$, uniformly in $|H|\lesssim N^{-1}$. This finishes the proof of Lemma \ref{EstBGG1}.

\begin{rem}\label{Rootswithoutgroup}
Note that Lemma \ref{EstBGG1} can be stated purely in terms of an integral root system without mentioning the ambient compact Lie group, and it still holds true this way. It can be seen either by the approach via BGG-Demazure operators which is purely a root system theoretic argument, or by the fact that, 
for any integral root system $\Phi$, there associates to it a unique compact simply connected semisimple Lie group equipped with this root system, thus the approach via Harish-Chandra's integral formula still works, even though the argument explicitly involves the group.  
\end{rem}

\subsection{From the Weight Lattice to the Root Lattice}
\label{From the Weight Lattice to the Root Lattice}
We say $\exp H$ is a \textit{corner} in the maximal torus provided 
\begin{align*}
\|\frac{1}{2\pi}\langle\alpha,H\rangle\|=0, \ \ \text{for all }\alpha\in P. 
\end{align*} 
In this section, we extend Corollary \ref{neighborhoodofidentity} 
to the scenarios when $\exp H$ is within a distance of $\lesssim N^{-1}$ from some corner. That is, when  
\begin{align}\label{NearChamber}
\|\frac{1}{2\pi}\langle\alpha,H\rangle\|\lesssim N^{-1}, \ \ \text{for all }\alpha\in P. 
\end{align}
To this end, we rewrite the Schr\"odinger kernel $K_N(t,x)$ as a finite sum of exponential sums over the root lattice: 
\begin{align*}
K_{N}(t,x)&=C\sum_{\mu\in \Lambda/\Gamma}\sum_{\lambda\in \mu+\Gamma}e^{-it(|\lambda|^2-|\rho|^2)}
\varphi(\frac{|\lambda|^2-|\rho|^2}{N^2})\frac{\prod_{\alpha\in P}\langle \alpha,\lambda\rangle}{\prod_{\alpha\in P}\langle \alpha,\rho\rangle}
\frac{\sum_{s\in W}\det{(s)}e^{i\langle s(\lambda),H\rangle}}{\sum_{s\in W}\det{(s)}e^{i\langle s(\rho), H\rangle}}\\
&=C\sum_{\mu\in \Lambda/\Gamma}\sum_{\lambda\in\Gamma}e^{-it(|\lambda+\mu|^2-|\rho|^2)}
\varphi(\frac{|\lambda+\mu|^2-|\rho|^2}{N^2})\frac{\prod_{\alpha\in P}\langle \alpha,\lambda+\mu\rangle}{\prod_{\alpha\in P}\langle \alpha,\rho\rangle}
\frac{\sum_{s\in W}\det{(s)}e^{i\langle s(\lambda+\mu),H\rangle}}{\sum_{s\in W}\det{(s)}e^{i\langle s(\rho), H\rangle}} \numberthis \label{WeightToRoot}
\end{align*}
where $C=\frac{e^{it|\rho|^2}}{|W|}$.

\begin{prop}\label{EstRoot}
Let $\mu$ be an element in the weight lattice $\Lambda$ and let 
\begin{equation}\label{KerRoot}
K^\mu_N(t,x)=\sum_{\lambda\in\Gamma}e^{-it(|\lambda+\mu|^2-|\rho|^2)}
\varphi(\frac{|\lambda+\mu|^2-|\rho|^2}{N^2})\frac{\prod_{\alpha\in P}\langle \alpha,\lambda+\mu\rangle}{\prod_{\alpha\in P}\langle \alpha,\rho\rangle}
\frac{\sum_{s\in W}\det{(s)}e^{i\langle s(\lambda+\mu),H\rangle}}{\sum_{s\in W}\det{(s)}e^{i\langle s(\rho), H\rangle}} 
\end{equation}
where $x$ is conjugate to $\exp H$. Then
\begin{align}\label{KNmutx}
|K^\mu_N(t,x)|\lesssim \frac{N^d}{(\sqrt{q}(1+N\|\frac{t}{2\pi D}-\frac{a}{q}\|^{1/2}))^{r}}
\end{align}
for $\frac{t}{2\pi D}\in\mathcal{M}_{a,q}$, uniformly for $\|\frac{1}{2\pi}\langle\alpha,H\rangle\|\lesssim N^{-1}$  for all $\alpha\in P$. 
\end{prop}

Using \eqref{WeightToRoot} and the finiteness of $\Lambda/\Gamma$, we have the following corollary. 

\begin{cor}\label{CloseTo0}
Inequality \eqref{KeyEst} holds for the case when $\|\frac{1}{2\pi}\langle\alpha,H\rangle\|\lesssim N^{-1}$ for all $\alpha\in P$.
\end{cor}
To prove Proposition \ref{EstRoot}, we first prove a variant of Lemma \ref{EstBGG1}. 
\begin{lem}\label{EstBGG2}
Let 
\begin{align}\label{EstBGG2character}
\chi^\mu(\lambda, H)=\frac{\sum_{s\in W}(\det s) e^{i \langle s(\mu+\lambda), H\rangle}}{e^{-i\langle\rho, H\rangle}\prod_{\alpha\in P}(e^{i\langle \alpha, H\rangle}-1)}
\end{align}
be defined as in Lemma \ref{EstBGG1}. Assume in addition that 
$\mu\in\Lambda$. 
Then $\chi^\mu(\lambda, H)$ as a function in $\lambda\in\Gamma$ is a pseudo-polynomial of degree $\frac{d-r}{2}$, uniformly in $H$ such that $\|\frac{1}{2\pi}\alpha(H)\|\lesssim N^{-1}$ for all $\alpha\in P$. 
\end{lem}
\begin{proof}
For all $H\in i\mathfrak{b}^*$ such that $\|\frac{1}{2\pi}\langle\alpha,H\rangle\|\lesssim N^{-1}$ for all $\alpha\in P$, by considering the dual basis of the simple roots $\{\alpha_1,\cdots, \alpha_r\}$, we can write 
\begin{align}\label{HH1H2}
H=H_1+H_2
\end{align}
such that 
\begin{align}\label{alphaH1}
|\frac{1}{2\pi}\langle\alpha_i, H_1\rangle|=\|\frac{1}{2\pi}\langle\alpha_i,H\rangle\|\lesssim N^{-1}, \ \ i=1,\cdots, r,
\end{align}
and 
\begin{equation}\label{alphaH2}
\langle\alpha_i,H_2\rangle\in 2\pi\mathbb{Z}, \ \ i=1,\cdots, r.
\end{equation}
This implies that $\exp H_2$ is a corner and 
\begin{align}\label{H1N-1}
|H_1|\lesssim N^{-1}. 
\end{align}
Then for $\lambda\in\Gamma=\mathbb{Z}\alpha_1+\cdots+\mathbb{Z}\alpha_r$, 
\begin{align}\label{chiH1+H3Pre}
\chi^\mu(\lambda, H)=\chi^\mu(\lambda, H_1+H_2)
&=\frac{\sum_{s\in W}\det{(s)}e^{i\langle s(\lambda+\mu),H_1\rangle}e^{i\langle s(\mu),H_2\rangle}}{e^{-i\langle\rho,H_1+H_2\rangle}\prod_{\alpha\in P}\left(e^{i\langle\alpha, H_1\rangle}-1\right)}.
\end{align}
Note that, see Corollary 4.13.3 in \cite{Var84},  
$s(\mu)-\mu\in\Gamma$ for all $\mu\in\Lambda$ and $s\in W$, which combined with \eqref{alphaH2} implies 
\begin{align*}
e^{i\langle s(\mu),H_2\rangle}
=e^{i\langle \mu,H_2\rangle}, \ \ \text{for all }\mu\in\Lambda, \ \ s\in W. 
\end{align*}
Then \eqref{chiH1+H3Pre} becomes 
\begin{align}\label{chiH1+H3}
\chi^\mu(\lambda, H)
&=\frac{e^{i\langle \mu,H_2\rangle}}{e^{-i\langle\rho, H_2\rangle}}\cdot \frac{\sum_{s\in W}\det{(s)}e^{i\langle s(\lambda+\mu),H_1\rangle}}{e^{-i\langle\rho,H_1\rangle}\prod_{\alpha\in P}\left(e^{i\langle\alpha, H_1\rangle}-1\right)}
=e^{i\langle \mu+\rho,H_2\rangle}\cdot\chi^\mu(\lambda, H_1), 
\end{align}
which is a pseudo-polynomial in $\lambda\in\Gamma$ of degree $\frac{d-r}{2}$ uniformly in $|H_1|\lesssim N^{-1}$ by Lemma \ref{EstBGG1}.
\end{proof}

\begin{proof}[Proof of Proposition \ref{EstRoot}]
Since $\prod_{\alpha\in P}\langle \alpha,\lambda+\mu\rangle$ is a polynomial, and  thus also a pseudo-polynomial in $\lambda$ of degree $|P|=\frac{d-r}{2}$, and $\chi^\mu(\lambda, H)$ is a pseudo-polynomial of degree $\frac{d-r}{2}$ uniformly in $\|\frac{1}{2\pi}\langle\alpha, H\rangle\|\lesssim N^{-1}$ for all $\alpha\in P$ by the previous lemma, 
\begin{align*}
f(\lambda)=\frac{\prod_{\alpha\in P}\langle \alpha,\lambda+\mu\rangle}{\prod_{\alpha\in P}\langle \alpha,\rho\rangle}\cdot 
\chi^\mu(\lambda, H)
\end{align*} 
is then a pseudo-polynomial of degree $d-r$ uniformly in $\|\frac{1}{2\pi}\langle\alpha, H\rangle\|\lesssim N^{-1}$ for all $\alpha\in P$.  Then the desired result comes from a direct application of Lemma \ref{WeylSum}. 
\end{proof}

\begin{exmp}\label{SU(2)}
We specialize the discussion in this section to the case $G=\text{SU}(2)$. Recall that $\Lambda=\mathbb{Z}w$, $\Gamma=\mathbb{Z}\alpha$ with $\alpha=2w$, thus $\Lambda/\Gamma\cong \{0, 1\}\cdot w$. \eqref{WeightToRoot} specializes to 
\begin{align*}
K_N(t,\theta)=\frac{e^{it}}{2}\left(K_N^0(t,\theta)+K_N^1(t,\theta)\right),
\end{align*}
where 
\begin{align*}
K_N^0=\sum_{\substack{m=2k,\\k\in\mathbb{Z}}}e^{-itm^2}
\varphi(\frac{m^2-1}{N^2})m\cdot\frac{e^{im\theta}-e^{-im\theta}}{e^{i\theta}-e^{-i\theta}},
\end{align*}
\begin{align*}
K_N^1=\sum_{\substack{m=2k+1,\\k\in\mathbb{Z}}}e^{-itm^2}
\varphi(\frac{m^2-1}{N^2})m\cdot\frac{e^{im\theta}-e^{-im\theta}}{e^{i\theta}-e^{-i\theta}},
\end{align*}
for $\theta\in\mathbb{R}/2\pi\mathbb{Z}$. Condition \eqref{NearChamber} specializes to $\|\frac{\theta}{\pi}\|\lesssim N^{-1}$. Write 
$\theta=\theta_1+\theta_2$, where $|\theta_1|\lesssim N^{-1}$, and $\theta_2=0, \pi$.  Then for $m=2k$, $k\in\mathbb{Z}$, 
\begin{align*}
\chi_{m}(\theta)&=\frac{1}{e^{-i\theta}(e^{i2\theta_1}-1)}\cdot (e^{im\theta_1}-e^{-im\theta_1})\\
&=\frac{1}{e^{-i\theta}(e^{i2\theta_1}-1)}\cdot \sum_{n=0}^\infty\frac{i^n}{n!}\left((m\theta_1)^n-(-m\theta_1)^n\right)\\
&=\frac{\theta_1}{e^{-i\theta}(e^{i2\theta_1}-1)}\cdot \sum_{n\text{ odd}}\frac{i^n}{n!}(2\theta_1^{n-1}m^n), \numberthis
\end{align*}
and similarly for $m=2k+1$, $k\in\mathbb{Z}$, 
\begin{align*}
\chi_{m}(\theta)=\frac{e^{i\theta_2}\theta_1}{e^{-i\theta}(e^{i2\theta_1}-1)}\cdot \sum_{n\text{ odd}}\frac{i^n}{n!}(2\theta_1^{n-1}m^n).
\end{align*}
Note that we are implicitly applying Proposition \ref{BGGAnti} so that
\begin{align*}
f_n(\theta_1):=(m\theta_1)^n-(-m\theta_1)^n
=\theta_1\cdot \delta f_n=\left\{\begin{array}{ll}\theta_1\cdot 2\theta_1^{n-1}m^n, & n\text{ odd},\\0, & n\text{ even}.\end{array}\right.
\end{align*}
If $|k|\lesssim N$, then it is clear that 
\begin{align*}
|D^L\chi_{2k}|\lesssim N^{1-L}, \ \ |D^L\chi_{2k+1}|\lesssim N^{1-L}, \ \ L\in\mathbb{Z}_{\geq 0},
\end{align*}
where $D$ is the difference operator with respect to the variable $k$. These two inequalities will give the desired estimates for $K_N^0$ and $K_N^1$ respectively using the Weyl sum estimate Lemma  \ref{WeylSum} in 1 dimension. 
\end{exmp}

\subsection{Root Subsystems}
\label{RootSubsystems}
To finish the proof of part (ii) of Theorem \ref{MainEstimate},  considering Corollaries \ref{First} and \ref{CloseTo0}, it suffices to prove \eqref{KeyEst} in the scenarios when $\exp H$ is away from all the corners by a distance of $\gtrsim N^{-1}$ but stays close to some cell walls within a distance of $\lesssim N^{-1}$. 
We will identify these other walls as belonging to a \textit{root subsystem} of the original root system $\Phi$, and then we will decompose the character, the weight lattice and thus the Schr\"{o}dinger kernel according to this root subsystem. 

\subsubsection{Identifying Root Subsystems and Rewriting the Character}
Given any $H\in i \mathfrak{b}^*$, let $Q_H$ be the subset of the set $\Phi$ of  roots defined by 
\begin{equation*}
Q_H:=\{\alpha\in  \Phi \mid \|\frac{1}{2\pi}\langle\alpha, H\rangle\|\leq N^{-1}\}.
\end{equation*}
Thus 
\begin{equation*}
\Phi\setminus Q_H=\{\alpha\in  \Phi \mid \|\frac{1}{2\pi}\langle\alpha, H\rangle\|> N^{-1}\}.
\end{equation*}
Define 
\begin{equation}\label{PhiH}
\Phi_H:=\{\alpha\in \Phi\mid \alpha\text{ lies in the $\mathbb{Z}$-linear span of }Q_H\},
\end{equation}
then $\Phi_H\supset Q_H$, and 
\begin{equation}\label{alphaPhiH}
\|\frac{1}{2\pi}\langle\alpha, H\rangle\|\lesssim N^{-1}, \ \ \forall\alpha\in \Phi_H,
\end{equation}
with the implicit constant independent of $H$, and 
\begin{equation}\label{alphaNoPhiH}
\|\frac{1}{2\pi}\langle\alpha, H\rangle\|> N^{-1}, \ \ \forall \alpha\in \Phi\setminus \Phi_H.
\end{equation}
Note that $\Phi_H$ is $\mathbbmsl{Z}$\textit{-closed} in $\Phi$, that is, no element in $\Phi\setminus \Phi_H$ lies in the $\mathbb{Z}$-linear span of $\Phi_H$.

\begin{prop}
$\Phi_H$ is an integral root system. 
\end{prop}

\begin{proof}
We check the requirements for an integral root system listed in \eqref{RootSystem}. Parts (ii) and (iv) are automatic from the fact that $\Phi_H$ is a subset of $\Phi$. Part (i) comes from the fact that $\Phi_H$ is a $\mathbb{Z}$-linear space. Part (iii) follows from the fact that $s_\alpha\beta$ is a $\mathbb{Z}$-linear combination of $\alpha$ and $\beta$, for all $\alpha,\beta\in \Phi_H$, and the fact that $\Phi_H$ is a $\mathbb{Z}$-linear space.  
\end{proof}

Then we say that $\Phi_H$ is a \textit{root subsystem} of $\Phi$.

Let $W_H$ be the Weyl group of $\Phi_H$. $W_H$ is generated by reflections $s_\alpha$ for $\alpha\in\Phi_H$ and $W_H$ is a subgroup of the Weyl group $W$ of $\Phi$. Let $P$ be a positive system of roots of $\Phi$ and $P_H=P\cap \Phi_H$. Then $P_H$ is a positive system of roots of $\Phi_H$. We rewrite the Weyl character as  
\begin{align*}
\chi_{\lambda}
&=\frac{\sum_{s\in W}\det s \ e^{i\langle s(\lambda), H\rangle}}{e^{-i\langle\rho, H\rangle}\prod_{\alpha\in P}(e^{i\langle\alpha, H\rangle}-1)}\\
&=\frac{\frac{1}{|W_H|}\sum_{s_H\in W_H}\sum_{s\in W}\det (s_Hs) \ e^{i\langle (s_Hs)(\lambda), H\rangle}}{e^{-i\langle\rho, H\rangle}{\left(\prod_{\alpha\in P\setminus P_H}(e^{i\langle\alpha, H\rangle}-1)\right)} {\left(\prod_{\alpha\in P_H}(e^{i\langle\alpha, H\rangle}-1)\right) }}\\
&=\frac{1}{|W_H|e^{-i\langle\rho, H\rangle}\prod_{\alpha\in P\setminus P_H}(e^{i\langle\alpha, H\rangle}-1)}\sum_{s\in W}\det s \cdot \frac{\sum_{s_H\in W_H}\det s_H \ e^{i\langle s_H(s(\lambda)), H\rangle}}{\prod_{\alpha\in P_H}(e^{i\langle\alpha, H\rangle}-1)}\\
&=C(H)\sum_{s\in W}\det s \cdot \frac{\sum_{s_H\in W_H}\det s_H \ e^{i\langle s_H(s(\lambda)), H\rangle}}{\prod_{\alpha\in P_H}(e^{i\langle\alpha, H\rangle}-1)},\numberthis
\end{align*}
where 
\begin{align}
C(H):=\frac{1}{|W_H|e^{-i\langle\rho, H\rangle}\prod_{\alpha\in P\setminus P_H}(e^{i\langle\alpha, H\rangle}-1)}.
\end{align}
Then by \eqref{alphaNoPhiH}, 
\begin{equation}\label{C(H)}
|C(H)|\lesssim N^{|P\setminus P_H|}.
\end{equation}

Let $V_H$ be the $\mathbb{R}$-linear span of $\Phi_H$ in $V=i\mathfrak{b}^*$ and let $H^{\parallel}$ be the orthogonal projection of $H$ on $V_H$. Let $H^{\perp}=H-H^{\parallel}$. Then
$H^\perp$ is orthogonal to $V_H$ and we have  
\begin{align*}
\chi_{\lambda}&=C(H)\sum_{s\in W}\det s \cdot \frac{\sum_{s_H\in W_H}\det s_H \ e^{i\langle s_H(s(\lambda)), H^\perp+H^\parallel\rangle}}{\prod_{\alpha\in P_H}(e^{i\langle\alpha, H^\perp+H^\parallel\rangle}-1)}\\
&=C(H)\sum_{s\in W}\det s \cdot \frac{\sum_{s_H\in W_H}\det s_H \ e^{i\langle s(\lambda), s_H(H^\perp)\rangle}e^{i\langle s_H(s(\lambda)), H^\parallel\rangle}}{\prod_{\alpha\in P_H}(e^{i\langle\alpha, H^\parallel\rangle}-1)}.\numberthis
\end{align*}
Note that since $H^\perp$ is orthogonal to every root in $\Phi_H$, $H^\perp$ is fixed by $s_\alpha$ for any $\alpha\in\Phi_H$, which in turn implies that $H^\perp$ is fixed by any $s_H\in W_H$, that is, $s_H(H^\perp)=H^{\perp}$. Then  
\begin{align*}
\chi_\lambda&=C(H)\sum_{s\in W}\det s \cdot \frac{\sum_{s_H\in W_H}\det s_H \ e^{i\langle s(\lambda), H^\perp\rangle}e^{i\langle s_H(s(\lambda)), H^\parallel\rangle}}{\prod_{\alpha\in P_H}(e^{i\langle\alpha, H^\parallel\rangle}-1)}\\
&=C(H)\sum_{s\in W}\det s \cdot e^{i\langle s(\lambda), H^\perp\rangle}\cdot \frac{\sum_{s_H\in W_H}\det s_H \ e^{i\langle s_H(s(\lambda)), H^\parallel\rangle}}{\prod_{\alpha\in P_H}(e^{i\langle\alpha, H^\parallel\rangle}-1)}.\numberthis
\end{align*}
Note that by the definition of $H^{\parallel}$, we have 
\begin{equation}\label{alphaPhiHperp}
\|\frac{1}{2\pi}\langle\alpha, H^\parallel\rangle\|\lesssim N^{-1}, \ \ \forall\alpha\in \Phi_{H}. 
\end{equation}
This means that $\exp H^{\parallel}$ is a corner in the maximal torus of the compact semisimple Lie group associated to the integral root system $\Phi_H$. 

Using the above formula, we rewrite the Schr\"odinger kernel \eqref{Schrodinger kernel2} as 
\begin{align}\label{KNKNs}
K_N(t,x)=\frac{C(H)e^{it|\rho|^2}}{(\prod_{\alpha\in P}\langle\alpha,\rho\rangle)|W|}\sum_{s\in W}\det s \ \cdot K_{N,s}(t,x)
\end{align}
where 
\begin{align*}
K_{N,s}(t,x)=\sum_{\lambda\in \Lambda}e^{i\langle s(\lambda), H^\perp\rangle-it|\lambda|^2} \varphi(\frac{|\lambda|^2-|\rho|^2}{N^2})\left(\prod_{\alpha\in P}\langle \alpha,\lambda\rangle\right)\frac{\sum_{s_H\in W_H}\det s_H \ e^{i\langle s_H(s(\lambda)), H^\parallel\rangle}}{\prod_{\alpha\in P_H}(e^{i\langle\alpha, H^\parallel\rangle}-1)}.
\end{align*}
Noting that for any $s\in W$, $|s(\lambda)|=|\lambda|$, $\prod_{\alpha\in P}\langle\alpha, s(\lambda)\rangle=\det s\  \prod_{\alpha\in P}\langle\alpha, \lambda\rangle$ by Proposition \ref{Anti-invariant by invariant}, and $s(\Lambda)=\Lambda$, we have 
\begin{equation*}
K_{N,s}(t,x)=\det s \  K_{N,\mathbbm{1}}(t,x)
\end{equation*}
where $\mathbbm{1}$ is the identity element in $W$. Then \eqref{KNKNs} becomes 
\begin{align}\label{KNKN1}
K_N(t,x)=\frac{C(H)e^{it|\rho|^2}}{(\prod_{\alpha\in P}\langle\alpha,\rho\rangle)} K_{N,\mathbbm{1}}(t,x).
\end{align}

\begin{prop}\label{PropKN1}
Recall that 
\begin{equation}\label{KN1}
K_{N,\mathbbm{1}}(t,x)=\sum_{\lambda\in \Lambda}e^{i\langle \lambda, H^\perp\rangle-it|\lambda|^2} \varphi(\frac{|\lambda|^2-|\rho|^2}{N^2})\left(\prod_{\alpha\in P}\langle \alpha,\lambda\rangle\right)\frac{\sum_{s_H\in W_H}\det s_H \ e^{i\langle s_H(\lambda), H^\parallel\rangle}}{\prod_{\alpha\in P_H}(e^{i\langle\alpha, H^\parallel\rangle}-1)}.
\end{equation}
Then 
\begin{align}\label{KN1Estimate}
|K_{N,\mathbbm{1}}(t,x)|\lesssim \frac{N^{d-|P\setminus P_H|}}{\left(\sqrt{q}(1+N|\frac{t}{2\pi D}-\frac{a}{q}|^{1/2})\right)^{r}}
\end{align}
for $\frac{t}{2\pi D}\in\mathcal{M}_{a,q}$, uniformly in $x\in G$.  
\end{prop}

Noting \eqref{C(H)} and \eqref{KNKN1}, the above proposition implies Part (ii) of Theorem \ref{MainEstimate}.

\begin{exmp}

Figure \ref{figure} is an illustration of the decomposition of the maximal torus of $SU(3)$ according to the values of $\|\frac{1}{2\pi}\langle\alpha, H\rangle\|$, $\alpha\in\Phi$. Here $P=\{\alpha_1, \alpha_2,\alpha_3=\alpha_1+\alpha_2\}$. The three proper root subsystems of $\Phi$ are $\{\pm \alpha_i\}$, $i=1,2,3$. The association of $\Phi_H$ to $H$ is as follows. 
\begin{align*}
H\in regions\  of\  color\  \begin{tikzpicture}\draw[fill=red, opacity=0.2](0,0)--(0,0.4)--(0.4,0.4)--(0.4,0);\end{tikzpicture}\  &\Leftrightarrow\  \Phi_H=\Phi,\\
H\in regions\  of\  color\  \begin{tikzpicture}\draw[fill=yellow, opacity=0.2](0,0)--(0,0.4)--(0.4,0.4)--(0.4,0);\end{tikzpicture}\  &\Leftrightarrow\  \Phi_H=\{\pm \alpha_1\},\\
H\in regions\  of\  color\  \begin{tikzpicture}\draw[fill=pink, opacity=0.2](0,0)--(0,0.4)--(0.4,0.4)--(0.4,0);\end{tikzpicture}\  &\Leftrightarrow\  \Phi_H=\{\pm \alpha_2\},\\
H\in regions\  of\  color\  \begin{tikzpicture}\draw[fill=cyan, opacity=0.2](0,0)--(0,0.4)--(0.4,0.4)--(0.4,0);\end{tikzpicture}\  &\Leftrightarrow\  \Phi_H=\{\pm \alpha_3\},\\
H\in regions\  of\  color\  \begin{tikzpicture}\draw[fill=green, opacity=0.2](0,0)--(0,0.4)--(0.4,0.4)--(0.4,0);\end{tikzpicture}\  &\Leftrightarrow\  \Phi_H=\emptyset.\\
\end{align*}

\begin{figure}
\centering
\scalebox{0.75}[0.75]
{
\begin{tikzpicture}

\draw (0, 0) -- (12,0);

\draw (0, 0) -- (6, 10.392);
\draw (0,0) -- (0, 13.856);
\node [left] at (0, 13.856) {\scalebox{1.33}{$\large{2\pi \frac{2\alpha_1}{\langle\alpha_1,\alpha_1\rangle}}$}};
\node [right] at (12, 6.928) {\scalebox{1.33}{$\large{2\pi\frac{2\alpha_3}{\langle\alpha_3,\alpha_3\rangle}}$}};
\draw (0,0) -- (12, -6.928);
\node [right] at (12, -6.928) {\scalebox{1.33}{$\large{2\pi\frac{2\alpha_2}{\langle\alpha_2,\alpha_2\rangle}}$}};
\draw (0, 13.856) -- (12, -6.928); 
\draw (0, 13.856) -- (12, 6.928);
\draw (12, 6.928) -- (12, -6.928);
\draw (0, 6.928) -- (12, 6.928);
\draw (6, -3.464) -- (12, 6.928);

\draw[dashed] (0, 0.5) -- (12,0.5);
\draw[dashed] (0.866, -0.5) -- (12, -0.5); 
\draw[dashed] (0.577, 0) -- (6.433, 10.142);
\draw[dashed] (0, 12.856) -- (11.134, -6.428); 
\draw[dashed] (0, 1) -- (5.577, 10.642);
\draw[dashed] (0.866,13.356) -- (12, -5.928); 
\draw[dashed] (0,6.428) -- (12, 6.428);
\draw[dashed] (0,7.428) -- (11.134, 7.428);
\draw[dashed] (5.567,-3.214) -- (11.567, 7.178);
\draw[dashed] (6.433,-3.714) -- (12, 5.928);
\draw[dashed] (0, 12.856) -- (0.433,13.606);
\draw[dashed] (11.567, -6.678) -- (12, -5.928);
\draw[dashed] (0,1) -- (0.866, -0.5);
\draw[dashed] (12, 5.928) -- (11.134, 7.428);

\draw [fill=red, opacity=0.2]
(12, 5.928) -- (12, 6.928) -- (11.134, 7.428) -- (11.423,6.928) -- (11.134,6.428) -- (11.7113, 6.428) -- cycle;

\draw[dashed] (0, 13.356) -- (0.866, 13.356);
\draw[dashed] (11.134, -6.428) -- (12, -6.428);

\draw [fill=red, opacity=0.2]
       (0,0) -- (0.866, -0.5) -- (0.577,0) -- (0.8657,0.5) -- (0.2887, 0.5) -- (0,1) -- cycle;

\draw [fill=red, opacity=0.2]
(0, 13.856) -- (0, 12.856) -- (0.2887, 13.356) -- (0.866, 13.356) -- cycle;

\draw [fill=red, opacity=0.2]
(12, -6.928) -- (11.134, -6.428) -- (11.7113,  -6.428) -- (12, -5.928) -- cycle;

\draw [fill=red, opacity=0.2]
       (7.423,0) -- (7.1343,0.5) -- (7.7113, 0.5) -- (8, 1) -- (8.2887, 0.5) -- (8.866, 0.5) -- (8.577, 0) -- (8.866, -0.5) -- (8.2887, -0.5) -- (8, -1) -- (7.7113, -0.5) -- (7.1343,-0.5) -- 
       cycle;

\draw [fill=red, opacity=0.2]
(3.423, 6.928) -- (3.1343, 7.428) -- (3.7113, 7.428) -- (4, 7.928) -- (4.2887, 7.428) -- (4.866, 7.428) -- (4.577, 6.928) -- (4.866, 6.428) -- (4.2887, 6.428) -- (4, 5.928) -- (3.7113, 6.428) -- (3.1343,6.428) 
--cycle;

\draw[fill=yellow, opacity =0.2]
(0,6.428) -- (3.1343,6.428) -- (3.423, 6.928) -- (3.1343, 7.428) -- (0, 7.428) -- cycle; 
\draw [fill=yellow, opacity=0.2]
 (11.134, 7.428) -- (11.423,6.928) -- (11.134,6.428) --(4.866, 6.428) --  (4.577, 6.928) -- (4.866, 7.428) --cycle;

\draw[fill=pink, opacity=0.2]
(12, 5.928) -- (8.866, 0.5)-- (8.2887, 0.5)--(8, 1)  --(11.134,6.428) -- (11.7113, 6.428) -- cycle; 

\draw[fill=cyan, opacity=0.2]
(0, 12.856) -- (0.2887, 13.356)  --  (0.866, 13.356)  --  (4, 7.928) --  (3.7113, 7.428) --(3.1343, 7.428) --(3.423, 6.928) -- cycle;

\draw[fill=pink, opacity=0.2]
(6.433, 10.142)-- (5.577, 10.642) -- (4, 7.928) -- (4.2887, 7.428) -- (4.866, 7.428) -- cycle;

\draw[fill=yellow, opacity=0.2]
(8.866, 0.5) -- (8.577, 0) -- (8.866, -0.5) -- (12, -0.5) -- (12, 0.5)--cycle; 

\draw[fill=cyan, opacity=0.2]
(8.866, -0.5) -- (8.2887, -0.5) -- (8, -1) -- (11.134, -6.428) -- (11.7113,  -6.428) -- (12, -5.928) -- cycle; 

\draw[fill=pink, opacity=0.2]
 (8, -1) -- (7.7113, -0.5) -- (7.1343,-0.5) -- (5.567,-3.214) -- (6.433,-3.714) -- cycle;

\draw [fill=pink, opacity=0.2]
(0.2887, 0.5) -- (0.8657,0.5) -- (4, 5.928) -- (3.7113, 6.428) --  (3.1343,6.428) -- (0,1) -- cycle;
\draw [fill=yellow, opacity=0.2]
(0.8657,0.5) -- (0.577,0) -- (0.866,-0.5)  -- (7.1343,-0.5) -- 
(7.423,0) -- 
(7.423,0) -- (7.1343,0.5) -- cycle;
\draw [fill=cyan, opacity=0.2]
(7.1343,0.5) -- (7.7113, 0.5) -- (8, 1) -- (4.866, 6.428) -- (4.2887, 6.428) -- (4, 5.928) -- cycle;

\draw [fill=green, opacity=0.2]
(0.8657,0.5) -- (7.1343,0.5) -- (4, 5.928) -- cycle;

\draw [fill=green, opacity=0.2]
(0,1)--(0,6.428) -- (3.1343,6.428) -- cycle;

\draw [fill=green, opacity=0.2]
(0,7.428)  -- (0,12.856) -- (3.1343, 7.428) -- cycle;

\draw [fill=green, opacity=0.2]
(4, 7.928) -- (0.866, 13.356) -- (5.577, 10.642) -- cycle; 

\draw [fill=green, opacity=0.2]
(6.433, 10.142) -- (4.866, 7.428) -- (11.134, 7.428) -- cycle; 

\draw [fill=green, opacity=0.2]
(4.866, 6.428) -- (8,1) -- (11.134,6.428) --cycle; 

\draw [fill=green, opacity=0.2]
(12, 0.5) -- (12, 5.928) -- (8.866, 0.5) -- cycle; 

\draw [fill=green, opacity=0.2]
 (8.866, -0.5) --  (12, -0.5) -- (12, -5.928) -- cycle; 

\draw [fill=green, opacity=0.2]
(0.866, -0.5) -- (7.137, -0.5 )  --(5.567,-3.214) -- cycle; 

\draw [fill=green, opacity=0.2]
(6.433,-3.714) -- (8, -1) -- (11.134,-6.428)-- cycle;

\draw[fill]  (6,6.7) circle [radius=0.05];
\node [right] at (6,6.7) {\scalebox{1.33}{$\large{H}$}};

\draw[fill]  (0,6.7) circle [radius=0.05];
\node [left] at (0,6.7) {\scalebox{1.33}{$\large{H^{\parallel}}$}};

\draw[fill]  (6,0) circle [radius=0.05];
\node [below] at (6,0) {\scalebox{1.33}{$\large{H^{\perp}}$}};

\draw[dotted] (0, 6.7) -- (6,6.7);
\draw[dotted] (6,6.7) -- (6,0);

\draw[<->] (12.2, 0.5) -- (12.2, 0); 
\node[right] at (12.2, 0.25) {\scalebox{1.33}{$\ \sim N^{-1}$}};
\end{tikzpicture}
}
\caption{
Decomposition of the maximal torus of SU(3) \\
according to the values of $\|\frac{1}{2\pi}\langle\alpha, H\rangle\|$, $\alpha\in\Delta$
}\label{figure}

\end{figure}

\end{exmp}

\subsubsection{Decomposition of the Weight Lattice}
To prove Proposition \ref{PropKN1}, we will make a decomposition of the weight lattice $\Lambda$ according to the root subsystem $\Phi_H$. First, 
we have the following lemma about root subsystems. Let $\text{Proj}_U$ denote the orthogonal projection map from the ambient inner product space onto the subspace $U$. 

\begin{lem}\label{LemmaRootSubsystem}
Let $\Phi$ be an integral root system in the space $V$ with the associated weight lattice $\Lambda_{\Phi}$. Let $\Psi$ be a root subsystem of $\Phi$. Then let $\Gamma_{\Psi}$ and $\Lambda_\Phi$ be the root lattice and weight lattice associated to $\Psi$ respectively. Let $V_\Psi$ be the $\mathbb{R}$-linear span of $\Psi$ in $V$. Let $\Upsilon_\Psi$ be the image of the orthogonal projection of $\Lambda_{\Phi}$ onto $V_\Psi$. Then the following statements hold true. \\
(1) $\Upsilon_\Psi$ is a lattice and $\Gamma_\Psi\subset\Upsilon_\Psi\subset\Lambda_\Psi$. In particular, the rank of $\Upsilon_\Psi$ equals the rank of $\Gamma_\Psi$ as well as $\Lambda_\Psi$. 
\\
(2) Let the ranks of $\Upsilon_\Psi$ and $\Lambda_\Phi$  be $r$ and $R$ respectively. Let $\{w_1,\cdots, w_r\}$ be a basis of $\Upsilon_\Psi$. Pick any $\{u_1,\cdots, u_r\}\subset \Lambda_\Phi$ such that $\text{Proj}_{V_\Psi}(u_i)=w_i$, $i=1,\cdots, r$. Then we can extend $\{u_1,\cdots, u_r\}$ into a basis $\{u_1,\cdots, u_r, u_{r+1},\cdots, u_R\}$ of $\Lambda_\Phi$. Furthermore, we can pick $\{u_{r+1},\cdots, u_R\}$ such that $\text{Proj}_{V_\Psi}(u_i)=0$ for $i=r+1,\cdots, R$. 
\end{lem} 

\begin{proof} (1)
It's clear that $\Upsilon_\Psi$ is a lattice. Let $\Gamma_\Phi$ be the root lattice associated to $\Phi$. Then $\Gamma_\Psi\subset\Gamma_\Phi$. 
Thus 
\begin{align*}
\Gamma_\Psi=\text{Proj}_{V_\Psi}(\Gamma_\Psi)
\subset\text{Proj}_{V_\Psi}(\Gamma_\Phi)
\subset\text{Proj}_{V_\Psi}(\Lambda_\Phi)=\Upsilon_\Psi.
\end{align*} 
On the other hand, 
for any $\mu\in \Lambda_\Phi$, $\alpha\in\Gamma_\Psi$, we have $\langle\text{Proj}_{V_\Psi}(\mu),\alpha\rangle=\langle \mu,\alpha\rangle$. This in particular implies 
\begin{align*}
2\frac{\langle\text{Proj}_{V_\Psi}(\mu),\alpha\rangle}{\langle\alpha,\alpha\rangle}=2\frac{\langle\mu,\alpha\rangle}{\langle\alpha,\alpha\rangle}\in \mathbb{Z}, \ \ \text{for all }\mu\in \Lambda_\Phi,\  \alpha\in\Gamma_\Psi.
\end{align*}
This implies $\text{Proj}_{V_\Psi}(\mu)\in\Lambda_\Psi$ for all $\mu\in\Lambda_\Phi$, that is, 
$\Upsilon_\Psi=\text{Proj}_{V_\Psi}(\Lambda_\Phi)\subset\Lambda_\Psi$.\\
(2) Let $S_{\Phi}:=\mathbb{Z}u_1+\cdots+\mathbb{Z}u_r$, then $S_{\Phi}$ is a sublattice of $\Lambda_\Phi$ of rank $r$. By the theory of modules over a principal ideal domain, there exists a basis $\{u_1',\cdots, u_R'\}$ of $\Lambda_\Phi$ and positive integers $d_1|d_2|\cdots|d_r$ such that 
$\{d_1u_1',\cdots, d_ru_r'\}$ is a basis of $S_\Phi$. 
Then we must have $d_1=d_2=\cdots=d_r=1$, since 
\begin{align*}
\mathbb{Z}d_1\text{Proj}_{V_\Psi}(u_1')+\cdots+\mathbb{Z}d_r\text{Proj}_{V_\Psi}(u_r')&=\text{Proj}_{V_\Psi}(S_\Phi)\\
&=\text{Proj}_{V_\Psi}(\Lambda_\Phi)
\supset \mathbb{Z}\text{Proj}_{V_\Psi}(u_1')+\cdots+\mathbb{Z}\text{Proj}_{V_\Psi}(u_r')\numberthis
\end{align*}
and that $u_1', \cdots, u_r'$ are $\mathbb{R}$-linear independent. Thus we have 
\begin{align*}
S_\Phi=\mathbb{Z}u_1+\cdots+\mathbb{Z}u_r=\mathbb{Z}u_1'+\cdots+\mathbb{Z}u_r'
\end{align*}
and then $\{u_1,\cdots, u_r, u_{r+1}',\cdots, u_R'\}$ is also a basis of $\Lambda_\Phi$. Furthermore, 
by adding a $\mathbb{Z}$-linear combination of $u_1,\cdots, u_r$ to each of $u_{r+1}',\cdots, u_R'$, we can assume that $\text{Proj}_{V_\Psi}(u_i')=0$, for $i=r+1,\cdots, R$.

\end{proof}

We apply the above lemma to the root subsystem $\Phi_H$ of $\Phi$. Let $V=i\mathfrak{b}^*$, $V_H$ be the $\mathbb{R}$-linear span of $\Phi_H$ in $V$, $\Gamma_H$ be the root lattice for $\Phi_H$, and let 
\begin{equation}\label{UH}
\Upsilon_H:=\text{Proj}_{V_H}(\Lambda).
\end{equation} 
Then by the above lemma, we have 
\begin{equation}\label{UHGH}
\Upsilon_H\supset \Gamma_H.
\end{equation}
Let $r_H$ be the rank of $\Phi_H$ as well as of $\Gamma_H$ and $\Upsilon_H$, and let $\{w_1,\cdots, w_{r_H}\}\subset\Upsilon_H$ such that 
\begin{equation*}
\Upsilon_H=\mathbb{Z}w_1+\cdots+\mathbb{Z}w_{r_H}. 
\end{equation*}
Pick $\{u_1,\cdots, u_{r_H}\}\subset \Lambda$ such that 
\begin{equation*}
\text{Proj}_{V_H}(u_i)=w_i, \ \ i=1,\cdots, r_H. 
\end{equation*}
Then by the above lemma, we can extend $\{u_1,\cdots, u_{r_H}\}$ into a basis $\{u_1,\cdots, u_r\}$ of $\Lambda$, such that 
\begin{equation}\label{ProjUrH+1}
\text{Proj}_{V_H}(u_i)=0,\ \ i=r_H+1,\cdots, r,
\end{equation}
with 
\begin{equation*}
\Lambda=\mathbb{Z}u_1+\cdots+\mathbb{Z}u_r.
\end{equation*}
Set
\begin{equation*}
\Upsilon'_H=\mathbb{Z}u_1+\cdots+\mathbb{Z}u_{r_H}\subset\Lambda,
\end{equation*}
then 
\begin{equation*}
\text{Proj}_{V_H}: \Upsilon'_H\xrightarrow{\sim}\Upsilon_H. 
\end{equation*}
Recalling \eqref{UHGH}, let $\Gamma'_H$ be the sublattice of $\Upsilon'_H$ corresponding to $\Gamma_H\subset\Upsilon_H$ under this isomorphism. More precisely, let $\{\alpha_1,\cdots, \alpha_{r_H}\}$ be a simple system of roots for $\Gamma_H$, then
\begin{align}\label{ProjG'G}
\text{Proj}_{V_H}: \Gamma'_H=\mathbb{Z}\alpha'_{1}+\cdots+\mathbb{Z}\alpha'_{r_H}\xrightarrow{\sim}\Gamma_H=\mathbb{Z}\alpha_{1}+\cdots+\mathbb{Z}\alpha_{r_H},\ \ \alpha_i'\mapsto\alpha_i,\ \ i=1,\cdots,r_H,
\end{align}
and we have 
\begin{equation}\label{Upsilon/Gamma}
\Upsilon'_H/\Gamma'_H\cong \Upsilon_H/\Gamma_H, \ \ |\Upsilon'_H/\Gamma'_H|=|\Upsilon_H/\Gamma_H|<\infty. 
\end{equation}
Decomposing the weight lattice as
\begin{align*}
\Lambda=\bigsqcup_{\mu\in\Upsilon'_H/\Gamma'_H}\left(\mu+\Gamma'_H+\mathbb{Z}u_{r_{H}+1}+\cdots+\mathbb{Z}u_{r}\right),
\end{align*}
we have 
\begin{align*}
K_{N,\mathbbm{1}}(t,x)&=\sum_{\substack{\mu\in\Upsilon'_H/\Gamma'_H,\\ \lambda'_1=n_1\alpha'_1+\cdots+n_{r_H}\alpha'_{r_H},\\ \lambda_2=n_{r_H+1}u_{r_{H}+1}+\cdots+n_ru_{r}}}
e^{i\langle \mu+\lambda'_1+\lambda_2, H^\perp\rangle-it| \mu+\lambda'_1+\lambda_2|^2} \varphi(\frac{|\mu+\lambda'_1+\lambda_2|^2-|\rho|^2}{N^2})\\
& \ \ \ \ \cdot \left(\prod_{\alpha\in P}\langle \alpha,\mu+\lambda'_1+\lambda_2\rangle\right)\frac{\sum_{s_H\in W_H}\det s_H \ e^{i\langle s_H(\mu+\lambda'_1+\lambda_2), H^\parallel\rangle}}{\prod_{\alpha\in P_H}(e^{i\langle\alpha, H^\parallel\rangle}-1)}.\numberthis
\end{align*}
Note that \eqref{ProjUrH+1} implies for $\lambda_2=n_{r_H+1}u_{r_{H}+1}+\cdots+n_ru_r$ that 
\begin{align*}
\langle s_H(\lambda_2),H^{\parallel}\rangle=
\langle \lambda_2,s_H(H^{\parallel})\rangle=0,
\end{align*}
and \eqref{ProjG'G} implies for $\lambda'_1=n_1\alpha'_1+\cdots+n_{r_H}\alpha'_{r_H}$ that 
\begin{align*}
\langle s_H(\lambda_1'),H^{\parallel}\rangle=
\langle \lambda'_1,s_H(H^{\parallel})\rangle=\langle \lambda_1,s_H(H^{\parallel})\rangle=\langle s_H(\lambda_1),H^{\parallel}\rangle
\end{align*}
where $\lambda_1=n_1\alpha_1+\cdots+n_{r_H}\alpha_{r_H}\in V_H$. Similarly, also note that 
\begin{align*}
\langle s_H(\mu), H^\parallel\rangle
=\langle s_H(\mu^{\parallel}), H^\parallel\rangle, \ \ \text{where } \mu^{\parallel}:=\text{Proj}_{V_H}(\mu). 
\end{align*}
Thus we write
\begin{align*}
K_{N,\mathbbm{1}}(t,x)&=\sum_{\mu\in\Upsilon'_H/\Gamma'_H}\sum_{\substack{\lambda'_1=n_1\alpha'_1+\cdots+n_{r_H}\alpha'_{r_H},\\ \lambda_1=n_1\alpha_1+\cdots+n_{r_H}\alpha_{r_H},\\ \lambda_2=n_{r_H+1}u_{r_{H}+1}+\cdots+n_ru_{r}}}
e^{i\langle \mu+\lambda'_1+\lambda_2, H^\perp\rangle-it| \mu+\lambda'_1+\lambda_2|^2} \varphi(\frac{|\mu+\lambda'_1+\lambda_2|^2-|\rho|^2}{N^2})\\
& \ \ \ \ \cdot \left(\prod_{\alpha\in P}\langle \alpha,\mu+\lambda'_1+\lambda_2\rangle\right) \frac{\sum_{s_H\in W_H}\det s_H \ e^{i\langle s_H(\mu^{\parallel}+\lambda_1), H^\parallel\rangle}}{\prod_{\alpha\in P_H}(e^{i\langle\alpha, H^\parallel\rangle}-1)}.\numberthis
\end{align*}

\begin{rem}\label{Weightlattice}
We have that in the above formula
\begin{align*}
\chi^{\mu^\parallel}(\lambda_1, H^\parallel)
:=\frac{\sum_{s_H\in W_H}\det s_H \ e^{i\langle s_H(\mu^{\parallel}+\lambda_1), H^\parallel\rangle}}{\prod_{\alpha\in P_H}(e^{i\langle\alpha, H^\parallel\rangle}-1)}
\end{align*}
is a character of the form \eqref{EstBGG2character}. Also note that $\mu^\parallel\in\text{Proj}_{V_H}(\Lambda)$ lies in the weight lattice $\Lambda_H$ of $\Phi_H$ by Lemma \ref{LemmaRootSubsystem}. 
\end{rem}

Noting \eqref{Upsilon/Gamma}, Proposition \ref{PropKN1} reduces to the following. 

\begin{prop}\label{PropKN1mu}
For $\mu\in\Upsilon'_H/\Gamma'_H$, let
\begin{align*}
K_{N,\mathbbm{1}}^{\mu}(t,x):&=\sum_{\substack{\lambda'_1=n_1\alpha'_1+\cdots+n_{r_H}\alpha'_{r_H},\\ \lambda_1=n_1\alpha_1+\cdots+n_{r_H}\alpha_{r_H},\\ \lambda_2=n_{r_H+1}u_{r_{H}+1}+\cdots+n_ru_{r},\\
n_1,\cdots,n_r\in\mathbb{Z}}}
e^{i\langle \mu+\lambda'_1+\lambda_2, H^\perp\rangle-it| \mu+\lambda'_1+\lambda_2|^2} \varphi(\frac{|\mu+\lambda'_1+\lambda_2|^2-|\rho|^2}{N^2})\\
& \ \ \ \ \cdot \left(\prod_{\alpha\in P}\langle \alpha,\mu+\lambda'_1+\lambda_2\rangle\right) \frac{\sum_{s_H\in W_H}\det s_H \ e^{i\langle s_H(\mu^\parallel+\lambda_1), H^\parallel\rangle}}{\prod_{\alpha\in P_H}(e^{i\langle\alpha, H^\parallel\rangle}-1)}.\numberthis
\end{align*}
Then 
\begin{align}\label{KN1mu}
|K_{N,\mathbbm{1}}^{\mu}(t,x)|\lesssim \frac{N^{d-|P\setminus P_H|}}{\left(\sqrt{q}(1+N|\frac{t}{2\pi D}-\frac{a}{q}|^{1/2})\right)^{r}}
\end{align}
for $\frac{t}{2\pi D}\in\mathcal{M}_{a,q}$, uniformly in $x\in G$.  
\end{prop}
\begin{proof}
We apply Lemma \ref{WeylSum} to the lattice 
$\mathbb{Z}\alpha_1'+\cdots+\mathbb{Z}\alpha_{r_H}'+\mathbb{Z}u_{r_H+1}+\cdots+\mathbb{Z}u_r$. Write 
\begin{align*}
\chi^{\mu^\parallel}(\lambda_1, H^\parallel)
=\frac{\sum_{s_H\in W_H}\det s_H \ e^{i\langle s_H(\mu^{\parallel}+\lambda_1), H^\parallel\rangle}}{\prod_{\alpha\in P_H}(e^{i\langle\alpha, H^\parallel\rangle}-1)}
\end{align*}
Then it suffices to show that 
\begin{align*}
\left|D_{i_1}\cdots D_{i_k}\left(\prod_{\alpha\in P}\langle \alpha,\mu+\lambda'_1+\lambda_2\rangle\chi^{\mu^\parallel}(\lambda_1, H^\parallel)\right) \right|\lesssim N^{d-|P\setminus P_H|-r-k}, 
\end{align*}
for $1\leq i_1,\cdots, i_k\leq r$, 
\begin{align*}
\lambda'_1&=n_1\alpha'_1+\cdots+n_{r_H}\alpha'_{r_H}, \\
\lambda_1&=n_1\alpha_1+\cdots+n_{r_H}\alpha_{r_H}, \\
\lambda_2&=n_{r_H+1}u_{r_{H}+1}+\cdots+n_ru_{r},
\end{align*}
uniformly in $|n_i|\lesssim N$, $i=1,\cdots, r$. Since $\prod_{\alpha\in P}\langle \alpha,\mu+\lambda'_1+\lambda_2\rangle$ is a polynomial and thus a pseudo-polynomial of degree $|P|$, it suffices to show that 
\begin{align}\label{Dichi||}
\left|D_{i_1}\cdots D_{i_k}(\chi^{\mu^\parallel}(\lambda_1, H^\parallel)) \right|\lesssim N^{d-|P\setminus P_H|-r-|P|-k}=N^{|P_H|-k}.
\end{align}
Since $\chi(\lambda_1)$ does not involve the variables $n_{r_H+1}, \cdots, n_r$, it suffices to prove \eqref{Dichi||} for $1\leq i_1,\cdots, i_k\leq r_H$ uniformly in $|\lambda_1|\lesssim N$. But this follows by applying 
Lemma \ref{EstBGG2} to the root system $\Phi_H$, noting Remark \ref{Weightlattice}. 
\end{proof}

\subsection{$L^p$ Estimates}
\label{Lp Estimates}
We prove in this section $L^p(G)$ estimates of the Schr\"odinger kernel for $p<\infty$. Though we do not apply them to the proof of the main theorem, they encapsulate the essential ingredients in the proof of the $L^\infty(G)$-estimate and are of independent interest. 

\begin{prop}\label{LpLp}
Let $K_N(t,x)$ be the Schr\"odinger kernel as in Proposition \ref{dispersive for Schrodinger}. Then for any $p>3$, we have 
\begin{align}\label{LpEstimates}
\|K_N(t,\cdot)\|_{L^p(G)}\lesssim\frac{N^{d-\frac{d}{p}}}{(\sqrt{q}(1+N\|\frac{t}{2\pi D}-\frac{a}{q}\|^{1/2}))^{r}}
\end{align} 
for $\frac{t}{2\pi D}\in\mathcal{M}_{a,q}$. 
\end{prop}

\begin{proof}
As a linear combination of characters, the  Schr\"odinger kernel $K_N(t,\cdot)$ is a central function. Then we can apply to it the Weyl integration formula \eqref{Weyl integration formula}
\begin{equation}\label{Weylintegration}
\|K_N(t,\cdot)\|^p_{L^p(G)}=\frac{1}{|W|}\int_B |K_N(t,b)|^p|D_P(b)|^2 \ db
\end{equation}
where $B$ is the maximal torus with normalized Haar measure $db$. Recall that we can parametrize $B=\exp \mathfrak{b}$ by 
$H\in i\mathfrak{b}^*\cong\mathfrak{b}$, and write 
\begin{align}\label{ParametrizeB}
B\cong i\mathfrak{b}^*/(2\pi\mathbb{Z}\alpha_1^\vee+\cdots+2\pi\mathbb{Z}\alpha_r^\vee)=[0,2\pi)\alpha_1^\vee+\cdots+[0,2\pi)\alpha_r^\vee,
\end{align}
where $\{\alpha^{\vee}_i=\frac{2\alpha_i}{\langle\alpha_i,\alpha_i\rangle}\mid i=1,\cdots,r\}$
is the set of simple coroots associated to a system of simple roots $\{\alpha_i\mid i=1,\cdots,r\}$.

We have shown in Section \ref{RootSubsystems} that each $H\in i\mathfrak{b}^*$ is associated to a root subsystem $\Phi_H$ such that \eqref{alphaPhiH} and \eqref{alphaNoPhiH} hold. Note that there are finitely many root subsystems of a given root system, thus $B$ is covered by finitely many subsets $R$ of the form 
\begin{align}\label{RegionR}
R=\{H\in B\mid \|\frac{1}{2\pi}\langle\alpha, H\rangle\|\lesssim N^{-1}, \forall \alpha\in\Psi; 
\|\frac{1}{2\pi}\langle\alpha, H\rangle\|> N^{-1}, \forall \alpha\in\Phi\setminus\Psi\}
\end{align}
where $\Psi$ is a root subsystem of $\Phi$. Thus to prove \eqref{LpEstimates}, using \eqref{Weylintegration}, it suffices to show 
\begin{align}\label{KNLpR}
\int_{R}|K_N(t,\exp H)|^p|D_p(\exp H)|^2\ dH\lesssim \left(\frac{N^{d}}{(\sqrt{q}(1+N\|\frac{t}{2\pi D}-\frac{a}{q}\|^{1/2}))^{r}}\right)^pN^{-d}.
\end{align}

By \eqref{C(H)}, \eqref{KNKN1} and \eqref{KN1Estimate}, 
we have 
\begin{align*}
K_N(t,\exp H)\lesssim\frac{1}{\prod_{\alpha\in P\setminus Q}(e^{i\langle\alpha,H\rangle}-1)}\cdot\frac{N^{d-|P\setminus Q|}}{\left(\sqrt{q}(1+N|\frac{t}{2\pi D}-\frac{a}{q}|^{1/2})\right)^{r}}
\end{align*}
where $P,Q$ are respectively the sets of positive roots of $\Phi$ and $\Psi$ with $P\supset Q$. Recalling $D_P(\exp H)=\prod_{\alpha\in P}(e^{i\langle\alpha, H\rangle}-1)$, \eqref{KNLpR} is then reduced to 
\begin{align*}
\int_R \left|\frac{1}{\prod_{\alpha\in P\setminus Q}(e^{i\langle\alpha, H\rangle}-1)}\right|^{p-2}\left|{\prod_{\alpha\in Q}(e^{i\langle\alpha, H\rangle}-1)}\right|^2\ dH\lesssim N^{p|P\setminus Q|-d}.
\end{align*}
Using
\begin{align*}
|e^{i\langle\alpha, H\rangle}-1|\lesssim\|\frac{1}{2\pi}\langle\alpha, H\rangle\|\lesssim |e^{i\langle\alpha, H\rangle}-1|, 
\end{align*}
it suffices to show
\begin{align}\label{R2-p2}
\int_R \left|\frac{1}{\prod_{\alpha\in P\setminus Q}\|\frac{1}{2\pi}\langle\alpha, H\rangle\|}\right|^{p-2}\left|{\prod_{\alpha\in Q}\|\frac{1}{2\pi}\langle\alpha, H\rangle\|}\right|^2\ dH\lesssim N^{p|P\setminus Q|-d}.
\end{align}
For each $H\in B$, we write 
\begin{equation*}
H=H'+H_0
\end{equation*}
such that 
\begin{align*}
\|\frac{1}{2\pi}\langle\alpha,H\rangle\|=|\frac{1}{2\pi}\langle\alpha,H'\rangle|, \ \ \langle\alpha,H_0\rangle\in 2\pi\mathbb{Z}, \ \ \forall \alpha\in P.
\end{align*}
We write 
\begin{align}\label{RR'H0}
R\subset\bigcup_{\substack{H_0\in B, \\ \langle\alpha, H_0\rangle\in2\pi\mathbb{Z},\forall \alpha\in P}}
R'+H_0
\end{align}
where 
\begin{align}\label{R'}
R'=\{H\in B\mid |\frac{1}{2\pi}\langle\alpha, H\rangle|\lesssim N^{-1}, \forall \alpha\in Q; 
|\frac{1}{2\pi}\langle\alpha, H\rangle|> N^{-1}, \forall \alpha\in P\setminus Q\}.
\end{align}
Note that $\langle\alpha,\alpha_i^\vee\rangle\in\mathbb{Z}$ for all $\alpha\in P$ and $i=1,\cdots, r$ due to the integrality of the root system, using \eqref{ParametrizeB}, we have that there are only finitely many $H_0\in B$ such that $\langle\alpha, H_0\rangle\in 2\pi \mathbb{Z}$ for all $\alpha\in P$. Thus using \eqref{RR'H0}, \eqref{R2-p2} is further reduced to 
\begin{align}\label{R'2-p2}
\int_{R'}\left|\frac{1}{\prod_{\alpha\in P\setminus Q}|\frac{1}{2\pi}\langle\alpha, H\rangle|}\right|^{p-2}\left|{\prod_{\alpha\in Q}|\frac{1}{2\pi}\langle\alpha, H\rangle|}\right|^2\ dH\lesssim N^{p|P\setminus Q|-d}.
\end{align}
Now we reparametrize $B\cong [0,2\pi)\alpha_1^\vee+\cdots+[0,2\pi)\alpha_r^\vee$ by 
\begin{align*}
H=\sum_{i=1}^rt_iw_i,\ \ (t_1,\cdots, t_r)\in D
\end{align*}
where $\{w_i, \ i=1,\cdots, r\}$ are the fundamental weights such that $\langle\alpha_i,w_j\rangle=\delta_{ij}\frac{|\alpha_i|^2}{2}$, $i,j=1,\cdots, r$, and $D$ is a bounded domain in $\mathbb{R}^r$. Then the normalized Haar measure $dH$ equals
\begin{align*}
dH=Cdt_1\cdots dt_r
\end{align*}
for some constant $C$. Let $s\leq r$ such that 
\begin{align*}
\{\alpha_1,\cdots, \alpha_s\}&\subset P\setminus Q,\\
\{\alpha_{s+1},\cdots, \alpha_r\}&\subset  Q.
\end{align*}
Using \eqref{R'}, we estimate
\begin{align*}
&\int_{R'}\left|\frac{1}{\prod_{\alpha\in P\setminus Q}|\frac{1}{2\pi}\langle\alpha, H\rangle|}\right|^{p-2}\left|{\prod_{\alpha\in Q}|\frac{1}{2\pi}\langle\alpha, H\rangle|}\right|^2\ dH\\ 
\numberthis \label{p-2>1}
&\lesssim \int_{R'}\frac{1}{|t_1\cdots t_s|^{p-2}}N^{(p-2)(|P\setminus Q|-s)}N^{-2|Q|}\ dt_1\cdots dt_r\\ 
&\lesssim N^{(p-2)(|P\setminus Q|-s)}N^{-2|Q|}\int_{\substack{|t_1|,\cdots, |t_s|\gtrsim N^{-1},\\ |t_{s+1}|,\cdots, |t_r|\lesssim N^{-1}}} \frac{1}{|t_1\cdots t_s|^{p-2}}\ dt_1\cdots dt_r. 
\end{align*}
If $p>3$, the above is bounded by
\begin{align*}
\lesssim N^{(p-2)(|P\setminus Q|-s)}N^{-2|Q|}
N^{s(p-3)-(r-s)}=N^{p|P\setminus Q|-d}, 
\end{align*}
noting that $2|P\setminus Q|+2|Q|+r=2|P|+r=d$.
\end{proof}

\begin{rem}
The requirement $p>3$ is by no means optimal. The estimate in \eqref{p-2>1} may be improved to lower the exponent $p$. We conjecture that \eqref{LpEstimates} holds for all $p>p_r$ such that $\lim_{r\to\infty}p_r=2$, where $r$ is the rank of $G$. 
\end{rem}

\section*{Acknowledgments}

I would like to thank my Ph.D. thesis advisors Prof. Monica Visan and Prof. Rowan Killip for their constant guidance and support. I especially thank them for giving me all the freedom to choose problems that suit my own interest. I would like to thank Prof. Rapha\"{e}l Rouquier and Prof. Terence Tao for sharing their expertise on the BGG-Demazure operators and the Weyl type sums respectively. I would also like to thank Jiayin Guo for pointing out several false conjectures of mine on root systems. My thanks also goes to Prof. Guozhen Lu who encouraged me to publish the paper. I am thankful to the editors and referees for their carefully reading the paper and their helpful suggestions on revising it. Last but not least, I am thankful to Mengmeng for her company and encouragement.

\end{document}